\documentclass[12pt]{amsart}

\usepackage{amsmath, amssymb, amsfonts, amsthm, amscd}

\usepackage[mathcal]{euscript}  
\usepackage{manfnt}
\usepackage{mathtools}
\usepackage{fullpage}
\usepackage{url}
\usepackage[all]{xy}
   \SelectTips{cm}{10}
\usepackage{txfonts}
\usepackage{enumitem}
\usepackage{bbm}
\usepackage[colorlinks=true,linkcolor=blue,urlcolor=blue,citecolor=blue, pagebackref]{hyperref}
\usepackage[usenames, dvipsnames, svgnames]{xcolor}

\usepackage[nonewpage]{imakeidx}  

\newtheorem{thm}{Theorem}[section]
\newtheorem{lemma}[thm]{Lemma}
\newtheorem{prop}[thm]{Proposition}
\newtheorem{cor}[thm]{Corollary}

\newtheorem{prop-conj}[thm]{Proposition-Conjecture}

\newtheorem{thmalph}{Theorem}

\theoremstyle{definition}
\newtheorem{defn}[thm]{Definition}

\theoremstyle{definition}
\newtheorem{rmk}[thm]{Remark}

\theoremstyle{definition}
\newtheorem{constr}[thm]{Construction}

\theoremstyle{definition}
\newtheorem{assumption}[thm]{Assumption}

\theoremstyle{definition}
\newtheorem{claim}[thm]{Claim}

\theoremstyle{definition}
\newtheorem{notation}[thm]{Notation}

\theoremstyle{definition}
\newtheorem{question}[thm]{Question}

\theoremstyle{definition}
\newtheorem{eg}[thm]{Example}

\DeclareFontFamily{U}{wncy}{}
    \DeclareFontShape{U}{wncy}{m}{n}{<->wncyr10}{}
    \DeclareSymbolFont{mcy}{U}{wncy}{m}{n}
    \DeclareMathSymbol{\Sha}{\mathord}{mcy}{"58}

\newcommand{\Q}{\mathbb{Q}}

\newcommand{\Z}{\mathbb{Z}}
\newcommand{\CC}{\mathbb{C}}
\newcommand{\RR}{\mathbb{R}}

\newcommand{\Qpb}{\overline{\mathbb{Q}}_p}
\newcommand{\Zpb}{\overline{\mathbb{Z}}_p}
\newcommand{\Fpb}{\overline{\mathbb{F}}_p}

\newcommand{\Fp}{\mathbb{F}_p}

\DeclareMathOperator{\Hom}{Hom}

\DeclareMathOperator{\End}{End}

\DeclareMathOperator{\Out}{Out}

\DeclareMathOperator{\Ad}{Ad}

\DeclareMathOperator{\rec}{rec}
\DeclareMathOperator{\tr}{tr}
\DeclareMathOperator{\im}{im}
\DeclareMathOperator{\coker}{coker}
\DeclareMathOperator{\Res}{Res}

\DeclareMathOperator{\rk}{rk}

\DeclareMathOperator{\Spec}{Spec}

\DeclareMathOperator{\Spf}{Spf}
\DeclareMathOperator{\Sp}{Sp}

\DeclareMathOperator{\Frac}{Frac}
\DeclareMathOperator{\Lie}{Lie}

\newcommand{\gal}[1]{\Gamma_{#1}} 
\newcommand{\Gal}{\mathrm{Gal}} 

\newcommand{\into}{\hookrightarrow}
\newcommand{\onto}{\twoheadrightarrow}

\newcommand{\mc}{\mathcal}
\newcommand{\mf}{\mathfrak}
\newcommand{\mr}{\mathrm}
\newcommand{\mbf}{\mathbf}
\newcommand{\mbb}{\mathbb}

\newcommand{\br}{\bar{\rho}} 
\newcommand{\fg}{\mathfrak{g}}

\newcommand{\fgder}{\mathfrak{g}^{\mathrm{der}}}

\DeclareMathOperator{\Lift}{\mathrm{Lift}}
\DeclareMathOperator{\Def}{\mathrm{Def}}
\newcommand{\Sets}{\mathbf{Sets}}

\newcommand{\unr}{\mathrm{unr}}

\newcommand{\Frob}{\mathrm{Frob}}

\newcommand{\tF}{\widetilde{F}}

\newcommand{\wt}{\widetilde}
\newcommand{\wh}{\widehat}
\newcommand{\inv}{\mathrm{inv}}

\newcommand{\vpi}{\varpi}

\newcommand{\un}[1]{\underline{#1}}
\newcommand{\ov}{\overline}

\newcommand{\n}{n}
\newcommand{\z}{Z}

\title{Relative deformation theory, relative Selmer groups, and lifting irreducible Galois representations}

\thanks{We would like to thank Gebhard Boeckle, Prakash Belkale,
  Arvind Nair, Ravi Ramakrishna, Jack Thorne, and Akshay Venkatesh for
  useful conversations and correspondence. We are especially grateful
  to Ravi Ramakrishna for providing detailed comments on an earlier
  draft. We also thank the referees for their helpful reports, which have improved the exposition. N.F. was supported by the DAE, Government of India, PIC
  12-R\&D-TFR-5.01-0500. C.K.  would like to thank TIFR, Mumbai for
  its hospitality, in periods when some of the work was carried
  out. S.P. was supported by NSF grants DMS-1700759, DMS-1752313, and DMS-2120325.}

\makeindex[name=terms,title=Index of terms and notation,columns=2]

\begin{document}
\author[N.~Fakhruddin]{Najmuddin Fakhruddin}
\address{School of Mathematics, Tata Institute of Fundamental Research, Homi Bhabha Road, Mumbai 400005, INDIA}
\email{naf@math.tifr.res.in}
\author[C.~Khare]{Chandrashekhar Khare}
\address{UCLA Department of Mathematics, Box 951555, Los Angeles, CA 90095, USA}
\email{shekhar@math.ucla.edu}
\author[S.~Patrikis]{Stefan Patrikis}
\address{Department of Mathematics, Ohio State University, 100 Math Tower, 231 West 18th Ave., Columbus, OH 43210, USA}
\email{patrikis.1@osu.edu}

\begin{abstract}
We study irreducible odd mod $p$ Galois representations $\br \colon \Gal(\overline{F}/F) \to G(\overline{\mathbb{F}}_p)$, for $F$ a totally real number field and $G$ a general reductive group. For $p \gg_{G, F} 0$, we show that any $\br$ that lifts locally, and at places above $p$ to de Rham and Hodge--Tate regular representations, has a geometric $p$-adic lift. We also prove non-geometric lifting results without any oddness assumption. 
\end{abstract}


\maketitle
\centerline{In memory of Jean-Pierre Wintenberger 1954--2019}

\tableofcontents

\section{Introduction}
Let $\overline{\Z}_p$ be the integral closure of $\Z_p$ in $\Qpb$, and let $G$ be a smooth group scheme over $\overline{\Z}_p$ such that $G^0$ is a split connected reductive group. 

\subsection{The lifting problem for odd representations}

The starting point of this paper is the following basic question:
\begin{question}\label{question}
Let $F$ be a number field with algebraic closure $\overline{F}$ and
absolute Galois group $\gal{F}= \Gal(\overline{F}/F)$, and let $\br
\colon \Gal(\overline{F}/F) \to G(\overline{\mathbb{F}}_p)$
\index[terms]{R@$\br$} be a continuous homomorphism. Does there exist a lift $\rho$
\[
\xymatrix{
& G(\overline{\Z}_p) \ar[d] \\
\gal{F} \ar@{-->}[ur]^{\rho} \ar[r]_-{\br} & G(\overline{\mathbb{F}}_p) 
}
\]
that is geometric \index[terms]{geometric} in the sense of Fontaine--Mazur?
\end{question}
This question has attracted a great deal of attention, at least since Serre proposed his modularity conjecture (\cite{serre:conjectures}). We begin by recalling a few instances of this general problem, beginning with Serre's conjecture. Serre proposed that every irreducible representation 
\[
\br \colon \gal{\Q} \to \mr{GL}_2(\Fpb)
\]
that was moreover \textit{odd} in the sense that $\det \br(c)=-1$ for any complex conjugation $c \in \gal{\Q}$ should be isomorphic to the mod $p$ reduction of a $p$-adic Galois representation attached to a classical modular eigenform. In particular, such a $\br$ should admit a geometric $p$-adic lift. 
Around the time Serre first made his conjecture, as recounted in a letter of Serre to Tate on 12th July, 1974 (\cite{serre-tate:correspondance}), Deligne raised the objection that the conjecture
implied the existence of geometric lifts of $\br$ which  were
moreover minimally ramified (for example unramified outside $p$ if
$\br $ is unramified outside $p$). The papers \cite{khare-wintenberger:serre0}, \cite{khare:serrelevel1}, \cite{khare-wintenberger:serre1}, \cite{khare-wintenberger:serre2} prove Serre's modularity conjecture, and as a key step lift  $\br$ to a geometric representation with prescribed local properties.  The proof of this key step  uses  the modularity lifting results of Wiles and Taylor (\cite{wiles:fermat}, \cite{taylor-wiles:fermat}). In contrast, prior to the resolution of Serre's conjecture, Ramakrishna (\cite{ramakrishna:lifting}, \cite{ramakrishna02}) developed a beautiful, purely Galois-theoretic, method that in most cases settled Question \ref{question} in the setting of Serre's conjecture ($F= \Q$, $G= \mr{GL}_2$, $\br$ odd and irreducible). Ramakrishna's lifts cannot be ensured to be  minimally ramified.

We might then turn to asking Question \ref{question} for $\br \colon \gal{\Q} \to \mr{GL}_2(\Fpb)$ that are \textit{even}, in the sense that $\det(\br(c))=1$. For instance, suppose that the image of $\br$ is $\mr{SL}_2(\mathbb{F}_p)$. Any geometric lift would (for $p \neq 2$) itself be even, and so conjecturally would be the $p$-adic representation $\rho$ attached to an algebraic Maass form. Such a $\rho$ should, up to twist, have finite image (because up to twist the associated motive should have Hodge realization of type $(0,0)$); but for $p > 5$, Dickson's classification of finite subgroups of $\mr{PGL}_2(\CC)$ rules out the possibility of such a lift. Thus one expects that $\br$ has no geometric lift. We have no general means of translating this conjectural heuristic into a proof, but Calegari (\cite[Theorem 5.1]{calegari:even2}) has given an ingenious argument that proves unconditionally that certain such even $\br$ have no geometric lift. 

In other settings, Question \ref{question} is even more
mysterious. For instance, if $G = \mr{GL}_2$ and $F/\Q$ is
quadratic imaginary, we do not even have a reliable heuristic for
predicting whether $\br \colon \gal{F} \to \mr{GL}_2(\Fpb)$ should
have a geometric lift! It is a remarkable and widely-tested phenomenon
that torsion cohomology (Hecke eigen-) classes for the locally
symmetric spaces associated to congruence subgroups of $\mr{GL}_2/F$
need not lift to characteristic zero; one can ask whether after
raising the level (passing to a finite covering space of the
arithmetic 3-manifold) their Hecke eigensystems lift, and that the
corresponding Galois-theoretic statement holds as well. However, we
have little evidence to support this.

This paper addresses Question \ref{question} for general $G$, but for
$\br$ that are \textit{odd}  \index[terms]{odd}
in a sense generalizing Serre's formulation for $\mr{GL}_2$. The following definition is essentially due to Gross (\cite{gross:odd}), who suggested parallels between this class of Galois representations and the ``odd" representations of Serre's original conjecture:
\begin{defn}\label{odd}
We say $\br \colon \gal{F} \to G(\Fpb)$ is odd if for all $v \mid \infty$,
\[
h^0(\Gamma_{F_v}, \br(\fgder)) = \dim (\mr{Flag}_{G^0}),
\]
where $\fgder$ is the Lie algebra of the derived group $G^{\mr{der}}$
of $G^0$, and $\mr{Flag}_{G^0}$ \index[terms]{F@$\mr{Flag}_{G^0}$} is the flag variety of $G^0$.
\end{defn}
Note that for any involution of $\fgder$, the dimension of the space of invariants must be at least $\dim (\mr{Flag}_{G^0})$. An adjoint group contains an order 2 element whose invariants have dimension $\dim (\mr{Flag}_{G^0})$ if and only if $-1$ belongs to the Weyl group of $G$. When $-1$ does not belong to the Weyl group, we can (after choosing a pinning) find such an order two element in $G \rtimes \Out(G)$; for more details, see \cite[\S 4.5, \S 10.1]{stp:exceptional}. Also note that the definition implies that $F$ is totally real. That said, the ``odd" case does have implications in certain CM settings. For example, let $F$ be quadratic imaginary, and let $\br \colon \gal{F} \to \mr{GL}_n(\Fpb)$ be an irreducible representation such that
\[
\br^c \cong \br^\vee \otimes \mu|_{\gal{F}},
\]
where $\mu \colon \gal{\Q} \to \Fpb^\times$ is a character. Moreover assume that when we realize this essential conjugate self-duality as a relation
\[
\br(cgc^{-1})= A {}^t \br(g)^{-1} A^{-1} \mu(g)
\]
for some $A \in \mr{GL}_n(\Fpb)$ (and all $g \in \gal{F}$), the scalar $A \cdot {}^t A^{-1}$ (which is easily seen to be $\pm1$) actually equals $+1$. Then the pair $(\br, \mu)$ can be extended to a homomorphism 
\[
\bar{r} \colon \gal{\Q} \to (\mr{GL}_n \times \mr{GL}_1)(\Fpb) \rtimes \{1, j\},
\]
where $j^2=1$ and $j(g, a)j^{-1}= (a\cdot {}^t g^{-1}, a)$, and this $\bar{r}$ is odd in the sense of Definition \ref{odd}.

There are essentially two techniques for approaching cases of Question
\ref{question}. For classical groups, automorphy lifting and potential
automorphy theorems, via a technique introduced in
\cite{khare-wintenberger:serre0}, yield the most robust results. For
instance, the strongest lifting results in the previous example ($\br$
essentially conjugate self-dual over a quadratic imaginary field)
follow from the work of Barnet-Lamb, Gee, Geraghty, and Taylor
(\cite{blggt:potaut}). For general $G$, however, we do not have a good
understanding of automorphic Galois representations, and we must rely
on purely Galois-theoretic methods. Ramakrishna developed the first
such method in the papers \cite{ramakrishna:lifting} and \cite{ramakrishna02}, which, as noted above,
resolved Question \ref{question} in the setting of Serre's original
modularity conjecture ($F=\Q$, $G= \mr{GL}_2$, $\br$ odd and
irreducible). Our work relies on the methods of \cite{ramakrishna:lifting} and \cite{ramakrishna02},
particularly as extended in the ``doubling method'' of \cite{klr} and
the work of Hamblen and Ramakrishna (\cite{ramakrishna-hamblen}).

From now on, we will replace $\Zpb$ (resp.~$\Fpb$) by the ring of
integers $\mc{O}$ \index[terms]{O@$\mc{O}$} in some finite extension $E$ \index[terms]{E@$E$} of $\Q_p$ (resp.~the
residue field $k$ of $\mc{O}$). We let $\vpi$ \index[terms]{p@$\varpi$} denote a uniformizer of
$\mc{O}$ and ${m_{\mc{O}}} = (\vpi)$ the maximal ideal of $\mc{O}$. Thus we
now take $G$ as before but defined over $\mc{O}$, and we study
continuous homomorphisms $\br \colon \gal{F, {\mc{S}}} \to G(k)$, where ${\mc{S}}$
is a finite set of primes containing those above $p$, and $\gal{F, {\mc{S}}}$ \index[terms]{G@$\gal{F, \mc{S}}$}
denotes $\Gal(F({\mc{S}})/F)$ for $F({\mc{S}})$ the maximal (inside $\ov{F}$) Galois
extension of $F$ that is unramified away from ${\mc{S}}$.

There are several difficulties in extending the method of
\cite{ramakrishna02} to lifting odd irreducible representations to
$G(k)$ for general groups:
\begin{itemize}
\item In the arguments of \cite{ramakrishna02} one must construct at
  all primes $v$ at which $\br$ is ramified a formally smooth
  irreducible component of the local lifting ring
  $R^{\square}_{\br|_{\gal{F_v}}}$ (for $v$ not above $p$) or a
  formally smooth component of the lifting ring that parametrizes
  lifts of a fixed inertial type and (Hodge--Tate regular) $p$-adic Hodge type (for $v$
  above $p$). Such components do not
    always  exist in the level of generality in which we work.
\item The lack of smooth components as above, and more precisely the fact that $\br|_{\gal{F_v}}$ may not have a Witt-vector valued lift corresponding to a formally smooth point of the generic fiber local lifting ring, necessitates
  working with more general (typically ramified) coefficients
  $\mc{O}$, while in \textit{loc.~cit.~}one can work with the ring of Witt
  vectors $W(k)$.  This causes complications related to the fact that
  if $\mc{O}$ is ramified, $\mc O/\varpi^2$ is of characteristic $p$ and hence  isomorphic to the dual numbers
  $k \oplus k[\varepsilon]$.
\item The auxiliary prime arguments of \cite{ramakrishna02} break down as the image of $\br$ gets smaller. For general $G$, where many possible images can still lead to ``irreducible'' $\br$, this is a basic difficulty.
\end{itemize}

These difficulties are not as serious an impediment
for $G= \mr{GL}_2$ as compared to the case of general $G$. In
\cite{ramakrishna02}, under mild hypotheses on $\br|_{\gal{\Q_p}}$,
the necessary local theory is worked out (we should note, however,
that particularly at the prime $p$ the situation is here considerably
simplified by working over $\Q_p$ rather than a ramified
extension). As for the global hypotheses, by a theorem of Dickson any
irreducible subgroup of $\mr{GL}_2(k)$ (for $p \geq 7$) either has
order prime to $p$, in which case one can take the ``Teichm\"{u}ller"
lift, or has projective image conjugate to a subgroup of the form
$\mr{PSL}_2(k')$ or $\mr{PGL}_2(k')$ for some subfield $k'$ of
$k$. This allows Ramakrishna to restrict to the case where the adjoint
representation $\mr{ad}^0(\br)$ is absolutely irreducible.

For higher-rank $G$, the global arguments of \cite{ramakrishna02} work
with little change under the corresponding assumption that the adjoint
representation $\br(\fgder)$ \index[terms]{R@$\br(\fgder)$} (this will be our notation for the Galois
module $\fgder$, equipped with the action of $\Gamma_F$ via
$\Ad \circ \br$) is absolutely irreducible. Such a generalization is
carried out in \cite{stp:exceptional}. The same paper also proves a
variant with somewhat smaller image, in which $\im(\br)$ contains
(approximately) $\varphi(\mr{SL}_2(k))$, where
$\varphi \colon \mr{SL}_2 \to G$ is a principal $\mr{SL}_2$. In this
case $\br(\fgder)$ decomposes into $r$ irreducible factors, where $r$
is the semisimple rank of $G$, and the final result depended on an
explicit analysis of this decomposition, requiring case-by-case
calculations depending on the Dynkin type, with the result only
verified for the exceptional groups via a computer calculation. More
seriously, the method did not apply to groups of type $D_{2m}$, for
which $\fgder$ is not multiplicity-free as an $\mr{SL}_2$-module (one
factor occurs with multiplicity two). Some other instructive examples
of how variants of the familiar Ramakrishna arguments still fail to
treat relatively simple images can be found in \cite{tang:thesisANT}.

\subsection{Main theorem}

Before explaining how we overcome the difficulties mentioned above, we will state the main theorem. From now on we will require of $G$ that the component group $\pi_0(G)$ is finite \'{e}tale of order prime to $p$. (See \S \ref{notation} and the beginning of \S \ref{defprelimsection} for the group theory and deformation theory notations.)
\begin{thmalph}[See Theorem \ref{mainthm}]\label{mainthmintro}
Let $p \gg_G 0$ be a prime. Let $F$ be a totally real field, and let $\br \colon \gal{F, {\mc{S}}} \to G(k)$ be a continuous representation unramified outside a finite set of finite places ${\mc{S}}$ containing the places above $p$. Let $\tF$ \index[terms]{F@$\tF$} denote the smallest extension of $F$ such that $\br(\gal{\tF})$ is contained in $G^0(k)$, and assume that  $[\tF(\zeta_p): \tF]$ is strictly greater than the integer $a_G$ arising in Lemma \ref{cyclicq} (which depends only on the root datum of $G$). Fix a geometric lift $\mu \colon \gal{F, {\mc{S}}} \to G/G^{\mr{der}}(\mc{O})$ of $\bar{\mu}:=\br \pmod{G^{\mr{der}}}$, and assume that $\br$ satisfies the following:
\begin{itemize}
\item $\br$ is odd, i.e. for all infinite places $v$ of $F$, $h^0(\gal{F_v}, \br(\fgder))= \dim(\mr{Flag}_{G^{\mr{der}}})$.
\item $\br|_{\gal{\tF(\zeta_p)}}$ is absolutely irreducible.
\item For all $v \in {\mc{S}}$, $\br|_{\gal{F_v}}$ has a lift $\rho_v \colon \gal{F_v} \to G(\mc{O})$ of type $\mu|_{\gal{F_v}}$; and that for $v \mid p$ this lift may be chosen to be de Rham and regular in the sense that the associated Hodge--Tate cocharacters are regular.
\end{itemize}
Then there exist a finite extension $E'$ of $E=\Frac(\mc{O})$ (whose
ring of integers and residue field we denote by $\mc{O}'$ and $k'$) depending only on the set $\{\rho_v\}_{v \in {\mc{S}}}$; a finite set of places $\wt{{\mc{S}}}$ containing ${\mc{S}}$; and a geometric lift 
\[
\xymatrix{
& G(\mc{O}') \ar[d] \\
\gal{F, \wt{{\mc{S}}}} \ar[r]_{\br} \ar[ur]^{\rho} & G(k') 
}
\]
of $\br$, and having projection to $G/G^{\mr{der}}(\mc{O'})$ equal to $\mu$, such that $\rho(\gal{F, \wt{{\mc{S}}}})$ contains
$\wh{G^{\mr{der}}}(\mc{O}')$. Moreover, if we fix an integer $t_0$ and
for each $v \in {\mc{S}}$ an irreducible component defined over $\mc{O}$
 and containing $\rho_v$ of:
\begin{itemize}
\item for $v \in {\mc{S}} \setminus \{v \mid p\}$, the generic fiber
  of the local lifting ring, $R^{\square,
    \mu}_{\br|_{\gal{F_v}}}[\frac{1}{\vpi}]$ (where $R^{\square,
    \mu}_{\br|_{\gal{F_v}}}$  pro-represents $\Lift_{\br|_{\gal{F_v}}}^{\mu}$); and
\item for $v \mid p$, the lifting ring $R_{\br|_{\gal{F_v}}}^{\square,
    \mu, \tau, \mbf{v}}[1/\vpi]$
  whose $\ov{E}$-points parametrize lifts of $\br|_{\gal{F_v}}$ with specified inertial type $\tau$ and Hodge type $\mbf{v}$ (see \cite[Prop. 3.0.12]{balaji} for the construction of this ring);  
\end{itemize}
then \index[terms]{R@$R^{\square,\mu}_{\br|_{\gal{F_v}}}$}
\index[terms]{R@ $R_{\br|_{\gal{F_v}}}^{\square, \mu, \tau, \mbf{v}}[1/\vpi]$}  
the global lift $\rho$ may be constructed such that, for all $v \in {\mc{S}}$, $\rho|_{\gal{F_v}}$ is congruent modulo $\vpi^{t_0}$ to some $\wh{G^{\mr{der}}}(\mc{O}')$-conjugate of $\rho_v$, and $\rho|_{\gal{F_v}}$ belongs to the specified irreducible component for every $v \in {\mc{S}}$.
\end{thmalph}
Thus we prove an essentially complete ``local-to-global" principle for finding geometric lifts of irreducible odd Galois representations. We make a few remarks:
\begin{rmk} $ $ 
  \begin{itemize}
  \item By $p \gg_G 0$ we mean that $p$ is larger than a constant
    depending only on the root datum of $G$.
\item The arguments proceed from a somewhat different global image assumption---see Assumption \ref{multfree}---but for $p \gg_{G, \tF} 0$ the absolute irreducibility hypothesis implies the other conditions in Assumption \ref{multfree} (see Corollary \ref{irrcor}). 
\item The need to extend the coefficient ring from $\mc{O}$ to $\mc{O}'$ arises entirely from local problems at the primes in ${\mc{S}}$: we need the existence of formally smooth points, approximating the given $\rho_v$ sufficiently well, on the given irreducible components of the local lifting rings. To find such points may require extending the coefficient ring (see Lemma \ref{smoothapprox}). (Once $\mc{O}'$ meets these requirements, we may if desired further enlarge it and still arrange that the lift has image containing $\wh{G^{\mr{der}}}(\mc{O}')$.) Once this $\mc{O}'$ depending on $\{\rho_v\}_{v \in {\mc{S}}}$ is fixed, it is not enlarged again in the course of the proof: see the proof of Theorem \ref{mainthm} (which proves Theorem \ref{mainthmintro}).     
\item The method also allows us to lift a representation $\rho_n
  \colon \gal{F, {\mc{S}}} \to G(\mc{O}/\vpi^n)$ to a geometric
  representation, provided that (in addition to the global hypotheses
  on $\br$) the restrictions $\rho_n|_{\gal{F_v}}$ for $v \in {\mc{S}}$ have
  $G(\mc{O})$-lifts $\rho_v$ as in the theorem statement. 
\item The bound on $p$ can be made effective: for detailed remarks, see Remark \ref{effective}.
\item In \S \ref{examples} we give some examples of the theorem. 
\item All the lifts $\rho$ produced by the theorem have image containing $\wh{G^{\mr{der}}}(\mc{O}')$; that is, we find lifts whose image is ``as large as possible" subject to the given $\im(\br)$. Thus, even in a setting where an obvious ``Teichm\"{u}ller" lift exists (i.e., when $|\im(\br)|$ is coprime to $p$), we produce very different sorts of lifts, producing congruences between finite-image and ``full" image Galois representations.
\item For more detailed remarks on how, for $G^0= \mr{GL}_n$, this theorem compares to results coming from potential automorphy theorems, see Remark \ref{potautcompare}.
\end{itemize}
\end{rmk}
We mention two variants of the theorem that are straightforward given
our techniques. The first (see Theorem \ref{notodd}) is a
non-geometric but finitely-ramified lifting theorem for $\br$ without
any constraints on the action of $\Gamma_{F_v}$ for $v \mid \infty$
(and in particular allowing $F$ to be any number field); this holds
under the same image hypotheses as Theorem \ref{mainthmintro}. The
second produces possibly de Rham but infinitely-ramified lifts,
generalizing the main theorem of \cite{klr} from the case $G =
\mr{GL}_2$ and $\mr{SL}_2(\Fp) \subset \im(\br)$.
\begin{thmalph}[See Corollary \ref{infiniteram}]\label{mainthmB}
Let $F$ be any number field. Assume $p \gg_G 0$, and let $\br \colon \gal{F, {\mc{S}}} \to G(k)$ be a representation satisfying
\begin{itemize}
 \item $H^1(\Gal(K/F), \br(\fgder)^*)$=0, where $K= \tF(\fgder,\mu_p)$.
 \item $\br(\fgder)$ and $\br(\fgder)^*$ are semisimple $\mathbb{F}_p[\gal{F}]$-modules (equivalently, semisimple $k[\gal{F}]$-modules) having no common $\Fp[\gal{F}]$-subquotient, and neither contains the trivial representation.
\end{itemize}
Fix a lift $\mu$ of $\bar{\mu}$ as in Theorem \ref{mainthmintro}. Assume that for all $v \in {\mc{S}}$, there are lifts $\rho_v \colon \gal{F_v} \to G(\mc{O})$ of $\br|_{\gal{F_v}}$ with multiplier $\mu$. Then there exists an infinitely ramified lift
\[
\xymatrix{
& G(\mc{O}) \ar[d] \\
\gal{F} \ar@{-->}[ur]^{\rho} \ar[r]_{\br} & G(k)
}
\]
with multiplier $\mu$ such that $\rho|_{\gal{F_v}}$ is $\wh{G^{\mr{der}}}(\mc{O})$-conjugate to $\rho_v$ for all $v \in {\mc{S}}$, and $\rho(\gal{F})$ contains $\wh{G^{\mr{der}}}(\mc{O})$.
\end{thmalph}
The results of Appendix \ref{groupsection} show that for $p \gg_G 0$,
the first two hypotheses of the theorem hold when
$\br|_{\gal{\tF(\zeta_p)}}$ is absolutely irreducible, and
$[\tF(\zeta_p):\tF]$ is greater than an integer $a_G$ (depending only
on $G$) described in Lemma \ref{cyclicq}.

\subsection{Strategy of  proof, the doubling method and  relative deformation theory}

\subsubsection*{\textbf{The method of Hamblen--Ramakrishna}}
We first briefly recall the original technique of Ramakrishna (see
\cite{ramakrishna02} and \cite{taylor:icos2}). Under the oddness
hypothesis, one defines a global Galois deformation problem (over $W(k)$, typically) 
by imposing formally smooth local deformation
conditions on the restriction of $\br$ to primes in $\mc{S}$, and whose
associated (mod $p$) Selmer and dual Selmer groups have the same
dimension (we will informally say that Selmer and dual Selmer are
``balanced"). In this setting an application of the Selmer group
variant of the Poitou--Tate sequence (see \cite[Lemma
1.1]{taylor:icos2}) implies that if the dual Selmer group vanishes,
then the corresponding universal deformation ring is $W(k)$ and
therefore gives rise to a geometric deformation of $\br$. The task,
then, is to allow ramification at a set ${\mc{Q}}$ of auxiliary primes such
that:
\begin{itemize}
\item the allowed ramification at each $q \in {\mc{Q}}$ is a formally smooth local condition;
\item the conditions at $q \in {\mc{Q}}$ have large enough tangent space that the resulting new Selmer and dual Selmer groups remain ``balanced" as we add each $q \in {\mc{Q}}$; and
\item when we have allowed the entire auxiliary set ${\mc{Q}}$ of
  ramification, the dual Selmer group, hence also the Selmer group,
  vanishes.
\end{itemize}

 In the case $F = \Q$, Ramakrishna takes a Steinberg local condition at primes $q \not \equiv \pm 1 \pmod p$ at which $\br$ is unramified with distinct Frobenius eigenvalues with ratio $q$. By comparing splitting conditions on Selmer and dual Selmer classes, he shows (when the projective image of $\br$ contains $\mr{PSL}_2(k)$) that such $q$ can be chosen that inductively decrease the size of Selmer and dual Selmer.

We refer to the process of starting with a Selmer group
$H^1_{\mc{L}}(\Gamma_{F,{\mc{S}}},\br(\fg^{\rm der}))$
\index[terms]{H@$H^1_{\mc{L}}(\Gamma_{F,{\mc{S}}},\br(\fg^{\rm der}))$}
(resp.~dual Selmer
group $H^1_{\mc{L}^\perp}(\Gamma_{F,{\mc{S}}},\br(\fg^{\rm der})^*)$ \index[terms]{H@$H^1_{\mc{L}^\perp}(\Gamma_{F,{\mc{S}}},\br(\fg^{\rm der})^*)$}), with
local conditions defined by subspaces
$L_v \subset H^1(\Gamma_{F_v}, \br(\fg^{\rm der}))$, $v \in {\mc{S}}$, and finding
a finite set of places ${\mc{Q}}$ with subspaces
$L_v \subset H^1(\Gamma_{F_v},\br(\fg^{\rm der}))$, $v \in {\mc{Q}}$, such that
$H^1_{\mc{L}}(\Gamma_{F,{\mc{S}}\cup {\mc{Q}}},\br(\fg^{\rm der}))$
(resp. $H^1_{\mc{L}^\perp}(\Gamma_{F,{\mc{S}} \cup {\mc{Q}}},\br(\fg^{\rm der})^*)$)
is $0$, as killing mod $p$ Selmer (resp.~dual Selmer). We call these
groups {\it intrinsic} mod $p$ (dual) Selmer groups as opposed to {\it
  relative} (extrinsic) versions of them that we define later.

In higher rank, it is better to think of the $\mr{GL}_2$ ``Steinberg''
condition as allowing unipotent ramification in the direction of a
fixed root space and constraining Frobenius to act by the cyclotomic
character on this root space, since a key point in controlling the
Selmer and dual Selmer groups simultaneously is that the Ramakrishna
deformation condition should intersect the unramified condition in a
codimension one subspace. At each step of the inductive argument that
decreases the size of the Selmer groups, one has to, given non-zero
Selmer and dual Selmer classes, be able to choose this root space in a
suitably general position with respect to the images of cocycles
representing these classes. This is not always possible when (as in
\cite{stp:exceptional}) the auxiliary primes are chosen so that
Frobenius acts by a regular semisimple element.  To overcome this, we
follow the path taken in \cite{ramakrishna-hamblen}, and generalize
the notion of \textit{trivial primes} from the work of Hamblen and
Ramakrishna (\cite{ramakrishna-hamblen}) to the present context. Thus,
we use auxiliary primes $v$ having the one behavior we are guaranteed
to find in the image of any representation, namely, that
$\br|_{\gal{F_v}}$ is trivial; note that as $\br(\mr{Frob}_v)$ is then
contained in \textit{every} maximal torus of $G$, we win a great deal
of flexibility in the choice of root space in which to allow
ramification (contrast the condition (6) in \cite[\S
5]{stp:exceptional} with our Proposition \ref{prop:killrel}).

Hamblen and Ramakrishna show how to deform a \textit{reducible} but
indecomposable representation $\gal{\Q} \to \mr{GL}_2(k)$ to an
irreducible representation over $W(k)$ by allowing Steinberg-type
ramification at primes $q$ such that $q \equiv 1 \pmod p$,
$q \not \equiv 1 \pmod {p^2}$, and $\br|_{\gal{\Q_q}}$ is trivial. The
resulting local condition on lifts of $\br|_{\gal{\Q_q}}$ is liftable
but is not representable, the latter point being reflected in the fact
that the local condition behaves very differently modulo different
powers of $p$: while its tangent space is ``too small" for the global
applications, certain lifts mod $p^n$ for $n \geq 3$ do indeed witness
that the condition is coming from a sufficiently large characteristic
zero condition (see Lemma \ref{extracocycles^N} for a precise
formulation of this distinction). The consequence of this distinction
is that the global argument must treat separately the problems of
lifting $\br$ to a mod $p^3$ representation and lifting it modulo
higher powers of $p$. Because of the generality in which we work, and
the demands of the relative deformation theory argument we have to
adopt, we need a more elaborate version that separates the two
problems of lifting mod $\vpi^N$ for some $N \gg 0$ and lifting beyond
mod $\vpi^N$.

 \subsubsection*{\textbf{The lifting method of this paper}}
 Now we come to the main technical innovations in the paper.  We do
 not spell out here the local conditions at the places where we allow
 our representations to ramify (for which see \S
 \ref{trivialsection}), but instead concentrate on the shape of the
 global arguments that are novel to our work. (The notation used here is lighter, less accurate, and not identical
 to what is used in the main text.)
 
 We find it convenient first to make a couple of definitions.  Let $v$
 be a finite place of $F$ and
 $\rho_M:\Gamma_{F_v} \to G(\mc O/\vpi^M)$ be a lifting of a residual
 representation $\br: \Gamma_{F_v} \to G(k)$. For any
 $G(\mc{O}/\vpi^r)$-valued homomorphism $\rho_r$ of a local or global
 Galois group, $\rho_r(\fgder)$ \index[terms]{R@$\rho_r(\fgder)$} will
 denote $\fgder \otimes_{\mc{O}} \mc{O}/\vpi^r$ equipped with the
 $\Ad \circ \rho_r$ action.
\begin{defn}\label{balancedness}
   We say that an $\mc{O}$-submodule \index[terms]{L@$L_{M,v}$}
   $L_{M,v} \subset H^1(\gal{F_v},\rho_M(\fgder))$ is \emph{balanced}
   \index[terms]{balanced} if
\begin{equation*} \label{eq:orderintro}
  |L_{M,v}| = \begin{cases}
    |\rho_M(\fgder)^{\Gamma_{F,v}}| &\mbox{ if } v \nmid p; \\
    |\rho_M(\fgder)^{\Gamma_{F,v}}|\cdot
    |\mc{O}/\varpi^M|^{\dim_k(\mf{n})[F_v:\Q_p]} & \mbox{ if } v \mid p.
  \end{cases}
\end{equation*}
Here $\mf{n}$ \index[terms]{N@$\mf{n}$} is the Lie algebra of the unipotent radical of a Borel
subgroup of $G$.  \end{defn}

\begin{defn}\label{relgoodpos}
  Given positive integers $M \leq m$, and a place $v$ of $F$, we say
  that, for $n \geq m$, a pair of representations $(\rho_n,\rho_{n+M})$, with
  $\rho_n \colon \gal{F_v} \to G(\mc{O}/\vpi^n)$,
  $\rho_{n+M} \colon \gal{F_v} \to G(\mc{O}/\vpi^{n+M})$, and
  $\rho_{n+M}$ reducing to $\rho_n$ modulo $\vpi^n$, is in {\it
    relative good position} with respect to the data
  $\left( \{ {\rm Lift}_v(\mc O/\vpi^n) \}_{n \geq m},L_{M,v} \right)$ if:
 \begin{itemize} 
 \item For all $n \geq m$, ${\rm Lift}_v(\mc O/\vpi^n)$  is a set of lifts  of $\br$ such that reduction induces a surjective map  ${\rm Lift}_v(\mc O/\vpi^{n+M}) \to {\rm Lift}_v(\mc O/\vpi^n)$;
  \item  $\rho_n$ and  $\rho_{n+M}$ belong to ${\rm Lift}_v(\mc O/\vpi^{n})$ and  ${\rm Lift}_v(\mc O/\vpi^{n+M})$;
\item $L_{M,v} \subset H^1(\gal{F_v},\rho_M(\fgder))$ is an $\mc
  O$-submodule, and the fibers of the map ${\rm Lift}_v(\mc
  O/\vpi^{n+M}) \to {\rm Lift}_v(\mc O/\vpi^n)$ are  stable under the natural  action of the preimage of $L_{M,v}$ in the space of one-cocycles;
\item $L_{M,v}$ is balanced.
\end{itemize}
  
If the data $\left(\{{\rm Lift}_v(\mc O/\vpi^n)\}_{n \geq m},L_{M,v}\right)$ is understood, then
we simply say that the pair of representations $(\rho_n,\rho_{n+M})$
is in relative good position.

\end{defn}

 Our results as in Theorem \ref{mainthmintro}  have a local-global
 flavor, and for simplicity in this section we assume that the  $\mc
 O$-valued  lifts $\rho_v$ for $v \in {\mc{S}}$ that we interpolate (mod
 $\vpi^{t_0}$) are smooth points of the corresponding local framed
 deformation ring.  (In this case our methods prove the existence of
 geometric lifts of $\br$ that are themselves $\mc{O}$-valued, without the need to make a finite extension.)  By a
 lemma that we deduce from a result of Serre (cf.~Lemma
 \ref{lem:serre}), given a positive integer
 $M$ (specified in advance, and for us coming from the global setup), there exists an $n_0 \gg 0$, in particular $n_0 \geq  {\rm
   max}(t,M)$, and an open neighborhood of
 the lift $\rho_v$ (viewed as a point in the generic fibre of a
 suitable framed deformation ring), enabling one for all $n\geq n_0$ to specify sets of lifts ${\rm Lift}_v({\mc O}/\vpi^n)$ that reduce to  $\rho_v \pmod{\vpi^{n_0}}$, and  
such that the fibers of the surjective  maps ${\rm Lift}_v({\mc
  O}/\vpi^{n+r})  \rightarrow {\rm Lift}_v({\mc O}/\vpi^n)$  for all $r
\leq M$ are  stable  under an $\mc O$-submodule  of cocycles in $Z^1(\Gamma_{F_v}, \rho_r(\fgder))$ that contains all coboundaries, and whose image in $H^1(\Gamma_v,\rho_r(\fgder))$ is  balanced (cf.~Definition \ref{balancedness}).  This purely local condition implies that the corresponding  Selmer groups $H^1_{\mathcal{L}_r}(\Gamma_{F,{\mc{S}}},\rho_r(\fgder))$  have the same cardinality
as $H^1_{\mathcal{L}_r^{\perp}}(\Gamma_{F,{\mc{S}}},(\rho_r(\fgder))^*)$. Furthermore the lifts in (the non-empty set) $\varprojlim {\rm Lift}_v({\mc O}/\vpi^n)$ lie in the same component of the local framed deformation ring as $\rho_v$ (for $v \in {\mc{S}}$).

As in \cite{ramakrishna-hamblen}, our work takes place in two steps
with different flavors.

\subsubsection*{\textbf{Step 1: Mod $\vpi^N$ liftings using the doubling method of \cite{klr}}}
Let $N$ be a positive integer, which we will choose sufficiently large compared to an integer $M$ coming from the global situation (namely, $\im(\br)$) and compared to local information (the fixed lifts $\rho_v$, $v \in {\mc{S}}$, and the local geometry of local lifting rings around the $\rho_v$, as just discussed). Using a refinement of the
doubling method that was introduced in \cite{klr}, we lift $\br$ to a
$\rho_N \colon \gal{F, {\mc{S}}'} \to G(\mc{O}/\vpi^N)$ that at places
$v \in {\mc{S}}$ equals (up to strict equivalence) the given liftings $\rho_v \pmod{\vpi^{t_0}}$. In doing 
so we have to enlarge the set ${\mc{S}}$ to a finite set of places ${\mc{S}}'
\supset {\mc{S}}$ by
allowing ramification at an auxiliary set of primes (cf.~Definition
\ref{relativetriv} and Lemma \ref{extracocycles^N}), which are again a
generalization of the trivial primes of Hamblen--Ramakrishna, and we have to
specify a class of liftings ${\rm Lift} _v({\mc O}/\vpi^n)$,
for $n \geq N-M$, $v \in {\mc{S}}'$, such that $\rho_N|_{\Gamma_{F_v}}$ is such a lifting,
and such that for all $v \in {\mc{S}}'$ and $1 \leq r \leq M$, the fibers of
the surjective map
${\rm Lift}_v({\mc O}/\vpi^{n+r}) \rightarrow {\rm Lift}_v({\mc
  O}/\vpi^n)$ are stable  under a set of
cocycles in $Z^1(\Gamma_{F_v}, \rho_r(\fgder))$ whose image $L_{r,v} $ in
$H^1(\Gamma_{F_v},\rho_r(\fgder))$ is balanced.

The construction of the lift $\rho_N \colon \gal{F, {\mc{S}}'} \to
G(\mc{O}/\vpi^N)$   is achieved by a  generalization of the methods of
\cite{klr} and \cite{ramakrishna-hamblen}, new arguments being  needed
to handle general $\im(\br)$.  In \S \ref{klrsection}, we start with
any mod ${\vpi^2}$ lift $\rho_2$ of $\br$ (easily seen to exist after
enlarging ${\mc{S}}$ by a set of trivial primes); to further lift it we
modify $\rho_2$ so that its local restrictions at primes of
ramification match certain specified local lifts (namely, those coming
from the reductions of the $G(\mc{O})$-lifts $\rho_v$ assumed to exist in
Theorem \ref{mainthmintro}). This leads to the following question:
given local cohomology classes $z_{\mc{T}}= (z_w)_{w \in {\mc{T}}} \in \bigoplus_{w
  \in {\mc{T}}} H^1(\gal{F_w}, \br(\fgder))$ (here ${\mc{T}}$ will be a finite set
of primes containing the original set ${\mc{S}}$ of ramification), can we
find a global class $h \in H^1(\gal{F, {\mc{T}}}, \br(\fgder))$ such that
$h|_{\gal{F_w}}=z_w$ for all $w \in {\mc{T}}$?

The answer is no, so we aim for the next best thing: to enlarge ${\mc{T}}$ to
a finite set ${\mc{T}} \cup {\mc{U}}$, and to find a class
$h^{\mc{U}} \in H^1(\gal{F, {\mc{T}} \cup {\mc{U}}}, \br(\fgder))$ such that $h^{\mc{U}}|_{\mc{T}}=
z_{\mc{T}}$. This would allow us to modify $\rho_2$ to some
$(1+\vpi h^{\mc{U}})\rho_2$ that is well-behaved at primes in ${\mc{T}}$. The
problem here is that we sacrifice control at the primes in ${\mc{U}}$, and
this necessitates the use of an idea from \cite{klr} (as exploited in
a simpler setting than ours by \cite{ramakrishna-hamblen}), which we will refer
to as the ``doubling method": roughly speaking, we consider two such
sets ${\mc{U}}$ and ${\mc{U}}'$, with corresponding cocycles $h^{\mc{U}}$ and $h^{{\mc{U}}'}$. By
considering all possibilities $(1+\vpi(2h^{\mc{U}}-h^{{\mc{U}}'}))\rho_2$ as ${\mc{U}}$ and
${\mc{U}}'$ vary (each through \v{C}ebotarev multi-sets of primes), we show
by a limiting argument that there is \textit{some} pair of ${\mc{U}}$ and ${\mc{U}}'$ (but not, as far as we can tell, a
\v{C}ebotarev set of possible choices!) such that
$\rho'_2= (1+\vpi(2h^{\mc{U}}-h^{\mc{U}'}))\rho_2$ both has the desired behavior
at ${\mc{T}}$ and is under enough control at ${\mc{U}}$ and ${\mc{U}}'$ (the detailed
desiderata come out of Lemma \ref{extracocycles^N}) for subsequent
steps in the lifting argument. In \S \ref{klr^Nsection} we run a more
complicated version of this argument, iterating it to produce the
desired lift modulo $\vpi^N$. In both of these arguments, handling the
case of general $\im(\br)$ poses a significant challenge beyond the
$\mr{GL}_2$ arguments of \cite{klr} and \cite{ramakrishna-hamblen}; in
particular, handling multiplicities in the
$\Fp[\gal{F}]$-decomposition of $\br(\fgder)$ requires new arguments.

To summarize, at each stage in Step 1, when constructing a mod
$\varpi^{n+1}$ lifting of a mod $\varpi^n$ lift, the doubling method
(which is based on duality) needs to allow more primes to ramify to
overcome obstructions. This is used to overcome the numerical
disadvantage one is at for small $n$, on account of the tangent space
to the local conditions at the auxiliary trivial primes being ``too
small" and not balanced. In particular, since new ramification is
added in lifting from each $G(\mc{O}/\vpi^n)$ to
$G(\mc{O}/\vpi^{n+1})$, Step 1 alone will not lead to geometric lifts
of $\br$, and we move to Step 2. (Step 1, however, does suffice for
the proof of Theorem \ref{mainthmB}.)

\subsubsection*{\textbf{Step 2: Relative deformation theory}} Having ``risen
above'' the singularities of the local deformation rings by lifting
$\br$ to $\rho_N \colon \gal{F, {\mc{S}}'} \to G(\mc{O}/\vpi^N)$ for
$N \gg 0$, we would like to find a geometric lift of $\rho_N$
following the methods of Hamblen--Ramakrishna
\cite{ramakrishna-hamblen} (who only needed to consider $N=2$). We run
into the problem that we may not be able to kill the dual Selmer group
$H^1_{\mathcal{L}^{\perp}}(\Gamma_{F,{\mc{S}}'}, \br(\fg^{\rm der} )^*
)$.  To kill this group we are obliged to kill the corresponding
Selmer group following Ramakrishna's technique, and the classes coming
from inflation from $H^1({\rm im}(\rho_N), \br(\fg^{\rm der}))$ cannot
be killed using ``trivial primes'' that are good for $\rho_N$ (as in
Definition \ref{relativetriv}). Note that when $\mc O$ is ramified,
the group $H^1({\rm im}(\rho_N), \br(\fg^{\rm der}))$ does not vanish
even for $N=2$ (as $\mc O/\vpi^2$ is isomorphic to the dual numbers
$k \oplus k [\varepsilon]$ as alluded to earlier). It also does not
vanish for certain choices of $\im(\br)$ even when $\mc{O}$ is
unramified (cf. \cite[Example 5.5]{fkp:arxivold}).
 
One of our main observations at this stage is that we can still kill,
for an appropriate choice of $M \leq N$ and for a choice of a finite
set of ``trivial primes'' ${\mc{Q}}$, the relative mod $p$ dual Selmer
group
$\ov{H^1_{\mc{L}_M^{\perp}}(\Gamma_{F,{\mc{S}' \cup \mc{Q}}},
  \rho_M(\fgder)^*)}$, namely the image of the map
$H^1_{\mathcal{L}_M^{\perp}}(\Gamma_{F,{\mc{S}}' \cup
  {\mc{Q}}},\rho_M(\fgder)^*) \rightarrow
H^1_{\mathcal{L}_1^{\perp}}(\Gamma_{F,{\mc{S}}' \cup
  {\mc{Q}}},\br(\fgder)^*)$ (cf.~Definition \ref{def:relativeSelmer}).
Here the choice of $M$ is dictated by having to ensure that the image
of
$H^1({\rm im}(\rho_N), \rho_M(\fg^{\rm der})) \rightarrow H^1({\rm
  im}(\rho_N), \br(\fg^{\rm der}))$ is $0$ (cf.~Lemma \ref{lemma:fn})
for $N \geq M$. This is deduced from well-known results on vanishing
of cohomology of semisimple Lie algebras and some results of Lazard
on cohomology of $p$-adic Lie groups (cf.~Corollary
\ref{cor:lazard}). A different consequence of the results of Lazard
has been used in a related context previously by Kisin (see
\cite[Lemma 6.2, Proposition 6.4]{kisin:geometricdef}); but whereas
\textit{loc.~cit.~}applies these results to study (via the
Taylor--Wiles method) the structure of a global deformation space at a
fixed characteristic zero point, we are using them in a quite
different way to show the existence of a characteristic zero point.

In addition to Lazard's results, a crucial ingredient
of the proof of the killing of relative (dual) Selmer (cf.~Proposition
\ref{prop:killrel} and Theorem \ref{thm:killrel}), beyond the
versatility of trivial primes for killing cohomology classes, is that
the relative mod $p$ Selmer and  mod $p$ dual Selmer groups  are balanced under the
assumption that $\br(\fgder)$ and $\br(\fgder)^*$ have no global
invariants (cf.~Lemma \ref{lem:bal}).

Because of our inability to kill intrinsic mod $p$ dual Selmer, and that we are able to kill  only a relative version of it,  after allowing ramification at a  finite set of trivial primes ${\mc{Q}}$
which have prescribed properties in $\rho_N$, we have to consider a more elaborate deformation
problem. Thus instead of directly lifting $\rho_N$, we now exploit the fact
that the lifting $\rho_N$ produced in Step 1, and the set ${\mc{Q}}$ used to
kill relative dual Selmer, is such that the pair
$(\rho_{N-M},\rho_N)$ is in ``relative good position'' (cf.~Definition
\ref{relgoodpos}) at all places in ${\mc{S}}' \cup {\mc{Q}}$
, including the places in ${\mc{S}}$ (this requires having chosen $N \gg 0$ relative both to $M$ and to the bounds coming from our local analysis). 
    We  construct pairs of representations $(\tau_n, \rho_{n+M})$ in relative good position at all places $v \in {\mc{S}}' \cup {\mc{Q}}$,
starting for  $n=N-M$ with the pair $(\rho_{N-M}, \rho_N)$. Namely, for each $n \geq N-M$, we  inductively construct  pairs of (fixed multiplier $\mu$) liftings $(\tau_n, \rho_{n+M})$, where $\tau_n \colon \gal{F, {\mc{S}}' \cup {\mc{Q}}} \to G(\mc{O}/\vpi^n)$ and $\rho_{n+M} \colon \gal{F, {\mc{S}}' \cup {\mc{Q}}} \to G(\mc{O}/\vpi^{n+M})$, with the following properties:
\begin{enumerate}
\item For each $v \in {\mc{S}}' \cup {\mc{Q}}$, $\tau_n|_{\gal{F_v}}$ belongs to
  ${\rm Lift}_v(\mc O/\vpi^n)$, and $\rho_{n+M}|_{\gal{F_v}}$ belongs to ${\rm Lift}_v(\mc O/\vpi^{n+M})$
\item $\tau_{n+1}= \rho_{n+M} \pmod{\vpi^{n+1}}$.
\item $\tau_n = \tau_{n+1} \pmod{ \vpi^n}$.
\item $(\tau_n, \rho_{n+M})$ are in relatively good position  at all places $v \in {\mc{S}}' \cup {\mc{Q}}$.
\end{enumerate}

\centerline{
\xymatrix{
(\tau_{N-M},&\rho_N)  \ar@{=>}[dl] \\
(\tau_{N+1-M},             &  \rho_{N+1}) \\
\ldots & \\
(\tau_n,  & \rho_{n+M}) \ar@{=>}[dl] \\
(\tau_{n+1}  & \rho_{n+1+M}) \\
\ldots &
}}

The $\rho_{n+M}$ may not be compatible as we increase $n$, but as the
$\tau_n$'s are we get our desired geometric lift (of $\rho_{N-M}$ and hence of $\br$) by setting  $\rho= \varprojlim \tau_n \colon \gal{F, {\mc{S}}' \cup {\mc{Q}}} \to G(\mc O)$.

To carry out the inductive step we use the vanishing of the relative dual Selmer and the ``diagram of relative deformation theory'' which arises by comparing the Poitou--Tate sequence for $\rho_M(\fgder)$ and $\rho_{M-1}(\fgder)$ coefficients:

  \[
    \xymatrix{
       H^1(\Gamma_{F, {\mc{S}}' \cup {\mc{Q}}}, \rho_M(\fgder)) \ar[r]
      \ar[d] & \bigoplus_{v \in {\mc{S}}' \cup {\mc{Q}}} \frac{H^1(\Gamma_{F_v}, \
      \rho_M(\fgder))}{L_{M,v}} \ar[r] \ar[d] &
    H^1_{\mc{L}_{M}^{\perp} \cup \{L_{M, v}^\perp\}_{v \in {\mc{Q}}}}(\Gamma_{F, {\mc{S}}' \cup {\mc{Q}}}, \rho_M(\fgder)^*)^{\vee} \ar[d]\\
      H^1(\Gamma_{F, {\mc{S}}' \cup {\mc{Q}}}, \rho_{M-1}(\fgder)) \ar[r]
       & \bigoplus_{v \in {\mc{S}}' \cup {\mc{Q}}} \frac{H^1(\Gamma_{F_v}, \
      \rho_{M-1}(\fgder))}{L_{M-1,v}} \ar[r]  &
    H^1_{\mc{L}_{M-1}^{\perp}\cup \{L_{M-1, v}^\perp\}_{v \in {\mc{Q}}}}(\Gamma_{F, {\mc{S}}' \cup {\mc{Q}}}, \rho_{M-1}(\fgder)^*)^{\vee} \\
  }
\]
in which the rows come from (a part) of the Poitou--Tate exact sequence, and the vertical arrows are induced by the reduction map $\rho_M(\fgder) \to \rho_{M-1}(\fgder)$ (with the third vertical arrow arising from dualizing twice), 
cf.  \S \ref{relativeliftsection} and  Theorem \ref{mainthm} (especially Claim \ref{relativediagram}).

\begin{rmk} $ $
\begin{itemize}
\item Thus the arguments are relative in two different, albeit
  related, aspects and apply once we have  lifted $\br$ to $\rho_N$
  for $N \gg 0$.
\begin{itemize}
\item The lifting arguments are relative to $\rho_N$. In fact we lift
  $\rho_{N-M}$ for suitable $0 \ll M \ll N$. This
  allows us to use only generic smoothness of generic fibres of local
  deformation rings.
\item The lifting problem is relative in that we lift pairs of representations. This allows us to get away with killing relative mod $p$ dual Selmer groups rather than intrinsic mod $p$  dual Selmer groups.
\end{itemize}

\item The obstacle that relative deformation theory helps overcome is
  that given $\rho_N: \Gamma_{F,{\mc{S}}’} \to G(\mc O/\vpi^N)$, imposing on
  a place $v$ that $\rho_N(\Frob_v)$ has a prescribed shape and that a
  given class $\phi \in H^1_{\mathcal{L}_M}(\Gamma_{F,{\mc{S}}’},\rho_M(\fgder))$ does not lie locally in a
  prescribed subspace $L_v \subset H^1(\Gamma_{F_v},\rho_M(\fgder))$
  might be incompatible.  The incompatibility, essentially a lack of
  linear disjointness issue, arises from the classes that are inflated
  from $H^1(\mr{im}(\rho_N),\rho_M(\fgder))$. This issue has been
  encountered before, for example in \cite{taylor:icos2}, \cite{khare:approx}, and \cite{khare-ramakrishna:liftingtorsion}.

  In the paper \cite{khare-ramakrishna:liftingtorsion} this issue is overcome by a different route
  which works under an ordinarity assumption (and certain smoothness assumptions). In \cite{khare-ramakrishna:liftingtorsion}, $\rho_N$  is first lifted to an ordinary, possibly non-geometric, representation by
   making the global, ordinary at places above $p$,
  deformation ring unobstructed.  As a second step, \cite{khare-ramakrishna:liftingtorsion} uses the
  smoothness of the global  ordinary deformation ring to massage the
  non-geometric lift to a geometric one (using the ideas of \cite{klr}). The relative deformation
  argument of this paper deals with the lack of linear disjointness
  alluded to above by a more direct route without relaxing
  the Selmer conditions at places above $p$.
  
  \item The analysis at the auxiliary (``trivial") primes in \cite{ramakrishna-hamblen} implicitly exploits the idea of using generic smoothness of generic fibers of local deformation rings. Additionally, G.~B\"{o}ckle has lifted odd irreducible representations $\br: \Gamma_\Q \rightarrow \mr{GL}_2(k)$
to geometric representations that are {\it finitely}  ramified at a finite auxiliary set in an unpublished manuscript that is of a few years' vintage. In this work he too anticipates using generic smoothness of generic fibers of local deformation rings, rather than finer information like  their formal smoothness.

\item In this work we use the vanishing of relative mod $p$ dual
  Selmer to prove  lifting results for residual Galois
  representations.  The vanishing of the  relative mod $p$ Selmer group 
   $\ov{H^1_{\mc{L}_M}(\Gamma_{F,{\mc{S}'} \cup {\mc{Q}}}, \rho_M(\fgder))}$ ensures infinitesimal uniqueness of  certain lifts of $\rho_n$ for $n \geq N$ with specified local properties at a finite set of places ${\mc{S}'} \cup {\mc{Q}}$.
This might be useful in proving automorphy of Galois representations.

 \item In an earlier version of this paper (\cite{fkp:arxivold},
   available at {\tt arXiv:1810.05803v1}), we could deal with the
   non-vanishing of $H^1(\im(\rho_2), \br(\fg^{\rm der}))$ (i.e., the
   lack of linear disjointness alluded to above) for $\rho_2$ a
   $W_2(k)$-valued lift of $\br$ (essentially) only when the adjoint
   representation $\br(\fgder)$ was multiplicity free as
   $\Fp[\gal{F}]$-module, with the help of a group-theoretic argument
   we owe to Larsen. This has been obviated in the present
   approach. The present paper supersedes the older version
   \cite{fkp:arxivold}.  The older version is no longer intended for
   publication in a journal, but will remain available on the
   arXiv. Some of the arguments of \cite{fkp:arxivold} might prove
   useful, for example in killing mod $p$ intrinsic (rather than
   relative) Selmer groups in certain situations.

\end{itemize}
\end{rmk}

Here is a summary of the contents of the paper. In \S \ref{trivialsection} and \S \ref{genericfibersection} we carry out the necessary preliminaries in local deformation theory: \S \ref{trivialsection} studies the local theory at our auxiliary primes, and \S \ref{genericfibersection} explains the consequences of results on generic fibers of local deformation rings.  We explain the relative deformation theory argument and prove the main result, Theorem \ref{mainthm}, in \S \ref{relativeliftsection}, using as input the technical mod $\vpi^{\n}$ lifting results of \S \ref{klrsection} and \S \ref{klr^Nsection}. In \S \ref{examples} we gather a few examples of the main theorem.  In  Appendix   \ref{groupsection} we present the group theory arguments needed to streamline some of the global hypotheses on $\br$ in \S \ref{klrsection} and \S \ref{klr^Nsection} to an irreducibility hypothesis.  In  Appendix  \ref{lazardsection} we deduce    results on cohomology of $p$-adic Lie groups  that are necessary  for the relative deformation theory argument.

Finally, we remark that a number of arguments are made technically
more intricate by the fact that we have worked with groups $G$ having
rather general (finite \'{e}tale, order prime to $p$) component groups. To get the essence of the argument nothing is lost if the reader  focuses on the case of connected adjoint groups $G$.

\subsection{Notation and conventions}\label{notation}
We embed local Galois groups into global Galois groups by fixing
embeddings $\overline{F} \into \overline{F}_v$. Certain assertions in this paper will depend on the particular choice of embedding: see Notation \ref{doublingnotation} for details and clarifications on this point. We write $\kappa$ \index[terms]{k@$\kappa$} for
the $p$-adic cyclotomic character and $\bar{\kappa}$ for its mod $p$
reduction, and we write $\Q_p/\Z_p(1)$ for the abelian group $\Q_p/\Z_p$ equipped with Galois module structure via $\kappa$. We once and for all fix an isomorphism (i.e. a compatible
collection of $p$-power roots of unity) $\zeta \colon \Q_p/\Z_p(1)
\xrightarrow{\sim} \mu_{p^{\infty}}(\overline{F})$, in turn identified with $
\mu_{p^\infty}(\overline{F}_v)$ when we have fixed $\ov{F} \into \ov{F}_v$, and use this to identify the Tate dual $V^*= \Hom(V,
  \mu_{p^{\infty}}(\overline{F}))$ of a $\gal{F}$-module $V$ with
  $\Hom(V, \Q_p/\Z_p(1))$. The reader should always assume we are
  doing this; only in the proof of Lemma \ref{localduality} will we
  make the identifications explicit. At some places $v$, we will need
  to analyze certain kinds of tamely-ramified deformations, and we
  will routinely write $\sigma_v$ and $\tau_v$ (or $\sigma$ and $\tau$ if there is no risk of confusion) for elements of $\Gal(F_v^{\mr{tame}, p}/F_v)$, $F_v^{\mr{tame}, p}$ denoting the maximal tamely-ramified with $p$-power ramification quotient of $\gal{F_v}$, that lift the (arithmetic) Frobenius and a generator of the $p$-part of the tame inertia group. See the beginning of \S \ref{trivialsection} for a discussion of normalization.

For any finite set of primes ${\mc{S}}$ of $F$, we let $\gal{F, {\mc{S}}}$ denote $\Gal(F({\mc{S}})/F)$, where $F({\mc{S}})$ is the maximal extension of $F$ inside $\overline{F}$ that is unramified outside the primes in ${\mc{S}}$; here we impose no constraint on the ``ramification" at $\infty$, but for notational convenience we do not want the set ${\mc{S}}$ to contain the archimedean places (as would often be the convention for what we are referring to as $\gal{F, {\mc{S}}}$).

Given a homomorphism $\rho \colon \Gamma \to H$ for some groups $\Gamma$ and $H$, and an $H$-module $V$, we will write $\rho(V)$ for the associated $\Gamma$-module (we typically apply this with $V$ the adjoint representation of an algebraic group). 

Let $\mc{O}$ be the ring of integers in a finite extension $E$
of $\Q_p$, let $\vpi$ be a uniformizer of $\mc{O}$, let
$m_{\mc{O}} = (\vpi)$ be the maximal ideal and let $k = \mc{O}/{m_{\mc{O}}}$ be
the residue field.  Let $e$ be the ramification index of
$\mc{O}/\Z_p$. For integers $0 < r \leq s$, we use the choice of
$\varpi$ to view $\mc{O}/\varpi^r$ as a submodule of
$\mc{O}/\varpi^s$ via the map induced by $1 \mapsto \vpi^{s-r}$.

Throughout this paper, except in Appendix \S \ref{lazardsection}, $G$
will be a smooth group scheme over $\mc{O}$ with Lie algebra
$\mathfrak{g}$ such that $G^0$ is split connected reductive, and
$G/G^0$ is finite \'{e}tale of order prime to $p$. We will sometimes
write $\pi_0(G)$ for this quotient $G/G^0$. We let $G^{\mr{der}}$
denote the derived group of $G^0$, and we denote by $\fgder$ the Lie
algebra of $G^{\mr{der}}$; in the sequel, when there is no chance of
confusion, we will sometimes abuse notation and use $\fgder$ also for
$\fgder \otimes_{\mc{O}}k$. We also let $Z_{G^0}$ and $\mf{z}_G$ be
the center of $G^0$ and its Lie algebra. More generally, for an
algebraic group $H$ we shall usually use $\mf{h}$ (i.e., the
corresponding lower case fraktur letter) to denote its Lie
algebra. 

For any complete local $\mc{O}$-algebra $\mc{O}'$ we set
$\wh{G}(\mc{O}') := \ker(G(\mc{O}') \to G(k'))$\index[terms]{G@$\wh{G}$}, where $k'$ is the
residue field of $\mc{O}'$, and we likewise set \index[terms]{G@$\wh{G^{\mr{der}}}$} $\wh{G^{\mr{der}}}(\mc{O}')= \ker(G^{\mr{der}}(\mc{O}') \to G^{\mr{der}}(k'))$.

For a continuous representation $\rho_s:\Gamma \to G(\mc{O}/\varpi^s)$
of a (topological) group $\Gamma$, we shall denote by $\rho_s(\fgder)$
the $\mc{O}$-module $\fgder \otimes_{\mc{O}}\mc{O}/\vpi^s$ with the
action of $\Gamma$ given by $\mr{Ad} \circ \rho_s$. Given moreover $r$ such that $0 <r \leq s$, we set $\rho_r:=  \rho_s \pmod {\vpi^r}$, and
the choice of uniformizer gives us inclusions $\rho_r(\fgder) \to \rho_s(\fgder)$
which we shall use without further mention.

\section{Deformation theory preliminaries}\label{defprelimsection}

\textbf{The following
  assumption on $p$ will implicitly be in effect for the remainder of
  the paper:}
\begin{assumption}\label{minimalp}
We assume that $p \neq 2$ is very good (\cite[\S 1.14]{carter:finitelie}) for $G^{\mr{der}}$, which in particular holds if $p \geq 7$ and $p \nmid n+1$ whenever $G^{\mr{der}}$ has a simple factor of type $A_n$. We also assume that the canonical central isogeny $G^{\mr{der}} \times Z_G^0 \to G^0$ has kernel of order prime to $p$ (and in particular is \'{e}tale). 
\end{assumption}
Then in particular we have a $G$-equivariant direct sum decomposition $\fg= \fgder \oplus \mf{z}_G$, $(\fgder)^{G^0}$ is zero, and there is a non-degenerate $G$-invariant trace form $\fgder \times \fgder \to k$ (\cite[1.16]{carter:finitelie}). The isogeny $G^{\mr{der}} \to G^{\mr{ad}}$ to the adjoint group of $G^{\mr{der}}$ also induces an isomorphism on Lie algebras. 

Let $\Gamma$ be a profinite group, and let $\br \colon \Gamma \to G(k)$ be a continuous homomorphism. Set 
\[
\bar{\mu}:= \br \pmod{G^{\mr{der}}} \colon \Gamma \to G/G^{\mr{der}}(k),
\] 
and fix a lift \index[terms]{M@$\mu$}  $\mu \colon \Gamma \to
G/G^{\mr{der}}(\mc{O})$ of \index[terms]{M@$\bar{\mu}$} $\bar{\mu}$;
we can always choose the ``Teichm\"{u}ller" lift, since
$G/G^{\mr{der}}(k)$ has order prime to $p$. Let $\mc{C}_{\mc{O}}$ be
the category of complete local noetherian algebras $R$ with $\mc{O}
\to R$ inducing an isomorphism of residue fields (and morphisms the
local homomorphisms), and let $\mc{C}_{\mc{O}}^f$
\index[terms]{C@$\mc{C}_{\mc{O}}^f$} be the full subcategory of those algebras that are artinian.  Note that for any $R \in \mc{C}_{\mc{O}}$, $\pi_0(G)(R) \xrightarrow{\sim} \pi_0(G)(k)$, so we will just identify any $\pi_0(G)(R)$ to this fixed finite group $\pi_0(G)(k)$. 

Define the lifting and deformation functors \index[terms]{L@$\Lift_{\br}$} \index[terms]{D@$\Def_{\br}$}
\index[terms]{L@$\Lift_{\br}^\mu$} \index[terms]{D@$\Def_{\br}^\mu$}
\[
\Lift_{\br}, \Def_{\br}, \Lift_{\br}^\mu, \Def_{\br}^\mu \colon \mc{C}_{\mc{O}} \to \Sets 
\]
by letting $\Lift_{\br}(R)$ be the set of lifts of $\br$ to $G(R)$, and by letting $\Lift_{\br}^\mu(R) \subset \Lift_{\br}(R)$ be the subset of lifts $\rho$ such that $\rho \pmod{G^{\mr{der}}}= \mu$; and then letting $\Def_{\br}(R)$ be the quotient of $\Lift_{\br}(R)$ by the equivalence relation
\[
\text{$\rho \sim \rho' \iff \rho= g \rho' g^{-1}$ for some $g \in \wh{G}(R)= \ker(G(R) \to G(k))$,} 
\]
and letting $\Def_{\br}^{\mu}(R)$ be the quotient of $\Lift_{\br}^{\mu}(R)$ by the equivalence relation
\[
\text{$\rho \sim \rho' \iff \rho= g \rho' g^{-1}$ for some $g \in \wh{G^{\mr{der}}}(R)= \ker(G^{\mr{der}}(R) \to G^{\mr{der}}(k))$.} 
\]
Note that $\wh{G}(R)= \wh{G^0}(R)$, and that conjugation by $\wh{G^{\mr{der}}}(R)$ preserves the property of having multiplier $\mu$. When $G=G^0$, $\wh{G}(R)$ and $\wh{G^{\mr{der}}}(R)$-conjugation induce the same equivalence relation on lifts since
$G^{\mr{der}} \cap Z_{G^0}$ has order prime to $p$. The tangent spaces of the lifting functors are canonically isomorphic to $Z^1(\Gamma, \br(\fg))$ (recall that $\br(\fg)$ is the $\Gamma$-module obtained by composing $\br$ with the adjoint representation of $G$) and $Z^1(\Gamma, \br(\fgder))$; the tangent spaces of the corresponding deformation functors are canonically isomorphic to $H^1(\Gamma, \br(\fg))$ and $H^1(\Gamma, \br(\fgder))$, and (by our running assumptions) the latter is a direct summand of the former. In some cases we will have a global Galois representation valued in a non-connected group $G$, but it will be convenient to develop certain local deformation conditions only for the group $G^0$: since $\wh{G}$ is contained in $G^0$ (and as above $\pi_0(G)$ has order prime to $p$), a $G^0$-deformation of a $G^0(k)$-valued $\br$ is exactly the same thing as a $G$-deformation of a $G^0$-valued $\br$. 

As usual, when $R \to R/I$ is a small extension the obstruction to lifting a $\rho \in \Lift_{\br}(R/I)$ to a $\tilde{\rho} \in \Lift_{\br}(R)$ is a class in $H^2(\Gamma, \br(\fg) \otimes_k I)$ (the two-cocyle one defines by choosing a topological lift of $\rho$ to
$G(R)$ takes values in $\ker(G(R) \to G(R/I))= \ker( G^0(R) \to G^0(R/I))= \exp(\fg \otimes_k I)$), and similarly for deforming lifts
of type $\mu$. Note also that when $\rho \in \Lift^\mu_{\br}(R/I)$ has a lift $\tilde{\rho} \in \Lift_{\br}(R)$, then it also has a lift in
$\Lift^\mu_{\br}(R)$: the discrepancy between $\tilde{\rho} \pmod{G^{\mr{der}}}$ and $\mu$ is measured by a class in $H^1(\Gamma, \bar{\mu}(\fg/\fgder)) \otimes_k I$, and the canonical map $H^1(\Gamma, \br(\mf{z}_G)) \to H^1(\Gamma, \bar{\mu}(\fg/\fgder))$ is an isomorphism (as $G^{\mr{der}} \cap Z_{G^0}$ has order prime to $p$). Thus we can modify the lift $\tilde{\rho}$ to one of type $\mu$.

\section{Local deformation theory: trivial primes}\label{trivialsection}
Let $F/\Q_{\ell}$ be a finite extension with residue field of order
$q$ (here 
$F$ denotes a local field, in contrast to the rest of the
paper). Assume $q$ is congruent to $1$ mod $p$, and let
$\br \colon \gal{F} \to G(k)$ be the trivial homomorphism; in
particular, all lifts of $\br$ land in $G^0$, and conjugation by $\wh{G}$ and $\wh{G^{\mr{der}}}$ 
define the same equivalence relation on lifts. Moreover, all lifts of
$\br$ factor through the quotient of $\gal{F}$ topologically generated
by a lift $\sigma$ of (arithmetic) Frobenius and a generator $\tau$ of
the $p$-part of the tame inertia group. At one point we will invoke a
calculation (Lemma \ref{localduality}) that depends on the choice of
$\tau$. Suppose that $F$ in fact contains $\mu_{p^b}$ for some integer
$b$, and that we have a fixed $(p^b)^{th}$ root of unity
\index[terms]{Z@$\zeta$} $\zeta \in \mu_{p^b}(F)$, which we may regard
(by evaluation at $\frac{1}{p^b}$) as a choice of isomorphism
$\zeta \colon \Z/p^b(1) \xrightarrow{\sim} \mu_{p^b}(F)$ (in the
global setting, this will come from a global choice, as in \S
\ref{notation}). We then choose $\tau$ \index[terms]{T@$\tau$} such
that for any uniformizer $\varpi_F$ of $F$,
$\frac{\tau(\varpi_F^{1/p^b})}{\varpi_F^{1/p^b}}= \zeta$ (which is
possible since $\mu_{p^b} \subset F$ implies that for any $a \in \Z$,
$\frac{\tau^a(\vpi_F^{1/p^b})}{\vpi_F^{1/p^b}}= \left(
  \frac{\tau(\vpi_F^{1/p^b})}{\vpi_F^{1/p^b}}\right)^a$).

\subsection{The local conditions at trivial primes} \label{s:loctriv}
We will now define the kinds of local lifts of $\br$ that we will make
use of at auxiliary primes. First, we introduce some notation. For a
split maximal torus $T$ \index[terms]{T@$T$} of $G^0$ (over $\mc{O}$)
and an $\alpha \in \Phi(G^0, T)$, \index[terms]{P@$\Phi(G^0, T)$} we
let \index[terms]{U@$U_{\alpha}$} $U_{\alpha} \subset G^0$ denote the
root subgroup that is the image of the root homomorphism
(``exponential mapping") \index[terms]{U@$u_{\alpha}$} $u_{\alpha} \colon \mf{g}_{\alpha} \to G$. The homomorphism $u_{\alpha}$ is a $T$-equivariant isomorphism $\mf{g}_{\alpha} \to U_{\alpha}$ (see \cite[Theorem 4.1.4]{conrad:luminy}), and its characterizing properties (\textit{loc.~cit.}) imply that $Z_{G^0}(\mf{g}_{\alpha})= Z_{G^0}(U_{\alpha})$.
\begin{defn}\label{triviallifts}
Fix a split maximal torus $T$ of $G^0$ (over $\mc{O}$) and an $\alpha \in \Phi(G^0, T)$. Define $\Lift^{\alpha}_{\br}(R)$\index[terms]{L@$\Lift^{\alpha}_{\br}$} to be the set of lifts $\widehat{G}(R)$-conjugate to one satisfying 
\begin{itemize}
\item $\rho(\sigma) \in T \cdot Z_{G^0}(\mf{g}_{\alpha})(R)$
\item Under the composite (note that $T$ normalizes the centralizer) 
\[
T \cdot Z_{G^0}(\mf{g}_{\alpha})(R) \to T(R)/(T(R) \cap Z_{G^0}(\mf{g}_{\alpha})(R)) \xrightarrow{\alpha} R^\times,
\] 
$\rho(\sigma)$ maps to $q$.
\item $\rho(\tau) \in U_{\alpha}(R)$.
\end{itemize}
\end{defn}
\begin{lemma}\label{trivsmooth}
 For any pair $(T, \alpha)$ consisting of a split maximal torus $T$ of $G^0$ and an $\alpha \in \Phi(G^0, T)$, the functor $\Lift^{\alpha}_{\br}$ is formally smooth, i.e. for all maps $R \to R/I$ in $\mc{C}^f_\mc{O}$ with $I \cdot \mf{m}_R=0$, $\Lift_{\br}^{\alpha}(R) \to \Lift^{\alpha}_{\br}(R/I)$ is surjective.
 
 Similarly, if we fix a lift $\mu \colon \gal{F} \to G/G^{\mr{der}}(\mc{O})$ of the multiplier character $\bar{\mu}:= \br \pmod{G^{\mr{der}}}$, then the sub-functor of lifts with multiplier $\mu$, $\Lift^{\mu, \alpha}_{\br}$\index[terms]{L@$\Lift^{\mu, \alpha}_{\br}$}, is formally smooth.
\end{lemma}
\proof
It is convenient to begin with a slightly different description of $\Lift_{\br}^{\alpha}$ that will circumvent the need to know that $Z_{G^0}(\mf{g}_{\alpha})$ is smooth over $\mc{O}$. To that end, let $Z_{\alpha}$ be the open subscheme of $Z_{G^0}(\mf{g}_{\alpha})$ obtained by removing all non-identity components of the special fiber $Z_{G^0_k}(\mf{g}_{\alpha} \otimes_{\mc{O}} k)$. Set $\mf{g}_{\alpha, k}=\mf{g}_{\alpha} \otimes_{\mc{O}} k$. We first claim the special fiber $Z_{\alpha, k} \to \Spec k$ is smooth. By our assumptions on $p$, $Z_{G^0_k}(\mf{g}_{\alpha, k})$ is smooth if and only if $Z_{G^{\mr{der}}_k}(\mf{g}_{\alpha, k})$ is smooth, and then the assumption that $p$ is very good for $G^{\mr{der}}$ implies, by a criterion of Richardson (\cite[Theorem 2.5]{jantzen:nilporbits}), that $Z_{G^{\mr{der}}_k}(\mf{g}_{\alpha, k})$ is smooth (recall that $\fgder$ has a non-degenerate trace form). In particular, $Z_{\alpha, k}$ is smooth. Since $Z_{\alpha, k}$ has a single irreducible component, we can now apply \cite[Remark 4.3, Lemma 4.4]{booher:minimal} to deduce that $Z_{\alpha} \to \Spec \mc{O}$ is smooth.

We next claim that $\Lift^{\alpha}_{\br}$ is equivalently defined by replacing $Z_{G^0}(\mf{g}_{\alpha})$ with $Z_{\alpha}$ in Definition \ref{triviallifts}. First note that for any object $R$ of $\mc{C}^f_{\mc{O}}$, the fiber over the identity of $Z_{G^0}(\mf{g}_{\alpha})(R) \to Z_{G^0}(\mf{g}_{\alpha})(k)$ is contained in $Z_{\alpha}(R)$, and that $T$ normalizes $Z_{\alpha}$ (as functors of Artin rings). Now let $x \in T(R)Z_{G^0}(\mf{g}_{\alpha})(R)$ be an element in the fiber over $1 \in G(k)$, and correspondingly write $x=t\cdot c$. Writing $\bar{c}$ for the image of $c$ in $G(k)$, we have $\bar{c} \in \ker(\alpha|_T)$. This kernel is smooth (our assumptions on $p$ imply that $X^{\bullet}(T)/\Z \alpha$ has no $p$-torsion), so we can lift $\bar{c}$ to an element $t' \in \ker(\alpha|_T)(R)$. Then writing $x=(tt')({t'}^{-1}c)$ we have exhibited $x$ as an element of $T(R) Z_{\alpha}(R)$. Since $\br$ is trivial, we conclude that $\Lift^{\alpha}_{\br}$ can equivalently be defined with $Z_{\alpha}$ in place of $Z_{G^0}(\mf{g}_{\alpha})$.

With this reinterpretation, we can now check formal smoothness of $\Lift_{\br}^{\alpha}$. Let $\rho$ be any element of $\Lift^{\alpha}_{\br}(R/I)$. Since $\widehat{G}$ is formally smooth, we may assume $\rho$ satisfies the three bulleted items of Definition \ref{triviallifts}. Write $\rho(\sigma)=t_{\sigma}c_{\sigma}$ and $\rho(\tau)=u_{\alpha}(x)$ for some $t_{\sigma} \in T(R/I)$ satisfying $\alpha(t_{\sigma})=q$, $c_{\sigma} \in Z_{\alpha}(R/I)$, and $x \in R/I$. Since $T$ and $Z_{\alpha}$ are formally smooth, we can choose lifts $\wt{t}_{\sigma} \in T(R)$, $\wt{c}_{\sigma} \in Z_{\alpha}(R)$, and $\wt{x} \in R$. We can write $\alpha(\wt{t}_{\sigma})=q+i$ for some $i \in I$, and then we replace $\wt{t}_{\sigma}$ by $\wt{t}_{\sigma} \alpha^\vee(1-\frac{i}{2})$ (recall that $p$ is odd). Since $I \cdot \mf{m}_R=0$, we then find that $\tilde{\rho}(\sigma)=\wt{t}_{\sigma} \wt{c}_{\sigma}$, $\tilde{\rho}(\tau)=u_{\alpha}(\wt{x})$ defines a lift $\tilde{\rho} \in \Lift^{\alpha}_{\br}(R)$ of $\rho$. The fixed multiplier analogue is clear from the remarks in \S \ref{defprelimsection}.
\endproof

\begin{rmk}
We could have argued directly with the original definition using $Z_{G^0}(\mf{g}_{\alpha})$, but lacking a generalization to all groups of \cite[\S 4.4]{booher:minimal}---namely, sections of $Z_{G^0}(\mf{g}_{\alpha})$ hitting any irreducible component in the special fiber---we would only have obtained the smoothness for $p \gg_G 0$ but non-effective (resorting to a spreading-out argument).
\end{rmk}

We now carry out the local calculation needed for Theorem \ref{klr^N}. In the application, we will only need to make use of the behavior of a local deformation functor beyond a certain fixed lift modulo $\vpi^n$ (for some $n$), so we begin by introducing a relative analogue of the functor $\Lift^{\alpha}_{\br}$:
\begin{defn}\label{relativetriv}
Suppose $F/\Q_{\ell}$ is a finite extension with residue field of order $q \equiv 1 \pmod p$, let $M \geq 1$ be a fixed integer, and suppose $\rho_M \colon \gal{F} \to G(\mc{O}/\vpi^M)$ is a homomorphism whose mod $p$ reduction $\br$ is trivial and that belongs to $\Lift_{\br}^{\alpha}(\mc{O}/\vpi^M)$ for some pair $(T, \alpha)$ of a split maximal torus and a root. Then we define the following functor on the over-category $(\mc{C}^f_{\mc{O}})_{/{(\mc{O}/\vpi^M)}}$ of objects of $\mc{C}^f_{\mc{O}}$ equipped with an augmentation to $\mc{O}/\vpi^M$: given an $R \to \mc{O}/\vpi^M$ (we will for notational ease not write the augmentation in what follows), define $\Lift^{\alpha}_{\rho_M}(R)$\index[terms]{L@$\Lift^{\alpha}_{\rho_M}$} to be the set of lifts 
\[
\xymatrix{
& G(R) \ar[d] \\
\gal{F} \ar[r]_-{\rho_M} \ar[ur]^{\rho} & G(\mc{O}/\vpi^M)
}
\]
of $\rho_M$ that are \index[terms]{G@$\wh{G}^{(M)}$} $\wh{G}^{(M)}(R):= \ker\left(G(R) \to G(\mc{O}/\vpi^M)\right)$-conjugate, or equivalently \index[terms]{G@$(\wh{G^{\mr{der}}})^{(M)}$} $(\wh{G^{\mr{der}}})^{(M)}(R):= \ker \left(G^{\mr{der}}(R) \to G^{\mr{der}}(\mc{O}/\vpi^M)\right)$-conjugate to one satisfying
\begin{itemize}
\item $\rho(\sigma) \in T \cdot Z_{G^0}(\mf{g}_{\alpha})(R)$
\item Under the composite  
\[
T \cdot Z_{G^0}(\mf{g}_{\alpha})(R) \to T(R)/(T(R) \cap Z_{G^0}(\mf{g}_{\alpha})(R)) \xrightarrow{\alpha} R^\times,
\] 
$\rho(\sigma)$ maps to $q$.
\item $\rho(\tau) \in U_{\alpha}(R)$.
\end{itemize}
Similarly define $\Lift^{\mu, \alpha}_{\rho_M}$\index[terms]{L@$\Lift^{\mu, \alpha}_{\rho_M}$} to be the sub-functor of lifts with prescribed multiplier $\mu$.
\end{defn}
Note that $\wh{G}^{(M)}$ and $(\wh{G^{\mr{der}}})^{(M)}$ are formally smooth, so $\Lift^{\alpha}_{\rho_M}$ and $\Lift^{\mu, \alpha}_{\rho_M}$ are formally smooth, just as in Lemma \ref{trivsmooth}. When $M=1$, we just recover $\Lift_{\br}^{\alpha}$ and $\Lift_{\br}^{\mu, \alpha}$.

We will need two lemmas about the finer structure of $\Lift_{\br}^{\mu, \alpha}$; both express the fact that the lifting functor behaves ``as if formally smooth" modulo $\vpi^n$ for large enough $n$. Recall that for any $G(\mc{O}/\vpi^r)$-valued homomorphism $\rho_r$, $\rho_r(\fgder)$ will denote $\fgder \otimes_{\mc{O}} \mc{O}/\vpi^r$ equipped with the $\Ad \circ \rho_r$ action.
\begin{lemma}\label{extracocycles^N}
Let $F/\Q_{\ell}$ be a finite extension with residue field of order $q$, let $M \geq 1$ be a fixed integer, and let $1 \leq s \leq M$ be another fixed integer. Suppose $\rho_{M+s} \colon \gal{F} \to G(\mc{O}/\varpi^{M+s})$ is a homomorphism with multiplier $\mu$ satisfying:
\begin{itemize}
\item The reduction $\rho_s := \rho_{M+s} \pmod{\vpi^s}$ is trivial (mod center), and $q \equiv 1 \pmod{\vpi^s}$; but $q \not \equiv 1 \pmod{\vpi^{s+1}}$.
\item There is a suitable choice of split maximal torus $T$ and root $\alpha \in \Phi(G^0, T)$ such that $\rho_{M+s}(\sigma) \in T(\mc{O}/\vpi^{M+s})$, $\alpha(\rho_{M+s}(\sigma))= q$, and $\rho_{M+s}(\tau) \in U_{\alpha}(\mc{O}/\vpi^{M+s})$. In particular, $\rho_{M+s} \in \Lift^{\mu, \alpha}_{\rho_M}(\mc{O}/\vpi^{M+s})$. 
\item For any root $\beta \in \Phi(G^0, T)$, $\beta(\rho_{M+s}(\sigma)) \not \equiv 1 \pmod{\vpi^{s+1}}$.
\end{itemize}
Then for all $1 \leq r \leq M$ there are spaces of cocycles
$Z^{\alpha}_r \subset Z^1(\gal{F}, \rho_r(\fgder))$, with images $L_r^{\alpha} \subset H^1(\gal{F}, \rho_r(\fgder))$ such that
\begin{itemize}
\item $Z_r^{\alpha}$ contains all coboundaries and is free over $\mc{O}/\vpi^{r}$ of rank $\dim(\fgder)$.
\item For any integers $a, b>0$ such that $a+b=r$, the natural maps induce short exact sequences
\[
0 \to Z_a^{\alpha} \to Z_r^{\alpha} \to Z_b^{\alpha} \to 0.
\]
\item For all $m \geq 2s+M$ (in particular, for all all $m \geq 3M$), and any lift $\rho_{m} \in \Lift^{\mu, \alpha}_{\rho_M}(\mc{O}/\vpi^m)$ of $\rho_{M+s}$, the fiber of $\Lift^{\mu, \alpha}_{\rho_M}(\mc{O}/\vpi^{m+r}) \to \Lift^{\mu, \alpha}_{\rho_M}(\mc{O}/\vpi^m)$ over $\rho_m$ is non-empty and $Z_r^{\alpha}$-stable.
\end{itemize}
\end{lemma}
\proof
Let $M, r, s, m$ be as in the statement. The fiber of $\Lift^{\mu, \alpha}_{\rho_M}(\mc{O}/\vpi^{m+r}) \to \Lift^{\mu, \alpha}_{\rho_M}(\mc{O}/\vpi^m)$ over a fixed lift $\rho_m$ is non-empty by the previously-indicated formal smoothness, so consider a mod $\vpi^{m+r}$ lift $\rho_{m+r}$ in this fiber. We will construct the spaces of cocycles $Z^{\alpha}_r$ by analyzing other elements of this fiber, and it will be evident from the construction that $Z^{\alpha}_r$ depends only on $\rho_{r+s}$ (not on $m$).

Note that $2m-2s \geq m+r$, so for any $X \in \fg$ we have the following computation in $G(\mc{O}/\vpi^{m+r})$:
\begin{align*}
(1+\vpi^{m-s}X)\rho_{m+r}(\gamma)(1+\vpi^{m-s}X)^{-1}&=
                                                       (1+\vpi^{m-s}X)\rho_{m+r}(\gamma)(1-\vpi^{m-s}X)
  \\&=(1+\vpi^{m-s}(X-\Ad(\rho_{m+r}(\gamma))X))\rho_{m+r}(\gamma)\\&=
  \left(1+\vpi^{m} \left
  (\frac{X-\Ad(\rho_{m+r}(\gamma))X}{\vpi^s}\right) \right ) \rho_{m+r}(\gamma),
\end{align*}
where this last expression makes sense since $\rho$ is trivial modulo $\vpi^s$. Moreover, the expression $\frac{X-\Ad(\rho_{m+r}(\gamma))X}{\vpi^s} \pmod{\varpi^r}$ only depends on $\rho_{r+s}$, and in particular only on $\rho_{M+s}$. Thus, for each $X \in \fgder$, we have a cocycle $\phi_X^r \in Z^1(\gal{F}, \rho_r(\fgder))$ given by
\[
\phi_X^r(\gamma)= \frac{X-\Ad(\rho_{m+r}(\gamma))X}{\vpi^s},
\]
and the action of any such $\phi_X^r$ preserves the fiber of $\Lift^{\alpha}_{\rho_M}(\mc{O}/\vpi^{m+r}) \to \Lift^\alpha_{\rho_M}(\mc{O}/\vpi^m)$ over $\rho_m$. 
In particular, for any root $\beta \in \Phi(G^0, T)$, taking $X=X_{\beta}$ a generator (over $\mc{O}$) of $\fg_{\beta}$, the cocycle $\phi^r_{X_{\beta}}$ has (by the third bulleted hypothesis) the property that $\phi^r_{X_\beta}(\sigma)$ is an $\mc{O}^\times$-multiple of $X_{\beta}$. Our space $Z_r^{\alpha}$ is defined as the span of the following three kinds of cocycles:
\begin{itemize}
\item $\phi_{X_\beta}^r$ for all $\beta \in \Phi(G^0, T)$.
\item For all $X$ in an $\mc{O}$-basis of $\ker(\alpha|_{\mf{t}^{\mr{der}}})$, the unramified cocycles $\phi_X^{\mr{un}, r}$ given by $\phi_X^{\mr{un}, r}(\sigma)=X$, $\phi_X^{\mr{un}, r}(\tau)=0$. To see that these are indeed cocycles, we just note that $X \pmod{\vpi^r}$ belongs to $\rho_r(\fg)^{\gal{F}}$, which reduces to the two assumptions that $\rho_r(\sigma) \in T(\mc{O}/\vpi^r)$ and $[X, X_{\alpha}]=0$.
\item The ramified cocycle $\phi_{\alpha}^r$ given by $\phi_{\alpha}^r(\sigma)=0$ and $\phi_{\alpha}^r(\tau)=X_{\alpha}$. The cocycle condition for $\phi_{\alpha}^r$ reduces to the fact that $\alpha(\rho_r(\sigma))= q$.
\end{itemize}
That $\phi_X^r$ and $\phi_{\alpha}^r$ preserve the fiber is clear when acting on a lift $\rho_{m+r}$ that satisfies $\rho_{m+r}(\sigma) \in T \cdot Z_{G^0}(\fg_{\alpha})(\mc{O}/\vpi^{m+r})$ and $\rho_{m+r}(\tau) \in U_{\alpha}(\mc{O}/\vpi^{m+r})$; in general it follows since for $g \in (\wh{G^{\mr{der}}})^{(M)}(\mc{O}/\vpi^{m+r})$
\[
g(1+\vpi^m \phi)\rho_{m+r}g^{-1}= (1+\vpi^m \phi) g \rho_{m+r} g^{-1},
\] 
as $r \leq M$.

We claim that the $\mc{O}/\vpi^r$-span $Z_r^{\alpha}$ of the collection $\{\phi_{X_\beta}^r, \phi_X^r, \phi_{\alpha}^r\}$ satisfies all the properties in the Lemma's conclusion. To see that $Z_r^{\alpha}$ is free of rank $\dim(\fgder)$, note that a linear combination of the $\phi^r_{X_\beta}, \phi_X^{\mr{un}, r}, \phi_\alpha^r$ is---by evaluating at $\sigma$ and then at $\tau$---seen to be a multiple of $\varpi$ if and only if each coefficient is a multiple of $\vpi$. The claimed exact sequences induced by reduction modulo $\vpi^b$ are clear from the construction.

Finally, to see that $Z_r^{\alpha}$ contains
$B^1(\gal{F},
\rho_r(\fgder))$, 
note that the latter module is spanned by
$\gamma \mapsto X-\Ad(\rho_r(\gamma))X$ for $X$ in a basis of
$\rho_r(\fgder)$. The coboundaries thus generated by the
$\{X_\beta\}_{\beta \in \Phi(G^0, T)}$ are clearly in the span of the
$\phi_{X_\beta}^r$. For $X \in \mf{t}^{\mr{der}}$, the corresponding
coboundary vanishes on $\sigma$ and maps $\tau$ to
$X-\Ad(\rho_r(\tau))X$. Since $\rho_r(\tau)=u_{\alpha}(y)$ for some
$y$, $X-\Ad(\rho_r(\tau))X$ is a multiple of $X_{\alpha}$, and so this
coboundary is in the span of $\phi^r_{\alpha}$.  \endproof

\begin{rmk} $ $
\begin{itemize}

\item Theorem \ref{klr^N} will use  Lemma \ref{extracocycles^N}  in the cases $s \in \{1, 2, e\}$.

\item Proposition \ref{prop:killrel} and Theorem \ref{thm:killrel} will use Lemma \ref{extracocycles^N} in the case $s=M$. More precisely, in Lemma \ref{lem:triv2} we record some additional properties satisfied by the spaces $Z^{\alpha}_r$ under somewhat stronger hypotheses, and we apply this refinement in Proposition \ref{prop:killrel} and Theorem \ref{thm:killrel}.
\end{itemize}

\end{rmk}
The following lemma is a softer version of the previous one, and we will apply it in a slightly different setting; we will use it in combination with the results of \S \ref{genericfibersection}, namely Proposition \ref{prop:gensmooth}. For a residual representation $\br \colon \gal{F} \to G(k)$, multiplier character $\mu$, and inertial type $\tau \colon I_F \to G(E)$, we let $R_{\br}^{\square, \mu, \tau}[1/\vpi]$ denote the quotient of the lifting ring $R_{\br}^{\square, \mu}$ parametrizing lifts with inertial type $\tau$.
\begin{lemma}\label{auxlift}
Let $F/\Q_{\ell}$ be a finite extension with residue field of order $q \equiv 1 \pmod p$. Let $n \geq 2$ be an integer, and let $\rho_n \colon \gal{F} \to T \cdot U_{\alpha}(\mc{O}/\vpi^n)$ be a homomorphism with multiplier $\mu$ lifting the trivial representation $\br$ and satisfying $\rho_n(\sigma)=t_n \in T(\mc{O}/\vpi^n)$, $\rho_n(\tau)=u_{\alpha}(y_n) \in U_{\alpha}(\mc{O}/\vpi^n)$, and $\alpha(t_n) \equiv q \pmod{\vpi^n}$ (in particular, $\rho_n \in \Lift_{\br}^{\mu, \alpha}(\mc{O}/\vpi^n)$). Then there is some $\rho \in \Lift_{\rho_n}^{\mu, \alpha}(\mc{O})$ defining a formally smooth point of the generic fiber $R^{\square, \mu, \tau_0}_{\br}[1/\vpi]$, where $\tau_0$ denotes the trivial inertial type.
\end{lemma}
\begin{proof}
Choose any lift $t \in T(\mc{O})$ of $t_n$ such that $\alpha(t)=q$,
$\beta(t) \neq q$ for any root $\beta \neq \alpha$, and the image of
$t$ in $G/G^{\mr{der}}$ is $\mu(\sigma)$. Since $p$ is odd and $T$ is
formally smooth, it is easy to see as in the proof of Lemma
\ref{trivsmooth} (or see also Lemma \ref{generalposition}) that such a
$t$ exists. Let $y \in \mc{O}$ be any lift of $y_n$, and requiring
that $y \neq 0$. Then $\rho(\sigma)=t$, $\rho(\tau)=u_{\alpha}(y)$
defines an element of $\Lift_{\rho_n}^{\mu, \alpha}(\mc{O})$, and
clearly the associated inertial type is trivial. To check that the
point of $R_{\br}^{\square, \mu, \tau_0}[1/\vpi]$ corresponding to
$\rho$ is formally smooth, we use the criterion of \cite[Corollary
3.3.4, Remark 3.3.6]{Bellovin-Gee-G}: it suffices to check that
$H^0(\gal{F}, \rho(\fgder)(1))=0$ 
(to take the multiplier into account, we use $\fgder$ in place of
$\fg$ and then we use the Killing form to identify $\rho(\fgder)$
with its dual).
Suppose $X$ belongs to $(\rho(\fgder)(1))^{\gal{F}}$, and decompose
$X= Z+ \sum_{\beta} X_{\beta}$, where $Z \in \mf{t}^{\mr{der}}(1)$ and
$X_{\beta} \in \fg_{\beta}(1)$. Then
$\sigma \cdot X= qZ+\sum_{\beta} q\beta(t)X_{\beta}$, so we must have
$Z=0$ and $q\beta(t)=1$ for any $\beta$ such that $X_{\beta} \neq
0$. By construction, the only $\beta \in \Phi(G^0, T)$ such that
$\beta(t)=q^{-1}$ is $\beta=-\alpha$, so $X= X_{-\alpha}$. Then
$\tau \cdot X=X$ implies (since $y \neq 0$) that $X=0$.
\end{proof}
\begin{rmk}\label{auxliftcoeff}
We emphasize that the formally smooth point $\rho$ is defined over $\mc{O}$ itself. In our global argument, we apply this lemma at the first $e$ stages (that is, up to $G(\mc{O}/\vpi^e)$) of our lifting argument, but not at subsequent stages, and so it is important that the lifts produced by this lemma not force us to enlarge $e$.
\end{rmk}

\subsection{Local duality pairing at trivial primes}
We will later require the following calculation of the local duality pairing at trivial primes: for $r=1$ we will use this Lemma throughout \S \ref{doubling}, and for general $r$ we will use it in Lemma \ref{lem:triv2}. When $r=1$ (or analogously for general $r$ when $e=1$), the trace pairing and choice of generator $\zeta$ of $\mu_p$ induces, for any $k[\gal{F}]$-module $W$, an isomorphism of $k[\gal{F}]$-modules
\[
\tr_{k/\Fp} \colon \Hom_k(W, k(1)) \xrightarrow{\sim} \Hom_{\Fp}(W, \Fp(1)) \xrightarrow[\zeta]{\sim}\Hom_{\Fp}(W, \mu_p)= W^*,
\]
where the $k$-structure on the target is just induced by the $k$-multiplication on $W$. In general, we simply fix a generator of the $\mc{O}$-module $\Hom_{\Z_p}(\mc{O}/\vpi^r, \Q_p/\Z_p)$; then for a finite free $\mc{O}/\vpi^r$-module $W$, we obtain an isomorphism
\begin{equation}\label{generator}
 \Hom_{\mc{O}}(W, \mc{O}/\vpi^r) \xrightarrow{\sim} \Hom(W, \Q_p/\Z_p)
\end{equation}
by composing with our fixed generator. If now $W$ is moreover an $\mc{O}/\vpi^r[\gal{F}]$-module, then having fixed a choice $\zeta \colon \Q_p/\Z_p(1) \xrightarrow{\sim} \mu_{p^\infty}$ of $p$-power roots of unity, we can identify the Tate dual $W^*$ with $\Hom_{\mc{O}}(W, \mc{O}/\vpi^r(1))$, and we then define the $\mc{O}$-linear local duality by
\[
\inv_F( \cdot \cup \cdot ) \colon H^1(\gal{F}, W) \times H^1(\gal{F}, W^*) \to H^2(\gal{F}, W \otimes_{\mc{O}} W^*) \to H^2(\gal{F}, \mc{O}/\vpi^r(1)) \xrightarrow{\sim} \mc{O}/\vpi^r,
\]
where the last isomorphism is induced by the composite
\[
H^2(\gal{F}, \mc{O}/\vpi^r(1)) \to H^2(\gal{F}, E/\mc{O}(1)) \cong H^2(\gal{F}, \Q_p/\Z_p(1)) \otimes_{\Z_p} \mc{O} \xrightarrow[\inv_F \circ \zeta^{-1}]{\sim} \Q_p/\Z_p \otimes_{\Z_p} \mc{O} \cong E/\mc{O}.
\]
In all, this duality pairing depends on the choice in Equation (\ref{generator}), but only up to $\mc{O}^\times$-scaling (so, e.g., $\mc{O}$-submodules are canonically defined), and it is independent of the choice of $\zeta$.
\begin{lemma}\label{localduality}
Let $W$ be a free $\mc{O}/\vpi^r$-module equipped with trivial $\gal{F}$-action, and assume that $q \equiv 1 \pmod{\vpi^r}$, so that $W^*$ is also a trivial $\gal{F}$-module. Identify $W^* \cong \Hom_{\mc{O}}(W, \mc{O}/\vpi^r(1))$ as above, and write $\langle \cdot, \cdot \rangle \colon W \times W^* \to \mc{O}/\vpi^r$ for the $\mc{O}/\vpi^r$-linear evaluation pairing. Then the $\mc{O}/\vpi^r$-linear duality pairing
\[
\inv_F(\cdot \cup \cdot) \colon H^1(\gal{F}, W) \times H^1(\gal{F}, W^*) \to \mc{O}/\vpi^r
\]
has the following properties: if $\phi$ is unramified, then
\[
\inv_F(\phi \cup \psi)=- \langle \phi(\sigma), \psi(\tau) \rangle,
\]
and if $\psi$ is unramified, then
\[
\inv_F(\phi \cup \psi)= \langle \phi(\tau), \psi(\sigma) \rangle.
\]
\end{lemma}
\begin{rmk}
These identifications of course depend on the choice of $\tau$; see the beginning of this section for the discussion of how we calibrate $\tau$.
\end{rmk}
\proof
Since $q$ is an integer, $q \equiv 1 \pmod{\vpi^r}$ implies $q \equiv 1 \pmod{\vpi^{e \cdot \lceil \frac{r}{e} \rceil}}$. Since $W$ is trivial, the lemma reduces to the case where $W$ is free of rank one over $\mc{O}/\vpi^r$, and the above description of the $\mc{O}/\vpi^r$-linear duality pairing, which for $W=\mc{O}/\vpi^r$ is the $\mc{O}/\vpi^r$-linear extension of the $\Z/p^b$-linear duality pairing on the trivial module $\Z/p^b$, where $b= \lceil \frac{r}{e} \rceil$, shows we can further reduce to the case $W= \Z/p^b$. (We use here the above observation that $q \equiv 1 \pmod{p^b}$, seemingly slighter stronger than our assumption.)

Then the calculation can be performed, for instance, using the identity
\[
 \inv_F(\phi \cup \delta(a))= \phi(\rec_F(a))
\]
for any $\phi \in H^1(\gal{F}, W)= \Hom(\gal{F}^{\mr{ab}}, \Z/p^b)$
and
$a \in F^{\times}/(F^{\times})^{p^b} \xrightarrow[\sim]{\delta}
H^1(\gal{F}, \mu_{p^b})= H^1(\gal{F}, (\Z/p^b)^*)$ (the last
identification is the canonical one). If $\phi$ is unramified, then
$\phi(\rec_F(a))$ is simply $-v(a)\phi(\sigma)$ (writing $v$ for the
normalized valuation, and normalizing $\rec_F$ to take uniformizers to
geometric frobenii). On the other hand, if $\psi= \delta(a)$, then it
follows from our choice of $\tau$ at the beginning of
\S\ref{s:loctriv} that
\[
\psi(\tau)= \delta(a)(\tau)=\frac{\tau(a^{1/p^b})}{a^{1/p^b}}= \left(\frac{\tau(\varpi_F^{1/p^b})}{\varpi_F^{1/p^b}}\right)^{v(a)}= \zeta^{v(a)}.
\]
Then $\langle \phi(\sigma), \psi(\tau)\rangle= \zeta^{v(a)\phi(\sigma)}$, and via our isomorphism $\zeta \colon \Z/p^b \to \mu_{p^b}$ we thus identify $\langle \phi(\sigma), \psi(\tau)\rangle= -\inv_F(\phi \cup \psi)$, as desired. Now suppose $\psi$ is unramified. Then we identify $W=W^{**}$ and apply the previous step to find 
\[
\inv_F(\phi \cup \psi)=-\inv_F(\psi \cup \phi)= \langle \psi(\sigma), \phi(\tau) \rangle= \langle \phi(\tau), \psi(\sigma) \rangle.
\]
\endproof

\section{Generic fibers of local Galois deformation rings}\label{genericfibersection}
In this section we re-interpret known results (\cite{kisin:pst}, \cite{Bellovin-Gee-G}) on generic fibers of local (and fixed $p$-adic Hodge type) deformation rings. We will achieve enough control over a part of the corresponding integral deformation rings---as in the previous section, once we have lifted beyond a certain modulus $\vpi^M$---to use these softer results on generic fibers in place of the usual local demand of Ramakrishna-style arguments, which typically require having at hand a formally smooth irreducible component of the appropriate local deformation ring.

\subsection{Interpretation of a result of Serre}

Let $E$, $\mc{O}$, $\vpi$, ${m_{\mc{O}}}$, $k$ be as before.  Let $(R, m_R)$
be a complete local noetherian $\mc{O}$-algebra with residue field
$k$. In this section, we are interested in understanding the structure
of the sets $X_n$ of $\mc{O}/\varpi^n$-valued points of $\Spec(R)$
together with the reduction maps $\pi_{n,r}: X_{n+r} \to X_n$ for
$n, r \geq 0$.

Let $\Omega_{R/\mc{O}}$ \index[terms]{O@$\Omega_{R/\mc{O}}$} denote the module of continuous
differentials\footnote{$\Omega_{R/\mc{O}}$ represents the functor
  $\mr{Der}_{\mc{O}}^{\mr{cont}}(R, -)$ of continuous
  $\mc{O}$-linear derivations into finite $R$ modules. For $A =
  \mc{O}[[x_1,x_2,\dots,x_D]]$, $\Omega_{A/\mc{O}}$ is a free $A$-module with basis $\{dx_i\}_{i=1}^D$.}. For any $n$
and $r \leq n$, there is an action of
$\Hom_{\mc{O}}(\Omega_{R/\mc{O}} \otimes_R \mc{O}/{\varpi^n},
\mc{O}/\varpi^r)$ on the fiber of $\pi_{n,r}$ over $x_n \in X_n$,
where the $R$-module structure on $\mc{O}/\varpi^n$ used to form the
$\Hom$ corresponds to the chosen point $x_n \in X_n$ (viewed as a
closed subscheme of $\Spec(R)$), and the action is given as follows:
An element $x_{n+r} \in X_{n+r}$ is an $\mc{O}$-algebra homomorphism
$f_{n+r}:R \to \mc{O}/\varpi^{n+r}$ and an element
$d_r \in \Hom_{\mc{O}}(\Omega_{R/\mc{O}} \otimes_R \mc{O}/\varpi^r,
\mc{O}/\varpi^r)$ can be viewed as an $\mc{O}$-linear derivation
$d_r:R \to \mc{O}/\varpi^r$. Then $d_r$ acting on $f_{n+r}$ is
$f_{n+r}': R \to \mc{O}/\varpi^{n+r}$ given by
$f_{n+r}'(a):= f_{n+r}(a) + d_r(a)$, for $a \in R$, where we use the
inclusion of $\mc{O}/\varpi^r$ in $\mc{O}/\varpi^{n+r}$ as defined
in \S\ref{notation}. Using the fact that $d_r$ is a derivation and
$r \leq n$ one easily checks that $f_{n+r}'$ is an $\mc{O}$-algebra
map and the corresponding element $x_{n+r}' \in X_{n+r}$ reduces to
$x_n$.

If $R$ is formally smooth over $\mc{O}$ then these sets have a very
simple structure: all the reduction maps are surjective and the fibers
of the maps for $r\leq n$ are principal homogenous spaces over
$\Hom_{\mc{O}} (\Omega_{R/\mc{O}} \otimes_R \mc{O}/\varpi^r,
\mc{O}/\varpi^r)$, a free module over $\mc{O}/\varpi^r$ of rank
$\dim(R) -1$. For arbitrary $R$ the nonempty fibers of $\pi_{n,r}$ do
have the same principal homogeneous space property, but these maps
need not be surjective. However, we show below that under some
relatively mild conditions on $R$ there exist nonempty subsets
$Y_n \subset X_n$ such that $\pi_{n,r}$ induces surjections
$Y_{n+r} \to Y_n$ and, for fixed $r$ and $n \gg 0$, the fibers of
these maps are principal homogenous spaces over a suitable submodule
of
$\Hom_{\mc{O}}(\Omega_{R/\mc{O}} \otimes_R \mc{O}/\varpi^r,
\mc{O}/\varpi^r)$ that is free over $\mc{O}/\varpi^r$ of rank
equal to $\dim(R) - 1$.


\bigskip

Let $A = \mc{O}[[x_1,x_2,\dots,x_D]]$, with
$D = \dim_k m_R/(m_R^2, \varpi)$, and let $\alpha:A \to R$ be a
surjection of $\mc{O}$-algebras. The map $\alpha$ induces an
inclusion of the set $X$ of $\mc{O}$-valued points of $\Spec(R)$ into
$({m_{\mc{O}}})^D$ with
image a closed subset\footnote{We use parentheses to indicate that here ${m_{\mc{O}}}$ is viewed
  simply as an open subset of $\mc{O}$ with its $\varpi$-adic topology
  and the exponent $D$ denotes the $D$-fold Cartesian product.}. We also have an identification of the
$E$-valued points of $\Spec(R[\varpi^{-1}])$ with the $\mc{O}$-valued
points of $\Spec(R)$: any map of $\mc{O}$-algebras from
$R[\varpi^{-1}] \to E$ must map all the $\alpha(x_i)$ to elements of
${{m_{\mc{O}}}} \subset E$.

We assume that there exist $y \in X$ such that $\Spec(R[\varpi^{-1}])$
is formally smooth and of dimension $d$ at $y$. Since $X$ is defined
as a closed subset (in the $\varpi$-adic topology) of $({m_{\mc{O}}})^D$ by
finitely many power series, generators $f_1,f_2,\dots, f_e$ of
$\ker(\alpha)$ (note that $e$ is not the ramification index here), it
follows by the Jacobian criterion and the implicit function theorem
that there is an open set $y \in U \subset X$ such that the Jacobian
matrix of $f_1,f_2,\dots,f_e$ has rank $D-d$ at all points in $U$. (We
assume here and in what follows that $e \geq 1$, since otherwise $R$
is formally smooth over $\mc{O}$, and the main results of this section
are clear.) Thus, $U$ is a ``$\varpi$-adic'' (analytic) manifold of
dimension $d$ (in the naive sense).

\begin{lemma} \label{lem:torsion}
  If $U$ is compact, there exists an integer $v \geq 0$ such that the
  torsion in $\Omega_{R/\mc{O}}\otimes_{R,y} \mc{O}$ is annihilated by
  $\varpi^v$ for all $y \in U$.
\end{lemma}

\begin{proof}
  Let $y \in U$ and consider the Jacobian matrix of
  $f_1,f_2,\dots,f_e$. Since $X$ is a manifold of dimension $d$ at
  $y$, this matrix evaluated at $y$ has rank $D-d$, so it has an
  invertible $D-d \times D-d$ minor. By continuity, this minor remains
  invertible and its determinant has constant valuation in an open
  neighbourhood of $y$. We conclude (using the compactness of $U$) by
  applying Lemma \ref{lem:minor} below.
\end{proof}

\begin{lemma} \label{lem:minor} Let $\mc{O}$ be any discrete valuation
  ring with uniformizer $\varpi$ and let $E$ be its quotient
  field. Let $M$ be an $e \times D$ matrix with entries in $\mc{O}$
  such that its rank (as a matrix with entries in $E$) is $D-d$. If
  $M$ has an $(D-d) \times (D-d)$ minor whose determinant has
  valuation $v$, then the torsion submodule of the quotient $Q$ of
  $\mc{O}^D$ by the submodule generated by the rows of $M$ is
  annihilated by $\varpi^v$.
\end{lemma}

\begin{proof}
  We may assume that $e = D-d$ and then by performing row operations
  and permuting the columns, we may assume that the first $D-d$
  columns of $M$ form an upper triangular matrix whose diagonal
  entries have product a unit times $\varpi^v$. It follows from
  this that the image in $Q$ of the submodule of $\mc{O}^D$ spanned by
  $\varpi^ve_i$, $i=1,2,\dots,D-d$, with $e_i$ the $i$-th standard basis
  vector, is contained in the image of the submodule spanned by
  $e_{D-d+1},\dots,e_D$. This submodule must be torsion free since the
  rank of $Q$ is $d$, so the lemma is proved.
\end{proof}

We now analyze the structure of the reductions of $U$ modulo $\varpi^n$ for
$n \gg 0$. This has essentially been done by Serre
\cite{serre:cebotarev}, but we explain part of his proof in order to make
explicit a couple of points that are crucial for our applications.

\begin{lemma} \label{lem:serre} There exists a compact open set $Y$
  such that $y \in Y \subset U$ and submodules ${{\z}}_r$ of
  $\Hom_{\mc{O}}(\Omega_{R/\mc{O}} \otimes_{R,y} \mc{O},
  \mc{O}/\varpi^r)$, $r>0$, that are free over $\mc{O}/\varpi^r$ of rank
  $d$ with the following property: For any integer $n > 0$ let $Y_n$
  be the reduction of $Y$ modulo $\varpi^n$, so we have induced maps
  $\pi^Y_{n,r}:Y_{n+r} \to Y_{n}$. Given any integer $r_0 > 0$, there
  exists an integer $n_0>0$, such that for any $n \geq n_0$ the fibers
  of the map $\pi_{n,r}^Y\,$, for $n \geq n_0$ and
  $0 \leq r \leq r_0\,$, are nonempty principal homogenous spaces over
  ${{\z}}_r$.  Here the action of ${{\z}}_r$ is induced by the natural action of
  $\Hom_{\mc{O}}(\Omega_{R/\mc{O}} \otimes_{R,y} \mc{O},
  \mc{O}/\varpi^r)$ on the fibers of the map
  $\pi_{n,r}:X_{n+r} \to X_n$.
\end{lemma}

\begin{proof}
  Following \S 3.3 of Serre \cite{serre:cebotarev}, we first consider
  some subsets $Y$ of $({m_{\mc{O}}})^D$ and examine the properties of their
  reductions.
\begin{enumerate}[label=(\roman*)]
\item $Y$ is defined by equations of the form
  \begin{align*}
    x_{d+1}& = \phi_{d+1}(x_1,\dots,x_d) \\
    x_{d+2} &= \phi_{d+2}(x_1,\dots,x_d) \\
    \dots & \\
    x_{D} &= \phi_{D}(x_1,\dots,x_d)
  \end{align*}
  where $\phi_i \in A$, for $i=d+1,\dots, D$, and $y = 0$. Clearly $Y$
  is the set of $\mc{O}$-valued points of the ring $R'$, defined as
  the quotient of $A$ by the ideal generated by
  $\{x_{d+i} - \phi_{d+i}\}_{i=d+1}^D$.  The ring $R'$ is formally
  smooth over $\mc{O}$, so the desired statements hold for $Y$ with
  ${{\z}}_r = \Hom_{\mc{O}}(\Omega_{R'/\mc{O}} \otimes_{R',y} \mc{O},
  \mc{O}/\varpi^r)$ viewed as a submodule of
  $\Hom_{\mc{O}}(\Omega_{A/\mc{O}} \otimes_{A} \mc{O},
  \mc{O}/\varpi^r)$ using the surjection
  $\Omega_{A/\mc{O}} \to \Omega_{R'/\mc{O}}$ induced by the quotient
  map $\alpha':A \to R'$, and with $n_0 = r_0$.
\item $Y$ is as in the previous case except that we permute the
  coordinates $(x_1,\dots,x_D)$. It is clear that the statement holds
  in this case as well.
  \item $Y$ is of the form
    \[
      Y = y + \varpi^s Y'
    \]
    where $y \in ({m_{\mc{O}}})^D$, $Y'$ is as in the previous case, and
    $s\geq 0$ is any integer. In this case we may take ${{\z}}_r$ as above
    and $n_0 = r_0 +s$.
    \end{enumerate}
  By Proposition 11 of \cite{serre:cebotarev},\footnote{Serre
    considers the case $E = \Q_p$ but the proof extends to general $E$
    \emph{mutatis mutandis}.} there is an open set $Y$ with
  $y \in Y \subset U$ such that $Y$ is of type (iii) above.  This
  completes the proof, except that the module ${{\z}}_r$ that we get is, \textit{a
  priori}, only a submodule of
  $\Hom_{\mc{O}}(\Omega_{A/\mc{O}} \otimes_{A,y} \mc{O}, \mc{O}/\varpi^r)$. To
  see that ${{\z}}_r$ is a submodule of
  $\Hom_{\mc{O}}(\Omega_{R/\mc{O}} \otimes_{R,y} \mc{O}, \mc{O}/\varpi^r)$,
  let $V_n$ be the reduction modulo $\varpi^n$ of the $\mc{O}$-valued
  points of $\Spec(A)$ and let $\pi_{n,r}^A:V_{n+r} \to V_n$ be the
  reduction maps. The inclusions $Y_n \subset X_n \subset V_n$ are
  compatible with the reduction maps and the inclusions
  $X_n \subset V_n$ are also compatible with the action of
  $\Hom_{\mc{O}}(\Omega_{R/\mc{O}} \otimes_{R,y} \mc{O}, \mc{O}/\varpi^r)$ on
  the fibers of $\pi_{n,r}$ and of
  $\Hom_{\mc{O}}(\Omega_{A/\mc{O}} \otimes_{A,y} \mc{O}, \mc{O}/\varpi^r)$ on
  the fibers of $\pi_{n,r}^A$ (for $n \geq r$). Since the nonempty
  fibers of $\pi_{n,r}$ (resp.~$\pi_{n,r}^A$) are principal homogenous
  spaces over
  $\Hom_{\mc{O}}(\Omega_{R/\mc{O}} \otimes_{R,y} \mc{O}, \mc{O}/\varpi^r)$
  (resp.~
  $\Hom_{\mc{O}}(\Omega_{A/\mc{O}} \otimes_{A,y} \mc{O}, \mc{O}/\varpi^r)$),
  it follows that ${{\z}}_r$ must be contained in (the image of)
  $\Hom_{\mc{O}}(\Omega_{R/\mc{O}} \otimes_{R,y} \mc{O}, \mc{O}/\varpi^r)$.
\end{proof}

We now give an explicit description of the modules ${{\z}}_r$ from
Lemma \ref{lem:serre}. By Lemma \ref{lem:torsion}, there is an integer
$v$ such that the $\varpi$-power torsion of
$\Omega_{R/\mc{O}}\otimes_{R,y}\mc{O}$ is annihilated by $\varpi^v$
for all $y \in Y$. For each $r > 0$, consider the $\mc{O}$-submodule
${{\z}}_r' \subset \Hom_R(\Omega_{R/\mc{O}} \otimes_{R,y} \mc{O},
\mc{O}/\varpi^r)$ given by all homomorphisms which are trivial on the
$\varpi$-power torsion in $\Omega_{R/\mc{O}}\otimes_{R,y}\mc{O}$. It
is clearly free over $\mc{O}/\varpi^r$ of rank $d$.
\begin{lemma} \label{lem:tr}
  The submodules ${{\z}}_r$  and ${{\z}}_r'$ are equal for all $r$.
\end{lemma}

\begin{proof}
  Let $n$ be any integer such that $n > n_0$, where $n_0$ is as in the
  conclusion of Lemma \ref{lem:serre} for $r_0= \max\{2v, 2r\}$.  By
  that lemma, the orbit of the reduction $y_{n+r}$ of $y$ in $Y_{n+r}$
  under the action of ${{\z}}_r$ consists of points which can be lifted to
  $Y$, in particular to $Y_{n+2r}$. If
  $t \in \Hom_R(\Omega_{R/\mc{O}} \otimes_{R,y} \mc{O},
  \mc{O}/\varpi^r)$ is such that $t$ is nontrivial on the torsion in
  $\Omega_{R/\mc{O}}\otimes_{R,y}\mc{O}$ and $r \leq v$, then $t$
  cannot be lifted to
  $\Hom_R(\Omega_{R/\mc{O}} \otimes_{R,y} \mc{O},
  \mc{O}/\varpi^{2r})$ since composing the map from the torsion to
  $\mc{O}/\varpi^{2r}$ with the projection to $\mc{O}/\varpi^r$ one gets the
  zero map. But this implies that $t\cdot y_{n+r}$ cannot be lifted to
  $Y_{n+2r}$, which is a contradiction. Thus, ${{\z}}_r \subset {{\z}}_r'$ for
  $v \leq r$ and by the equality of ranks we see that ${{\z}}_r =
  {{\z}}_r'$. The statement for all $r$ then follows from the fact that for
  any $r > 1$, ${{\z}}_{r-1}$ is the reduction of ${{\z}}_r$ modulo $\varpi^{r-1}$
  (this follows from their defining property) and similarly for
  ${{\z}}_{r-1}'$.
\end{proof}

\subsection{Application to deformation rings}
In our applications $R$ will be a suitable quotient of the universal
framed (fixed ``determinant") deformation ring of a mod $\varpi$
representation $\br:\Gamma_F \to G(k)$, for $F/\Q_l$ a finite
extension, arising from choosing an irreducible component of the
generic fiber of the spectrum of the lifting ring
$R^{\square, \mu}_{\br}[1/\vpi]$ (when $l\neq p$) or of the fixed inertial type $\tau$ and 
fixed $p$-adic Hodge type $\mbf{v}$ lifting ring
$R_{\br}^{\square, \mu,  \tau, \mbf{v}}[1/\vpi]$ (when $l =p$: see
\cite[Prop. 3.0.12]{balaji} for the construction of this ring), and
then letting $R$ be the quotient ring corresponding to the Zariski
closure of this component in $\Spec(R^{\square, \mu}_{\br})$ (with its
reduced subscheme structure).

The points of $Y_n$ are identified with (certain) $\mc{O}/\varpi^n$-valued points of $\Spec(R)$, so correspond to lifts of $\br$ to $\mc{O}/\varpi^n$. A point $y \in Y$ corresponds to a lift $\rho: \Gamma_F \to G(\mc{O})$ of $\br$ and $\Hom_R(\Omega_{R/\mc{O}}\otimes_{R,y} \mc{O}, \mc{O}/\varpi^r)$ is naturally isomorphic to an $\mc{O}$-submodule of the group of one-cocycles $Z^1(\Gamma_F, \rho(\fgder) \otimes_{\mc{O}}\mc{O}/\varpi^r)$. The group of one-cocyles also contains the group of coboundaries 
$B^1(\Gamma_F, \rho(\fgder) \otimes_{\mc{O}}\mc{O}/\varpi^r)$.

\begin{lemma} \label{lem:bound}%
  For all $r > 0$ we have $B^1(\Gamma_F, \rho(\fgder)
  \otimes_{\mc{O}}\mc{O}/\varpi^r) \subset {{\z}}_r$.
\end{lemma}

\begin{proof}
  By Lemma \ref{lem:serre}, we may chooose $n \gg 0$ such that the
  fibers of the map $\pi_{n,r}^Y\,$, are nonempty principal homogenous
  spaces over ${{\z}}_r$.  Let $y \in Y$ and also assume that $n$ is
  sufficiently large that the $(\wh{G^{\mr{der}}})^{(n)}(\mc{O})$-orbit (recall
  that $(\wh{G^{\mr{der}}})^{(n)}$ was defined in Definition \ref{relativetriv}) of
  $y$ is contained in $Y$. Let $y_n$ (resp.~$y_{n+r}$) be the
  reduction of $y$ modulo $\varpi^n$
  (resp.~$\varpi^{n+r}$). Clearly $y_n$ is fixed by the
  $(\wh{G^{\mr{der}}})^{(n)}(\mc{O})$ action, so this action maps $y_{n+r}$ into
  another lift of $y_n$ in $Y_{n+r}$. By the first assumption on $n$,
  such a lift differs from $y_{n+r}$ by an element of ${{\z}}_r$.  On the
  other hand, viewing elements of $Y_{n+r}$ as lifts of $\br$ one
  easily sees that the action of $(\wh{G^{\mr{der}}})^{(n)}(\mc{O})$ corresponds
  precisely to changing these elements by coboundaries via the
  identification of $(\wh{G^{\mr{der}}})^{(n)}(\mc{O})/(\wh{G^{\mr{der}}})^{(n+r)}(\mc{O})$ with
  $\mf{g}^{\mr{der}}_r$ given by the exponential map. Thus, all coboundaries are
  in ${{\z}}_r$, so the lemma follows.
\end{proof}
In fact, conjugation of lifts by the group $\wh{G^{\mr{der}}}(\mc{O})$ induces an action on $R[1/\vpi]$, and hence on $R$, by \cite[Lemma 3.4.1]{Bellovin-Gee-G}; see Remark \ref{whconj} below.

\begin{defn}
We define the submodule $L_{\rho,r} \subset H^1(\Gamma_F, \rho(\fgder) \otimes_{\mc{O}} \mc{O}/\varpi^r)$ to be the image of ${{\z}}_r$.
\end{defn}

Putting all of the above together we get:
\begin{prop} \label{prop:gensmooth} Let $\br:\Gamma_F \to G(k)$ be any
  representation with $F/\Q_l$ a finite extension. Let $R$ be chosen
  as above, arising from a choice of irreducible component of either
  $R^{\square, \mu}_{\br}[1/\vpi]$ ($\ell \neq p$) or some
  $R_{\br}^{\square, \mu, \tau, \mbf{v}}$ ($\ell=p$). Assume that $\Spec(R)$
  has an $\mc{O}$-valued point $y$ such that the corresponding point
  of $\Spec(R[1/\varpi])(E)$ is contained in the smooth locus and let
  $\rho: \Gamma_F \to G(\mc{O})$ be the corresponding lift of
  $\br$. Then there exists a nonempty open set
  $y \in Y \subset \Spec(R)(\mc{O}) = \Spec(R[1/\varpi])(E)$
  with the following properties: Let $Y_n$ be the image of $Y$ in
  $\Spec(R)(\mc{O}/\varpi^n)$ and for $n,r \geq 0$ let
  $\pi_{n,r}^Y: Y_{n+r} \to Y_n$ be the induced maps.
  \begin{enumerate}
  \item Given $r_0 > 0$ there exists $n_0 >0$ such that for all $n \geq n_0$ and $0 \leq r \leq r_0$ the fibers of
    $\pi_{n,r}^Y$ are nonempty principal homogenous spaces over a
    submodule
    ${{\z}}_r \subset Z^1(\Gamma_F, \rho(\fgder) \otimes_{\mc{O}}
    \mc{O}/\varpi^r)$ which is free over $\mc{O}/\varpi^r$ of rank $d$.
  \item The $\mc{O}$-module inclusions $\mc{O}/\varpi^{r-1} \to \mc{O}/\varpi^r$, mapping $1$ to $\vpi$, and the
    surjections $\mc{O}/\varpi^{r} \to \mc{O}/\varpi^{r-1}$ induce inclusions
    ${{\z}}_{r-1} \to {{\z}}_r$ and surjections ${{\z}}_r \to {{\z}}_{r-1}$.
  \item \begin{equation*} \label{eq:order}
  |L_{\rho,r}| = \begin{cases}
    |\rho_r(\fgder)^{\Gamma_F}| &\mbox{ if } l \neq p; \\
    |\rho_r(\fgder)^{\Gamma_F}|\cdot |\mc{O}/\varpi^r|^{\dim\left(\Res_{F \otimes_{\Q_p} E/E}(G^0)/P_{\mbf{v}}\right)} & \mbox{ if } l = p 
  \end{cases}
\end{equation*}
Here $P_{\mbf{v}}$ is the parabolic subgroup of the Weil restriction $\Res_{F \otimes_{\Q_p} E/E}(G^0)$ defining the $p$-adic Hodge type $\mbf{v}$: see \cite[Definition 2.8.2]{Bellovin-Gee-G} for a precise definition.
\item The groups $L_{\rho,r}$ are compatible with the maps on
  cohomology induced  by the inclusions $\mc{O}/\varpi^{r-1} \to
  \mc{O}/\varpi^r$ and the surjections $\mc{O}/\varpi^r \to \mc{O}/\varpi^{r-1}$.
\end{enumerate}
\end{prop}

\begin{proof}
  The existence of $Y$ and (1) follow from Lemma \ref{lem:serre} and
  the discussion above. Item (2) follows from Lemma \ref{lem:tr}. Item
  (3) follows from \cite[Theorem A]{Bellovin-Gee-G} (using that for
  regular Hodge--Tate lifts the associated parabolic is a Borel), the
  exact sequence
\[
0 \to \rho_r(\fgder)^{\Gamma_F} \to   \rho_r(\fgder) \to
B^1(\Gamma_F, \rho_r(\fgder)) \to 0
\]
and the definition of $L_{\rho,r}$. To be precise, \cite[Theorem A]{Bellovin-Gee-G} is formulated for the lifting rings without fixing $\mu$, or for the lifting rings with a fixed projection to $G/[G, G]$. When $G$ is not connected, this need not be the same as fixing the projection to $G/G^{\mr{der}}$ as we do. Nevertheless, since $\Lift_{\br}^{\mu}$ and $\Def_{\br}^{\mu}$ are isomorphic to the lifting and deformation functors for $G/Z_{G^0}$-valued lifts of $\br \pmod{Z_{G^0}}$ (see the first paragraph of the proof of Claim \ref{orbitredn}), the result we use does literally follow from \cite[Theorem A]{Bellovin-Gee-G}. Item (4) follows directly from
(2).
\end{proof}
\begin{rmk}\label{whconj}
  We will later implicitly make use of \cite[Lemma
  3.4.1]{Bellovin-Gee-G}, that $\wh{G^{\mr{der}}}(\mc{O})$-conjugation preserves
  any such irreducible component $\Spec(R[1/\vpi])$ (again noting as in the proof of Proposition \ref{prop:gensmooth} that we may apply the results of \cite{Bellovin-Gee-G} via reduction to the adjoint case), as follows. In the
  global setting, we will fix a local lift $\rho$ contained in the
  Zariski closure $S$ in $\Spec(R)$ of a unique component of
  $\Spec(R[1/\vpi])$ and extract, as in Lemma \ref{lem:serre}, an
  analytic neighborhood of $\rho$ in $S(\mc{O})$; since our global
  lifts as constructed in \S \ref{klrsection} and \S
  \ref{klr^Nsection} will only locally interpolate
  $\rho \pmod{\vpi^n}$ modulo $\wh{G^{\mr{der}}}(\mc{O})$-conjugacy, we will
  implicitly in \S \ref{relativeliftsection} use the fact that
  Proposition \ref{prop:gensmooth} applies, with the same output
  $n_0$, to any $\wh{G^{\mr{der}}}(\mc{O})$-conjugate of $\rho$. Moreover, by the
  result just cited in \cite{Bellovin-Gee-G}, we remain on the same
  irreducible component $S$.
\end{rmk}

To refine the local conclusion of our main theorem, we will use the following:
\begin{lemma}\label{smoothapprox}
Let $R$ be a complete local noetherian $\mc{O}$-algebra, and assume that $\Spec(R[1/\vpi])$ has an open dense regular subscheme. Fix an $\mc{O}$-point $x \colon R \to \mc{O}$. Then there exists a finite extension $E'/E$, with ring of integers $\mc{O}'$, such that for every $t \geq 1$ there exists an $x_t \colon R \to \mc{O}'$ such that $x_t \equiv x \pmod{\vpi^t}$ and $x_t$ defines a formally smooth point of $\Spec(R[1/\vpi])$.
\end{lemma}
\begin{proof}
For ease of reference we use the language of rigid geometry, but this is certainly not essential for the proof.

Let $\mf{X} = \Spf(R)$ be the formal scheme over $\Spf(\mc{O})$
associated to $R$ and let $\mf{X}^{\mr{rig}}$ be the rigid space over
$E$ associated to $\mf{X}$ as in \cite[\S 7.1]{dejong:crystalline}. By
\cite[Lemma 7.1.9]{dejong:crystalline}, points of $\mf{X}^{\mr{rig}}$
are in canonical bijection with closed points of $\Spec(R[1/\vpi])$,
and there is a canonical isomorphism of the completion of the local
ring of a closed point of $\Spec(R[1/\vpi])$ with the completion of
the local ring of the associated point on $\mf{X}^{\mr{rig}}$. Thus,
formally smooth points on $\Spec(R[1/\vpi])$ correspond to smooth
points on $\mf{X}^{\mr{rig}}$. It follows that we may replace
$\Spec(R[1/\vpi])$ by $\Sp(B)$, where $B$ is an affinoid algebra over
$E$. By the regularity assumption, the singular locus of $\Sp(B)$,
which is defined by an ideal $I$, does not contain any of its
irreducible components, so it suffices to prove the statement with
formally smooth points replaced by points in the complement of the
zero set of any such ideal. By replacing $\Sp(B)$ by an irreducible
component containing $x$ we may then assume that $B$ is integral.

By the Noether normalization theorem for affinoid algebras (\cite[\S
6.1.2]{BGR}) there exists a finite homomorphism
$\phi:\mbb{T}_d \to B$, with $\mbb{T}_d$ a Tate algebra over $E$ of
dimension $d = \dim(B)$. Since $E$ has only finitely many extensions
of a fixed degree (in a fixed algebraic closure $\ov{E}$) we are
reduced to proving that for a point $x \in \mc{O}^d$ and any proper
ideal $I$ of $\mbb{T}_d$, there is a sequence of points
$\{x_n\}_{n \geq 0}$ of $\mc{O}^d$ converging to $x$ in the
$\vpi$-topology and with none of the $x_n$ in the closed subset
defined by $I$. Since the only element of $\mbb{T}_d$ vanishing on all
of $\mc{O}^d$ is $0$, by choosing $d-1$ general linear polynomials in
$\mbb{T}_d$ vanishing on $x$ we are reduced to the case $d=1$.  In
this case the claim is clear since any element of $\mbb{T}_1$ has only
finitely many zeros.
\end{proof}

\section{The doubling method: constructing mod $\vpi^{\n}$ lifts}\label{doublingsection}
Let $F$ be a number field, let ${\mc{S}}$ be a finite set of places of
$F$ containing all places above $p$, and let $\br \colon \gal{F,
  {\mc{S}}} \to G(k)$ be a continuous homomorphism. In this section we
explain a broad generalization of the techniques of \cite{klr}. This
will allow us to construct (after enlarging ${\mc{S}}$) a mod $\vpi^{\n}$
lift of $\br$ with good local properties, including but not limited to
interpolating (modulo $\wh{G^{\mr{der}}}$-conjugation) any fixed set of local mod
$\vpi^{\n}$ lifts at places in ${\mc{S}}$. We begin in \S
\ref{klrsection} with the case ${\n}=2$, which contains the essence of
the technique; then in \S \ref{klr^Nsection} (see Theorem \ref{klr^N})
we iterate (a rather more elaborate version of) this argument to
produce lifts modulo arbitrarily high powers of $\vpi$. 

\begin{notation}\label{doublingnotation}
In this section, ``dimension'' will unless otherwise noted refer to $\Fp$-dimension. For a (Galois) group $\Gamma$ and an $\Fp[\Gamma]$-module $W$, we will abbreviate $\dim(H^i(\Gamma, W))$ to $h^i(\Gamma, W)$.

We will say that Galois extensions $K_1, \ldots, K_s$ of a field $K$ are strongly linearly disjoint\index[terms]{s@strongly linearly disjoint} over $K$, if each is linearly disjoint from the composite of the other $s-1$, so that the restriction maps induce an isomorphism $\Gal(K_1 \cdots K_s/K) \xrightarrow{\sim}\prod_{i=1}^s \Gal(K_i/K)$.

In this and the following sections, we will study certain local-global
questions in the Galois cohomology of $F$. For finite sets of places
${\mc{S}}' \supset {\mc{S}}$, we will study the restriction maps \index[terms]{P@$\Psi_{{\mc{S}}'}$}
\[
\Psi_{{\mc{S}}'} \colon H^1(\gal{F, {\mc{S}}'}, M) \to \bigoplus_{v \in {\mc{S}}'} H^1(\gal{F_v}, M)
\]
for $M= \br(\fgder)$ or $M= \br(\fgder)^*$. For $h \in H^1(\gal{F, {\mc{S}}'}, M)$ and a subset ${\mc{T}} \subset {\mc{S}}'$, we will write $h|_{{\mc{T}}}$ for the restriction of $h$ to $\bigoplus_{v \in {\mc{T}}} H^1(\gal{F_v}, M)$. At the risk of pedantry, we would like to emphasize the following:
\begin{itemize}
\item The restriction map $\Psi_{{\mc{S}}'}$ is up to unique isomorphism independent of the choice of decomposition group at $v$: two such decomposition groups differ by conjugating by some $\sigma \in \gal{F}$, with $\sigma$ uniquely determined up to an element of one of the decomposition groups. This ambiguity induces a trivial (inner) automorphism of the cohomology.
\item We do however need to pin down particular decomposition groups and particular restriction maps. Thus we begin by fixing $\ov{F} \into \ov{F}_v$ for all $v \in {\mc{S}}$. Whenever a new prime $v$ is introduced to a superset of ${\mc{S}}$, we will fix a decomposition group at $v$. More precisely, any new $v$ will result from applying the \v{C}ebotarev density theorem in some finite Galois extension $M/F$, and typically we will really be choosing a prime $v_M$ of $M$ above $v$ such that the \textit{element} $\mr{Frob}_{v_M}$ is equal to a given pre-specified element $\sigma \in \Gal(M/F)$. There is then a positive-density set of primes $v'$ of $F$ for which there exists a choice $v'_M$ of prime over $v'$ such that $\mr{Frob}_{v'_M}= \sigma$.
\item Our use of trivial primes (as in \S \ref{trivialsection}) will furthermore allow the following. For $M$ equal to either of the Galois modules $\br(\fgder)$ or $\br(\fgder)^*$, for $\phi \in H^1(\gal{F, {\mc{S}}'}, M)$, and for $v$ a trivial prime, we will refer to quantities $\phi(\sigma_v)$ and $\phi(\tau_v)$ as actual elements of $M$. We make the meaning of this explicit, letting $K= F(\br, \mu_p)$ trivialize the Galois action on $M$, and letting $K_{\phi}$\index[terms]{K@$K_{\phi}$, for $\phi$ any global cohomology class} be the fixed field of the homomorphism $\phi|_{\gal{K}}$; $K_\phi$ is still a Galois extension of $F$. The expressions $\phi(\sigma_v)$ and $\phi(\tau_v)$ will then always mean that there is a decomposition group at $v$ (sufficiently specified by choosing a decomposition group in $\Gal(K_{\phi}/F)$) such that $\phi|_{\gal{F_v}}$, which is a well-defined 1-cocycle since $B^1(\gal{F_v}, M)=0$, takes the appointed values on a fixed generator $\tau_v$ of the pro-$p$ quotient of the tame inertia and on a choice of Frobenius element $\sigma_v$.
\item We will note the first point in our arguments at which a choice of decomposition group appears---it will be in Proposition \ref{klrstep1''}---but from then on this choice should be regarded as fixed.
\end{itemize}

\end{notation}
 
\subsection{Constructing the mod $\vpi^2$ lift}\label{klrsection}

We may assume $\br$ surjects onto $\pi_0(G)$ (if not, we replace $G$ by the preimage in $G$ of the image of $\br$ in $\pi_0(G)$; the deformation theory of $\br$ is unchanged by this replacement). There is then a unique finite Galois extension $\tF/F$ such that $\br$ induces an isomorphism $\Gal(\tF/F) \to \pi_0(G)$. We make the following assumptions on $\br$:
\begin{assumption}\label{multfree}
Assume $p \gg_G 0$, and let $\br \colon \gal{F, {\mc{S}}} \to G(k)$ be a continuous representation unramified outside a finite set of finite places ${\mc{S}}$; we may and do assume that ${\mc{S}}$ contains all places above $p$. Assume that $\br$ satisfies the following:
\begin{itemize} 
\item The field $K= F(\br, \mu_p)$\index[terms]{K@$K$}
      does not contain $\mu_{p^2}$.
 \item $H^1(\Gal(K/F), \br(\fgder)^*)$=0.
 \item $\br(\fgder)$ and $\br(\fgder)^*$ are semisimple $\mathbb{F}_p[\gal{F}]$-modules (equivalently, semisimple $k[\gal{F}]$-modules) having no common $\Fp[\gal{F}]$-subquotient, and neither contains the trivial representation.
\end{itemize}
How large $p$ must be given (the root datum of) $G$ can be extracted from the arguments of this section, but we do not make it explicit.
\end{assumption}
\begin{rmk}
Given a global Galois representation $\br \colon \gal{F, {\mc{S}}} \to G(k)$, we will refer to places $w \not \in {\mc{S}}$ of $F$ such that $\br|_{\gal{F_w}} =1$ and $N(w) \equiv 1 \pmod{p}$ as ``trivial primes."\index[terms]{t@trivial primes} All of the auxiliary primes constructed in this paper will satisfy this condition (and frequently some refinement of this condition).
\end{rmk}
\begin{rmk}\label{infinitetriv}
We make some additional remarks on these assumptions:
\begin{itemize}
\item We could carry out the analysis of this section without the assumption that $K$ does not contain $\mu_{p^2}$, but it is convenient and does not affect our application. In almost every case, we will use this to choose auxiliary trivial primes that also satisfy $N(w) \not \equiv 1 \pmod{p^2}$; the exception will be at one point in the proof of Theorem \ref{klr^N}, when we are constructing lifts modulo $\vpi^{\n}$.
Corollary \ref{infiniteram} constructs infinitely-ramified $p$-adic lifts, and it does not require this assumption on $\mu_{p^2}$. 
\item Note that we do not assume $H^1(\Gal(K/F), \br(\fgder))$ vanishes: in this section we will only need that non-zero classes in $H^1(\gal{F}, \br(\fgder)^*)$ have non-zero restriction to $\gal{K}$, and not the corresponding statement for $\br(\fgder)$. The relative deformation theory method of \S \ref{relativeliftsection} uses results of Lazard to prove a different vanishing result that is crucial to our argument: see Lemma \ref{lemma:fn} and Remark \ref{advanishing}.
\end{itemize}
\end{rmk}
We decompose 
\[
\br(\fgder)= \oplus_{i \in I} W_i^{\oplus m_i}
\] 
where each $W_i$ is an irreducible $\Fp[\gal{F}]$-module, and $W_i
\not \cong W_j$ for $i \neq j$. Dually we obtain the decomposition
$\br(\fgder)^*= \oplus_{i \in I} (W_i^*)^{m_i}$, where $W^*= \Hom_{\Fp}(W, \Fp)(1)$ is the (Tate twist of the) $\Fp$-dual. Each $W_i$ is a $k_{W_i}= \End_{\Fp[\gal{F}]}(W_i)$-module, and since $\mr{Br}(\Fp)=0$, $k_{W_i}$ is a finite field extension of $\Fp$. We may then also regard $W_i^*$ as the $k_{W_i}$-dual, with the trace identifying the $k_{W_i}$-vector spaces
\[
\tr_{k_{W_i}/\Fp} \colon \Hom_{k_{W_i}}(W_i, k_{W_i}) \xrightarrow{\sim} \Hom_{\Fp}(W_i, \Fp).
\]  

We begin by finding some mod $\vpi^2$ lift of $\br$.  Let ${\mc{T}} \supset
{\mc{S}}$ be a finite set of places with ${\mc{T}} \setminus
{\mc{S}}$ consisting of trivial primes $w$ not split in
$K(\mu_{p^2})$ such
that \index[terms]{S@$\Sha^1_{\mc{T}}(\gal{F, {\mc{T}}}, \br(\mf{g}^{\mr{der}})^*)$}
$\Sha^1_{\mc{T}}(\gal{F, {\mc{T}}}, \br(\mf{g}^{\mr{der}})^*)=0$. That
${\mc{T}}$ can be so arranged follows from the first two items of Assumption
\ref{multfree}: the cocycles in question restrict non-trivially to
$\gal{K}$, and then we choose places $v$ that are split in
$K$ and non-split in both
$K(\mu_{p^2})$ and the fixed field (over
$K$) of the cocycle (the latter two conditions are compatible whether
or not the fixed field is disjoint from
$K(\mu_{p^2})$, since they are both just the condition of being
non-trivial).  Following our conventions, when introducing auxiliary
primes we specify decomposition groups at these primes, which here
will only depend on choosing a prime above
$w$ in $K_{\psi}(\mu_{p^2})$ for some element $\psi \in H^1(\gal{F,
  {\mc{T}}}, \br(\fgder)^*)$.

By global duality, $\Sha^2_{\mc{T}}(\gal{F, {\mc{T}}},
\br(\mf{g}^{\mr{der}}))$ \index[terms]{S@ $\Sha^2_{\mc{T}}(\gal{F, {\mc{T}}}, \br(\mf{g}^{\mr{der}}))$} also vanishes, so to produce some lift $\rho_2 \colon \gal{F, {\mc{T}}} \to G(\mc{O}/\vpi^2)$ of $\br$ with multiplier $\mu$ it suffices to check there are no local lifting obstructions: 
\begin{lemma}\label{localmodp^2}
Assume $p \gg_G 0$. Then for all finite places $v$, the set of mod $\vpi^2$ local lifts $\Lift_{\br|_{\gal{F_v}}}(\mc{O}/\vpi^2)$ is non-empty, and similarly for lifts of type $\mu$.
\end{lemma}
\proof
For $p \gg_G 0$, there exists a faithful representation $r \colon G
\into \mr{GL}_n$ such that $\mf{g}$ is a direct summand of $\mf{gl}_n$
(as $G$-modules). The induced map $H^2(\gal{F_v}, \br(\fg)) \to H^2(\gal{F_v}, r\circ\br(\mf{gl}_n))$ is thus injective, and it clearly sends the obstruction to lifting $\br$ (to $G(\mc{O}/\vpi^2)$) to the obstruction to lifting $r \circ \br$. But the latter is unobstructed by \cite[Theorem 1.1]{bockle:modp2}. The fixed multiplier character analogue follows from choosing some lift modulo $\vpi^2$ and then, using the fact that $Z_{G^0} \to G/G^{\mr{der}}$ has kernel of order prime to $p$, modifying it to a lift of type $\mu$.
\endproof
In fact, in the application we will make a stronger assumption on $\br|_{\gal{F_v}}$ for $v \in {\mc{S}}$, obviating the need for this lemma; for now we are trying to proceed without superfluous hypotheses.

Combining the vanishing of $\Sha^2_{\mc{T}}(\gal{F, {\mc{T}}}, \br(\mf{g}^{\mr{der}}))$ with Lemma \ref{localmodp^2}, let $\rho_2 \colon \gal{F, {\mc{T}}} \to G(\mc{O}/\vpi^2)$ be any multiplier $\mu$ lift of $\br$.

In what follows, it will be technically convenient to enlarge the set ${\mc{T}}$ by trivial primes (and choices of decomposition groups) non-split in $K(\mu_{p^2})$, beyond what is necessary to annihilate $\Sha^1_{\mc{T}}(\gal{F, {\mc{T}}}, \br(\fgder)^*)$. We may and do assume that our ${\mc{T}}$ strictly contains whatever initial choice of ${\mc{T}}$ was used to annihilate the Shafarevich--Tate groups; more precise enlargements of this set ${\mc{T}}$ will follow. We can compute such an enlargement's effect on global Galois cohomology. More generally, if $W$ is any $\Fp[\gal{F, {\mc{T}}}]$-module, we have the analogous notion of a trivial prime $v$ for $W$: $W|_{\gal{F_v}}$ is trivial, and $N(v) \equiv 1 \pmod p$. If moreover $W$ satisfies $\Sha^1_{\mc{T}}(\gal{F, {\mc{T}}}, W)=0$, then for any trivial prime $v \not \in {\mc{T}}$ we have an exact sequence
\[
0 \to H^1(\gal{F, {\mc{T}}}, W) \to H^1(\gal{F, {\mc{T}} \cup v}, W) \to H^1(\gal{F_v}, W)/H^1_{\mr{unr}}(\gal{F_v}, W) \to 0,
\]
where the second map is given by evaluation at $\tau_v$; surjectivity of this map follows from the Greenberg--Wiles Euler-characteristic formula (\cite[Theorem 2.19]{DDT2}). (We will apply this surjectivity in Proposition \ref{doublingprop} with $W$ a constituent of $\br(\fgder)^*$.) In particular, the cokernel of the inflation map has dimension $\dim W$. 

We must modify our initial $\rho_2$ so that its local behavior allows further lifting. We will now fix certain local lifts to $G(\mc{O}/\vpi^2)$ that we would like to interpolate into a global mod $\vpi^2$ representation. In Theorem \ref{klr^N}, we will be more particular about what lifts we choose, and we will make additional assumptions on the local behavior of $\br$; so as to be clear about what assumptions are used at what point in the paper, we delay imposing these additional hypotheses. 
\begin{constr}\label{makelambda}
For the remainder of this section, we fix local lifts $\{\lambda_w\}_{w \in {\mc{T}}}$ as follows:
\begin{itemize}
\item For $w \in {\mc{S}}$, fix any lift $\lambda_w \in \Lift_{\br|_{\gal{F_w}}}^{\mu}(\mc{O}/\vpi^2)$. 
\item For $w \in {\mc{T}} \setminus {\mc{S}}$, we simply choose any unramified lift (with multiplier $\mu$) $\lambda_w$ with the property that the elements $\lambda_w(\sigma_w)$ generate $\wh{G^{\mr{der}}}(\mc{O}/\vpi^2)$. For this to be possible, we may have to enlarge the set ${\mc{T}}$, and we do this implicitly at this step of the argument. 
\end{itemize}
\end{constr}
Having fixed the local mod $\vpi^2$ lifts $\lambda_w$ as in Construction \ref{makelambda}, we have that for each $w \in {\mc{T}}$ there is a class $z_w \in H^1(\gal{F_w}, \br(\fgder))$ such that 
\begin{equation}\label{lambdas}
(1+\vpi z_w)\rho_2|_w \sim \lambda_w
\end{equation}
(with $\sim$ denoting strict equivalence). We wish to modify $\rho_2$ by a global cohomology class so that the resulting lift of $\br$ matches the specified local lifts $\lambda_w$. 

If there exists a global class $h \in H^1(\gal{F, {\mc{T}}}, \br(\fgder))$ mapping to $z_{\mc{T}}:= (z_w)_{w \in {\mc{T}}}$ under the localization map
\[
\Psi_{\mc{T}} \colon H^1(\gal{F, {\mc{T}}}, \br(\fgder)) \to \bigoplus_{w \in {\mc{T}}} H^1(\gal{F_w}, \br(\fgder)),
\]
then we proceed to \S \ref{klr^Nsection}. For the remainder of this section, we assume there is no such $h$. Denote by $\Psi_{\mc{T}}^*$ the corresponding localization map for $\br(\fgder)^*$.
To construct auxiliary primes, we will need the following lemma, which we will eventually apply to the irreducible constituents of $\br(\fgder)^*$:
\begin{lemma}\label{lindisjoint}
  Let $W$ be an irreducible $\mathbb{F}_p[\gal{F, {\mc{T}}}]$-module such
  that $H^1(\Gal(F(W)/F), W)=0$. Set
  $k_W= \End_{\mathbb{F}_p[\gal{F, {\mc{T}}}]}(W)$ (a finite extension of
  $\mathbb{F}_p$, since $\mr{Br}(\Fp)=0$), and set (only
    for this lemma) $K= F(W, \mu_p)$.  Let $\psi_1, \ldots, \psi_s$
  be a $k_W$-basis of $H^1(\gal{F, {\mc{T}}}, W)$. Then the fixed fields
  $K_{\psi_1}, \ldots, K_{\psi_s}$ of the cocycles $\psi_i$ are
  strongly linearly disjoint over $K$, and for each $i$,
  $\Gal(K_{\psi_i}/K) \xrightarrow[\psi_i]{\sim} W$. If moreover
  $\mu_{p^2}$ is not contained in $K$, and $W$ is not isomorphic to
  the trivial representation, then fixing an integer $n \geq 2$, a
  class $q_n \in \ker((\Z/p^n)^{\times} \to (\Z/p)^{\times})$, an
  element $w \in W$, and any non-zero class
  $\psi \in H^1(\gal{F, {\mc{T}}}, W)$, there exists a \v{C}ebotarev set
  $\mc{C}$ of trivial primes, and a choice of decomposition group at
  each such $v$, such that $\psi(\sigma_v)=w$ and
  $N(v) \equiv q_n \pmod{p^n}$ for all $v \in \mc{C}$.
\end{lemma}
\begin{proof}
We must show that restriction gives an isomorphism of $\Gal(K/F)$-modules
\[
 \Gal(K_{\psi_1}\cdots K_{\psi_{s}}/K) \xrightarrow{\sim} \prod_{i=1}^{s} \Gal(K_{\psi_i}/K) \xrightarrow{\sim} \prod_{i=1}^{s} W.
\]
To see this, we induct on the number of factors. For $s=1$, the isomorphism follows from simplicity of the $\mathbb{F}_p[\gal{F}]$-module $W$ (note that $\psi_i|_{\gal{K}} \neq 0$). If the linear disjointness is known for $\psi_1, \ldots, \psi_i$, and if $K_{\psi_{i+1}}$ is contained in the composite $K_{\psi_1}\cdots K_{\psi_i}$, then we have a map of $\mathbb{F}_p[\gal{F, {\mc{T}}}]$-modules
\[
 W^{\oplus i} \xleftarrow[\psi_1, \ldots, \psi_i]{\sim} \Gal(K_{\psi_1}\cdots K_{\psi_i}/K) \onto \Gal(K_{\psi_{i+1}}/K) \xrightarrow[\psi_{i+1}]{\sim} W.
\]
Since $W$ is irreducible, the composite $W^{\oplus i} \to W$ has the form $(a_1, \ldots, a_i)$ for some $a_i \in k_W$, and we deduce that $\psi_{i+1}= \sum_{j=1}^i a_j \psi_j$, contradicting linear independence. We conclude that $K_{\psi_{i+1}}$ is not contained in $K_{\psi_1} \cdots K_{\psi_i}$, but again since $W$ is irreducible this forces these fields to be strongly linearly disjoint over $K$.

Now let $\psi$ be any non-zero element of $H^1(\gal{F, {\mc{T}}}, W)$, with splitting field $K_\psi$. The last claim follows by applying the \v{C}ebotarev density theorem in the Galois extension $\Gal(K_{\psi}(\mu_{p^n})/F)$: indeed, the hypotheses ensure that $K_{\psi}$ and $K(\mu_{p^n})$ are strongly linearly disjoint over $K$, since otherwise $W$ would admit the trivial $\Fp[\gal{F}]$-quotient $\Gal(K(\mu_{p^2})/K)$. 
\end{proof}

We next explain in Proposition \ref{klrstep1''} how to interpolate the class $z_{\mc{T}}$ by a global class after allowing ramification at a finite number of additional primes. An important technical point in the proof of Proposition \ref{doublingprop} requires that we impose an additional (at this point rather unmotivated) condition on our trivial primes. We define a field $K'$ whose role will be to capture the intersections of the fixed fields $K_{h^{(v)}}$ of certain auxiliary cocycles $h^{(v)} \in H^1(\gal{F, {\mc{T}} \cup v}, \br(\fgder))$ for varying trivial primes $v$: the fields $K_{h^{(v)}}$ will all properly contain $K'$ and be strongly linearly disjoint over $K'$ (but not necessarily over $K$).
\begin{defn}
  Let $K'$\index[terms]{K@$K'$} be the composite of all (elementary) abelian $p$-extensions
  $L$ of $K$ ($= F(\br,\mu_p)$ as before) that are Galois over $F$ and
  that satisfy
\begin{itemize}
\item $L/F$ is unramified outside ${\mc{T}}$;
\item the $\Fp[\Gal(K/F)]$-module $\Gal(L/K)$ is isomorphic to one of the $W_i$.
\end{itemize}
Because the extensions $L/F$ are unramified outside ${\mc{T}}$ with
absolutely bounded degree, $K'$ is a finite extension of $F$. Primes
split in $K'$ are of course also split in $K$, and 
    $K'$, $K(\mu_{p^{\infty}})$, and $K_{\psi}$ for any given $\psi
    \in H^1(\gal{F, {\mc{T}}}, \br(\fgder)^*)$ are strongly linearly disjoint
    over $K$ since no two of $\br(\fgder)$, $\br(\fgder)^*$, and the
    trivial representation have a common subquotient. Note also that
    $K(\rho_2(\fgder))$, the fixed field of $\rho_2$ acting on
    $\fgder$, is contained in $K'$. In what follows, we will refer to trivial primes split in $K'$
as $K'$-trivial primes. 
\end{defn}
\begin{prop}\label{klrstep1''} Continue in the above setting, so that in particular 
  $\Sha^1_{\mc{T}}(\gal{F, {\mc{T}}},\br(\fgder)^*)=0$, and hence by duality
  $\Sha^2_{\mc{T}}(\gal{F, {\mc{T}}}, \br(\fgder))=0$.
      Fix a Galois
      extension $L/F$ containing $K'$ that is unramified outside ${\mc{T}}$
      and is strongly linearly disjoint over $K'$ from the composite
      of $K'$, $K(\mu_{p^\infty})$, and the fixed fields $K_{\psi}$ of
      any collection of classes
      $\psi \in H^1(\gal{F, {\mc{T}}}, \br(\fgder)^*)$. Let $r$ be the
      dimension of the cokernel of $\Psi_{\mc{T}}$, fix an integer
      $c \geq 2$, and for $i=1, \ldots, r$ fix some
      $q_i \in \ker((\Z/p^c)^\times \to (\Z/p)^\times)$.  Then we can
      find the following:
\begin{itemize}
\item a collection $\{Y_i\}_{i=1}^r$ of elements of $\bigoplus_{v \in \mc{T}} H^1(\gal{F_v}, \br(\fgder))$ inducing an $\Fp$-basis of the cokernel of $\Psi_{\mc{T}} \colon H^1(\gal{F, {\mc{T}}}, \br(\fgder)) \rightarrow \bigoplus_{v \in {\mc{T}}}  H^1(\gal{F_v}, \br(\fgder))$;
\item for all $i$, a split maximal torus $T_i$, a root $\alpha_i \in \Phi(G^0, T_i)$, and a root vector $X_{\alpha_i} \in \mf{g}_{\alpha_i}$;
\end{itemize}
and, once these choices are made, we can choose any elements $g_{L/K', 1}, \ldots, g_{L/K', r} \in \Gal(L/K')$, and then find:
\begin{itemize}
\item for each $i$, a \v{C}ebotarev set $\mc{C}_i$ of $K'$-trivial primes $v \not \in {\mc{T}}$; and
\item a choice of decomposition group at each $v \in \mc{C}_i$ and a class $h^{(v)} \in H^1(\gal{F, {\mc{T}}\cup v}, \br(\fgder))$; 
\end{itemize}
such that for all $v \in \mc{C}_i$:
\begin{itemize}
\item $q_i= N(v) \pmod{p^c}$ is independent of $v \in \mc{C}_i$ ;
\item $h^{(v)}|_{\mc{T}}= Y_i$; 
\item $h^{(v)}(\tau_v)$ is a non-zero element of $\Fp X_{\alpha_i}$; and
\item The image of $\sigma_v$ in $\Gal(L/K')$ is $g_{L/K, i}$.
\end{itemize}
(The last two items implicitly use the choice of decomposition group at the prime $v$: see the remarks in Notation \ref{doublingnotation}. We also emphasize that the elements $g_{L/K', i}$ are fixed \textit{after} the pair $(T_i, \alpha_i)$ has been identified.) 
\end{prop}
\begin{rmk}\label{multfreermk}
 One can ask whether it is possible to hit the class $z_{\mc{T}}$ by allowing only one additional prime of ramification; this is how the analogous argument in \cite{ramakrishna-hamblen} (for reducible two-dimensional $\br$) works. We have only been able to show such a statement when $\br(\fgder)$ is multiplicity-free as an $\Fp[\gal{F}]$-module, and even then only at the expense of arguments considerably more technical than those given here. Proposition \ref{klrstep1''} allows us to avoid this image restriction.
\end{rmk}
\begin{proof}
Since $\Sha_{\mc{T}}^2(\gal{F, {\mc{T}}}, \br(\fgder))=0$, the Poitou--Tate sequence
yields an exact sequence
\[
0 \to \Sha^1_{\mc{T}}(\gal{F, {\mc{T}}}, \br(\fgder)) \to H^1(\gal{F, {\mc{T}}}, \br(\fgder)) \xrightarrow{\Psi_{\mc{T}}} \bigoplus_{v \in {\mc{T}}} H^1(\gal{F_v}, \br(\fgder)) \to \left( H^1(\gal{F, {\mc{T}}}, \br(\fgder)^*)\right)^\vee \to 0,
\]
and $\coker(\Psi_{\mc{T}}) \xrightarrow{\sim} \left( H^1(\gal{F, {\mc{T}}}, \br(\fgder)^*)\right)^\vee$. In particular, if $\dim_{\Fp} \coker(\Psi_{\mc{T}})=r$ is non-zero, then $H^1(\gal{F, {\mc{T}}}, \br(\fgder)^*)$ contains a non-zero class $\psi_1$. We claim that we can choose a triple $(T_1, \alpha_1, X_{\alpha_1})$ consisting of a split maximal torus $T_1$, a root $\alpha_1 \in \Phi(G^0, T_1)$, and a root vector $X_{\alpha_1} \in \mf{g}_{\alpha_1}$ such that $\psi_1(\gal{K'})$ is not contained in $(\Fp X_{\alpha_1})^\perp$ (note that we work with the $\Fp$-span of $X_{\alpha_1}$ rather than the full root space; and here we are simply regarding $\fgder$ and its dual as $\Fp$-vector spaces). Indeed, for any $\Fp$-subspace $U$ not equal to the whole of $\br(\fgder)$, there is a root vector not in $U$. To check this, we must check that the $\Fp$-span, or equivalently the $k$-span, of all root vectors in $\fgder$ is equal to the whole of $\fgder$. This claim in turn reduces to the case in which $\fgder$ is simple, where again (using $p \gg_G 0$) it follows from irreducibility of $\fgder$ as a $k[G(k)]$-module. Thus to find the desired triple it suffices to note that $\psi_1(\gal{K'})$ is non-trivial, by combining the second and third parts of Assumption \ref{multfree}.

    We now apply the \v{C}ebotarev density theorem in the Galois extension $LK'K_{\psi_1}(\mu_{p^c})/F$: by the disjointness observations preceding the Proposition, there is a positive-density set $\mc{D}_1$ of $K'$-trivial primes $v$ and choice of decomposition group at $v$ such that $\sigma_v= g_{L/K', 1}$, $\psi_1(\sigma_v)$ is not in $(\Fp X_{\alpha_1})^\perp$, and $N(v) \equiv q_1 \pmod{p^c}$. For each $v_1 \in \mc{D}_1$, let $L_{v_1}= \{\phi \in H^1(\gal{F_{v_1}}, \br(\fgder)): \phi(\tau_{v_1}) \in \Fp X_{\alpha_1}\}$. Classes in $L_{v_1}^\perp$ are unramified, and so
    \[
    \left \{\psi \in H^1(\gal{F, {\mc{T}} \cup v_1}, \br(\fgder)^*): \psi|_{\gal{F_{v_1}}} \in L_{v_1}^{\perp} \right\}
    \]
is a subspace of $H^1(\gal{F, {\mc{T}}}, \br(\fgder)^*)$ (note there is no condition imposed on the restriction of $\psi$ to places in ${\mc{T}}$). To emphasize and more clearly keep track of the fact that for varying $v_1$ these are subspaces of this common cohomology group, we will write $H^1_{L_{v_1}^\perp}(\gal{F, {\mc{T}}}, \br(\fgder)^*)$ for this subspace, instead of the more typical $H^1_{L_{v_1}^\perp}(\gal{F, {\mc{T}} \cup v_1}, \br(\fgder)^*)$.

We now deduce from a few applications of the Greenberg--Wiles formula the following points:
\begin{itemize}
\item The fact that $\dim(L_{v_1})=1+\dim(L_{v_1}^{\mr{un}})$ (and $\dim(L_{v_1}^\perp)= \dim(L_{v_1}^{\mr{un}, \perp})-1$) and the existence of $\psi_1$ together imply that $h^1_{L_{v_1}^\perp}(\gal{F, {\mc{T}}}, \br(\fgder)^*)=r-1$, in the Selmer group notation of the previous paragraph.
\item The inclusion
\[
\Sha^1_{\mc{T}}(\gal{F, {\mc{T}}}, \br(\fgder)) \subseteq \ker\left(H^1_{L_{v_1}}(\gal{F, {\mc{T}} \cup v_1}, \br(\fgder)) \to \bigoplus_{v \in {\mc{T}}} H^1(\gal{F_v}, \br(\fgder))\right)
\]
is an equality. To see this, apply the Greenberg--Wiles formula to the Selmer systems $\mc{L}_1= \{0\}_{v \in {\mc{T}}} \cup \{L_{v_1}\}$ and $\mc{L}_2= \{0\}_{v \in {\mc{T}}}$ and use the previous bullet-point.
\item The cokernel of the restriction map $H^1_{L_{v_1}}(\gal{F, {\mc{T}} \cup v_1}, \br(\fgder)) \to \bigoplus_{v \in {\mc{T}}} H^1(\gal{F_v}, \br(\fgder))$ has dimension $r-1$: by the last bullet-point, it suffices to show that 
\[
h^1_{L_{v_1}}(\gal{F, {\mc{T}} \cup v_1}, \br(\fgder))=1+h^1(\gal{F, {\mc{T}}}, \br(\fgder)),
\]
which is immediate from the Greenberg--Wiles formula and the vanishing of $\Sha^1_{\mc{T}}(\gal{F, {\mc{T}}}, \br(\fgder)^*)$.
\end{itemize}
Now, if $r>1$, then for each $v_1 \in \mc{D}_1$ we can choose a
non-zero $\psi_2 \in H^1_{L_{v_1}^\perp}(\gal{F, {\mc{T}}},
\br(\fgder)^*)$. Note that $\psi_2$ depends on this subspace of the
Galois cohomology, rather than on $v_1$ itself. Then we can repeat the
above argument, choosing $(T_2, \alpha_2, X_{\alpha_2})$ such that
$\psi_2(\gal{K'})$ is not contained in $(\Fp X_{\alpha_2})^\perp$, and
then define a \v{C}ebotarev set $\mc{D}_2(v_1)$ (the notation includes
the dependence on the initial choice of $v_1$) as the set of
$K'$-trivial $v$ such that 
    (having fixed a suitable decomposition group at $v$) $\sigma_v= g_{L/K', 2}$, $N(v) \equiv q_2 \pmod{p^c}$, and $\psi_2(\sigma_v) \not \in (\Fp X_{\alpha_2})^\perp$. The same argument with the Greenberg--Wiles formula shows that 
\[
\ker\left(H^1_{L_{v_1}, L_{v_2}}(\gal{F, {\mc{T}} \cup v_1 \cup v_2}, \br(\fgder)) \to \bigoplus_{v \in {\mc{T}}} H^1(\gal{F_v}, \br(\fgder)) \right)
\]
equals $\ker\left(H^1_{L_{v_1}}(\gal{F, {\mc{T}} \cup v_1}, \br(\fgder)) \to \bigoplus_{v \in {\mc{T}}} H^1(\gal{F_v}, \br(\fgder))\right)$, and consequently that the dimension of the cokernel of the restriction map $H^1_{L_{v_1}, L_{v_2}}(\gal{F, {\mc{T}} \cup v_1 \cup v_2}, \br(\fgder)) \to \bigoplus_{v \in {\mc{T}}} H^1(\gal{F_v}, \br(\fgder))$ is now $r-2$. Proceeding inductively, we obtain \v{C}ebotarev sets $\mc{D}_s(v_{s-1})$, depending on $v_{s-1} \in \mc{D}_{s-1}(v_{s-2})$ (and so on), for $s=1, \ldots, r$, such that for all tuples $(v_1, \ldots, v_r)$ with each $v_s \in \mc{D}_s(v_{s-1})$, the restriction map
\[
H^1_{L_{v_1}, \ldots, L_{v_r}}(\gal{F, {\mc{T}} \cup v_1, \ldots, v_r}, \br(\fgder)) \to \bigoplus_{v \in {\mc{T}}} H^1(\gal{F_v}, \br(\fgder))
\]
is surjective. We are not done, however, as we need to show that we can in fact define \v{C}ebotarev sets for each $i =1, \ldots, r$ from which we can independently draw primes satisfying the conclusions of the Proposition.

To that end, the above argument produces an $\Fp$-basis $\psi_1, \ldots, \psi_r$ of $H^1(\gal{F, {\mc{T}}}, \br(\fgder)^*)$, a collection of root vectors $X_{\alpha_1}, \ldots, X_{\alpha_r}$, and a collection of elements $Y_1, \ldots, Y_r \in \bigoplus_{v \in {\mc{T}}} H^1(\gal{F_v}, \br(\fgder))$ that map to a basis of $\coker(\Psi_{\mc{T}})$: for $Y_i$, we take \textit{any} vector in the image of $H^1_{L_{v_i}}(\gal{F, {\mc{T}} \cup v_i}, \br(\fgder)) \to \bigoplus_{v \in {\mc{T}}} H^1(\gal{F_{v}}, \br(\fgder))$ that is not in $\im(\Psi_{\mc{T}})$. (These still span $\coker(\Psi_{\mc{T}})$ because if $\wt{Y}_i$ denotes a lift to $H^1_{L_{v_i}}(\gal{F, {\mc{T}}}, \br(\fgder))$ of $Y_i$, the $\{\wt{Y}_i\}_i$ span $\coker(H^1(\gal{F, {\mc{T}}}, \br(\fgder))\to H^1_{L_{v_1}, \ldots, L_{v_r}}(\gal{F, {\mc{T}} \cup v_1, \ldots, v_r}, \br(\fgder))$, since they are independent for ramification reasons.) For each $i$, we also can fix an $\Fp$-basis $\omega_{i, 1}, \ldots, \omega_{i, r-1}$ of $H^1_{L_{v_i}^\perp}(\gal{F, {\mc{T}}}, \br(\fgder)^*)$. Now we define the following \v{C}ebotarev condition:
\begin{align*}
  \mc{C}_i=\{&\text{$K'$-trivial primes $v$ such that 
               $N(v) \equiv q_i \pmod{p^c}$, $\sigma_v= g_{L/K', i}$, }\\
&\text{$\psi_i(\sigma_v) \not \in (\Fp X_{\alpha_i})^\perp$ and $\omega_{i, k}(\sigma_v) \in (\Fp X_{\alpha_i})^\perp$ for all $k=1, \ldots, r-1$}\}.
\end{align*}
We know that $v_i \in \mc{C}_i$, so each $\mc{C}_i$ is in fact a non-empty \v{C}ebotarev condition in the extension $LK'K_{\psi_i} \prod_{k=1}^{r-1} K_{\omega_{i, k}}(\mu_{p^c})/F$ (without this observation, the conditions defining $\mc{C}_i$ could be incompatible). Now, for all $v \in \mc{C}_i$, we define $L_v$ as before to be those classes $\phi \in H^1(\gal{F_v}, \br(\fgder))$ such that $\phi(\tau_v) \in \Fp X_{\alpha_i}$ and deduce an exact sequence
\[
0 \to \Sha^1_{\mc{T}}(\gal{F, {\mc{T}}}, \br(\fgder)) \to H^1_{L_v}(\gal{F, {\mc{T}} \cup v}, \br(\fgder)) \to \bigoplus_{w \in {\mc{T}}} H^1(\gal{F_w}, \br(\fgder)) \to (H^1_{L_v^\perp}(\gal{F, {\mc{T}}}, \br(\fgder)^*))^\vee \to 0;
\]
indeed, we apply the same Euler-characteristic arguments as above (using that $\psi_i|_v \not \in L_v^\perp$), note that the composite $H^1_{L_v}(\gal{F, {\mc{T}} \cup v}, \br(\fgder)) \to (H^1_{L_v^\perp}(\gal{F, {\mc{T}}}, \br(\fgder)^*))^\vee$ is zero (by the standard Poitou--Tate sequence for $\gal{F, {\mc{T}} \cup v}$ acting on $\br(\fgder)$), and then deduce the exactness by counting dimensions. We claim that $Y_i$ lies in the image of $H^1_{L_v}(\gal{F, {\mc{T}} \cup v}, \br(\fgder))$, for which it suffices to check that $Y_i$ annihilates $H^1_{L_v^\perp}(\gal{F, {\mc{T}}}, \br(\fgder)^*)$. We know that $Y_i$ annihilates $H^1_{L_{v_i}^\perp}(\gal{F, {\mc{T}}}, \br(\fgder)^*)$ (using exactness of the above sequence for $v_i$), so it suffices (and is in fact necessary) to observe that 
\[
H^1_{L_{v_i}^\perp}(\gal{F, {\mc{T}}}, \br(\fgder)^*)= H^1_{L_{v}^\perp}(\gal{F, {\mc{T}}}, \br(\fgder)^*);
\]
this holds because both subspaces of $H^1(\gal{F, {\mc{T}}}, \br(\fgder)^*)$ are equal to the span of $\omega_{i, 1}, \ldots, \omega_{i, r-1}$. We can therefore take $h^{(v)}$ to be an element of $H^1(\gal{F, {\mc{T}} \cup v}, \br(\fgder))$ restricting to $Y_i$; $h^{(v)}$ must be ramified at $v$ (since $Y_i$ is not in the image of $\Psi_{\mc{T}}$), so by definition of $L_v$ $h^{(v)}(\tau_v)$ is a non-zero multiple of $X_{\alpha_i}$.
\end{proof}
We will also need the following simpler variant of Proposition \ref{klrstep1''}.
\begin{lemma}\label{gens}
Continue with the hypotheses of Proposition \ref{klrstep1''}. Fix a Galois extension $L/F$ containing $K'$ that is unramified outside ${\mc{T}}$ and is strongly linear disjoint over $K'$ from the composite of $K'$, $K(\mu_{p^\infty})$, and the fixed fields $K_{\psi}$ of any collection of classes $\psi \in H^1(\gal{F, {\mc{T}}}, \br(\fgder)^*)$. Fix an element $g_{L/K'} \in \Gal(L/K')$, an integer $c \geq 2$, and a class $q_Z \in \ker((\Z/p^c)^\times \to (\Z/p)^\times)$. Let $Z \in \fgder$ be any non-zero element. There is a \v{C}ebotarev set $\mc{C}_Z$ of $K'$-trivial primes such that $N(v) \equiv q_Z \pmod{p^c}$ for each $v \in \mc{C}_Z$, and, for each $v \in \mc{C}_Z$, a decomposition group at $v$ and a class $h^{(v)} \in H^1(\gal{F, {\mc{T}} \cup v}, \br(\fgder))$ such that
\begin{itemize}
\item $\sigma_v= g_{L/K'}$;
\item the restriction $h^{(v)}|_{\mc{T}}$ is independent of $v \in \mc{C}$; and
\item $h^{(v)}(\tau_v)$ spans the line $\Fp Z$.
\end{itemize}
\end{lemma}
\begin{proof}
Recall from the discussion preceding Lemma \ref{localmodp^2} that we have enlarged ${\mc{T}}$ to ensure that for all $i \in I$, $H^1(\gal{F, {\mc{T}}}, W_i^*) \neq 0$. Since $(\Fp Z)^\perp \subset (\fgder)^*$ is a proper subspace, it does not contain some isotypic piece $(W_i^*)^{\oplus m_i}$, hence it does not contain some $\gal{F, {\mc{T}}}$-equivariantly embedded $W_i^* \hookrightarrow (W_i^*)^{\oplus m_i}$, and so there is a $\psi \in H^1(\gal{F, {\mc{T}}}, \br(\fgder)^*)$ such that $\psi(\gal{K'})$ is not contained in $(\Fp Z)^\perp$ (namely, a $\psi$ supported on a suitable copy of $W_i^*$). We can now repeat the argument of Proposition \ref{klrstep1''}. In brief, fix a $K'$-trivial prime $v_1$ such that $\psi(\sigma_{v_1})$ does not belong to $(\Fp Z)^\perp$, $\sigma_{v_1}=g_{L/K'}$, and $N(v_1) \equiv q_Z \pmod{p^c}$, and as before define $L_{v_1}$ to be the set of $\phi \in H^1(\gal{F_{v_1}}, \br(\fgder))$ such that $\phi(\tau_{v_1}) \in \Fp Z$. The same analysis shows that there is an element
\[
Y \in \im \Bigl (H^1_{L_{v_1}}(\gal{F, {\mc{T}} \cup v_1}, \br(\fgder)) \to
  \bigoplus_{w \in {\mc{T}}} H^1(\gal{F_w}, \br(\fgder) \Bigr ) \setminus \im(\Psi_{\mc{T}}),
\]
and that if we let $\omega_1, \ldots, \omega_s$ be a basis of the (codimension 1) subspace $H^1_{L_{v_1}^\perp}(\gal{F, {\mc{T}}}, \br(\fgder)^*)\subset H^1(\gal{F, {\mc{T}}}, \br(\fgder)^*)$, then
\begin{align*}
\mc{C}_Z= \{&\text{$K'$-trivial primes $v$: $N(v) \equiv q_Z \pmod{p^c}$, $\sigma_v= g_{L/K'}$,}\\
&\text{$\psi(\sigma_v) \not \in (\Fp Z)^\perp$ and $\omega_j(\sigma_v) \in (\Fp Z)^\perp$ for all $i=1, \ldots, s$}\}
\end{align*}
is a \textit{non-empty} (because $v_1 \in \mc{C}_Z$) \v{C}ebotarev condition. Then as in Proposition \ref{klrstep1''}, we also see that for all $v \in \mc{C}_Z$ there is a class $h^{(v)} \in H^1_{L_v}(\gal{F, {\mc{T}} \cup v}, \br(\fgder))$ such that $h^{(v)}|_{\mc{T}}= Y$, and $h^{(v)}(\tau_v)$ spans $\Fp Z$.
\end{proof}
    In the present subsection we can simply take $L=K'$, $c=2$, and each $q_i$ or $q_Z$ equal to some non-trivial element of $(\Z/p^2)^\times$ (trivial modulo $p$). Only in Theorem \ref{klr^N} will we use the full generality of the results we have just proven.

Now we fix any finite set of root vectors (for possibly different split maximal tori $T_a$) $\{X_{\alpha_a}\}_{a \in A}$ such that
\[
\sum_{a \in A} \Fp[\gal{F}]X_{\alpha_a}= \fgder.
\]
(Such a collection $\{X_{\alpha_a}\}$ clearly exists, since for any proper subspace $U$ of $\fgder$, there is some root vector not in $U$: see the proof of Proposition \ref{klrstep1''}.) Lemma \ref{gens} yields \v{C}ebotarev sets $\mc{C}_a= \mc{C}_{X_{\alpha_a}}$ and classes $Y_a \in \bigoplus_{w \in {\mc{T}}} H^1(\gal{F_w}, \br(\fgder))$ such that for all $v \in \mc{C}_a$, there is a class $h^{(v)} \in H^1(\gal{F, {\mc{T}} \cup v}, \br(\fgder))$ satisfying $h^{(v)}(\tau_v) \in \Fp X_{\alpha_a} \setminus 0$ and $h^{(v)}|_{\mc{T}}=Y_a$; moreover, $N(v) \in (\Z/p^2)^\times$ is independent of $v \in \mc{C}_a$. 
    Here decomposition groups at these primes get fixed once and for all; in particular, there is a fixed prime above each of these $v$'s in $K$, and we can also think of these sets as defining positive-density subsets of $K$ in this way.
Consider the class
\[
z_{\mc{T}}'= z_{\mc{T}}- \sum_{a \in A} Y_a.
\]
This new element may or may not be in the image of $\Psi_{\mc{T}}$, but we can in any case invoke Proposition \ref{klrstep1''} to produce a finite set $\{Y_b\}_{b \in B} \subset \bigoplus_{w \in {\mc{T}}} H^1(\gal{F_w}, \br(\fgder))$ that spans $\coker(\Psi_{\mc{T}})$ over $\Fp$, and, for each $b\in B$, a \v{C}ebotarev set $\mc{C}_b$ of $K'$-trivial primes and a root vector $X_{\alpha_b}$, and for each $v \in \mc{C}_b$ a class $g^{(v)}\in H^1(\gal{F, {\mc{T}} \cup v}, \br(\fgder))$ such that $g^{(v)}|_{\mc{T}}=Y_b$ and $g^{(v)}(\tau_v) \in \Fp X_{\alpha_b} \setminus 0$ (the reason for the shift in notation to $g^{(v)}$ will become apparent at the end of this paragraph); again, $N(v) \in (\Z/p^2)^\times$ is non-trivial and independent of $v \in \mc{C}_b$. In particular,
we can write
\[
z_{\mc{T}}'= h^{\mr{old}}+ \sum_{b \in B} c_bY_b
\]
for some class $h^{\mr{old}} \in H^1(\gal{F, {\mc{T}}}, \br(\fgder))$ and some $c_b \in \Fp$. We discard those $b \in B$ such that $c_b=0$. Thus, for all tuples 
\[
(v_a)_{a \in A} \times (v_b)_{b \in B} \in \prod_{a \in A} \mc{C}_a \times \prod_{b \in B} \mc{C}_b,
\]
we can write 
\[
z_{\mc{T}}= h^{\mr{old}}|_{\mc{T}}+\sum_{a \in A} h^{(v_a)}|_{\mc{T}} + \sum_{b \in B} c_b g^{(v_b)}|_{\mc{T}}.
\]
Note that the vectors $h^{(v_a)}(\tau_{v_a})$ are non-zero multiples of $X_{\alpha_a}$ for all $a \in A$, so the collection $\{h^{(v_a)}(\tau_{v_a})\}_{a \in A}$ is a set of $\Fp[\gal{F}]$-generators of $\fgder$. Having made note of this, we will in the argument that follows not need to preserve the distinction between the sets $A$ and $B$, so we set $N= A \cup B$. In order to preserve this uniformity of notation, for all $b \in B$ and $v \in \mc{C}_b$ we set $h^{(v)}= c_b g^{(v)}$, so we can re-express the above equality as
\begin{equation} \label{z_T}
z_{\mc{T}}= h^{\mr{old}}|_{\mc{T}}+ \sum_{n \in N} h^{(v_n)}|_{\mc{T}}
\end{equation}
for any $\un{v}=(v_n)_{n \in N} \in \prod_{n \in N} \mc{C}_n$. As noted above, there is a tuple of congruence classes $(q_n)_{n \in N} \in (\Z/p^2)^\times$ such that $N(v_n) \equiv q_n \pmod{p^2}$ for all $\un{v} \in \prod_{n \in N} \mc{C}_n$.

    For each $n \in N$, Proposition \ref{klrstep1''} or Lemma \ref{gens} has specified a decomposition group at each $v_n \in \mc{C}_n$. In particular, there is a specified prime $v_{n, K'}$ of $K'$ above $v_n$, and for an $N$-tuple $\un{v}= (v_n)_{n \in N}$, we write $\un{v}_{K'}$ for the $N$-tuple $(v_{n, K'})_{n \in N}$. To be precise, we let let $\mc{C}_{K'} = \prod_{n \in N} \mc{C}_{n, K'}$ denote this set of tuples $(v_{n, K'})_{n \in N}$ of primes of $K'$, but to avoid overburdening the notation, we may omit the additional subscripts if it is clear we are working with primes of $K'$. Since $K'/F$ is split at each $v_n$, each $\mc{C}_{n, K'}$ is still a positive-density set of primes of $K'$. Given the fixed (purely local) choice of $\sigma_{v_n}$ and $\tau_{v_n}$ in the tame quotient of $\gal{F_{v_n}}$, all the properties in the conclusions of Proposition \ref{klrstep1''} and Lemma \ref{gens}, except for the condition in $\Gal(L/K')$, depend only on this prime $v_{n, K'}$ of $K'$, since they involve specification of the conjugacy class of $\sigma_{v_n}$ in an abelian extension of $K'$. We will continue to write $\sigma_{v_n}$ with the understanding that it means the Frobenius element induced by the choice of decomposition group, but keeping in mind that in any abelian extension of $K'$ it only depends on $v_{n, K'}$.

We will need to argue in terms of Dirichlet densities of $N$-tuples of
primes. For any number field $M$, we define the Dirichlet density of a
subset $P$ of $\{\text{primes of $M$}\}^N$ to be (if it exists) \index[terms]{D@$\delta(P)$}
\[
\delta(P)= \lim_{s \to 1^+} \frac{\sum_{\underline{v} \in P} N(\underline{v})^{-s}}{\sum_{\text{all $\underline{v}$}} N(\underline{v})^{-s}},
\]
where $N(\underline{v})= \prod_{n \in N} N(v_n)$. In particular, the
density of a product $P= \prod_{n \in N} P_n$ of sets $P_n$ of primes
exists if each $P_n$ has a density, and in this case $\delta(P)=
\prod_{n \in N} \delta(P_n)$. We make a corresponding definition of
upper Dirichlet density $\delta^+(P)$ \index[terms]{D@$\delta^+(P)$} of a set of $N$-tuples of primes. In particular, the preceding discussion yields \v{C}ebotarev sets $\mc{C}= \prod_{n \in N} \mc{C}_n$ and $\mc{C}_{K'}= \prod_{n \in N} \mc{C}_{n, K'}$ of positive Dirichlet density in their respective number fields.

The following argument substantially uses global duality, and we need to preface with a technical clarification of what coefficients we can take in the duality pairings. We have the $\Fp[\gal{F}]$-isotypic decomposition 
\[
\br(\fgder)= \bigoplus_{i \in I} V_i = \bigoplus_{i \in I} W_i^{\oplus m_i},
\]
where the various $W_i$ are mutually non-isomorphic irreducible $\Fp[\gal{F}]$-modules with endomorphism algebras $k_{W_i} = \mr{End}_{\Fp[\gal{F}]}(W_i)$ (a finite extension of $\Fp$). We may (and do) fix an isomorphism of $V_i$
with $W_i \otimes_{\Fp}\mbb{F}_{p^{m_i}}$ as
$k_{W_i}[\gal{F}]$-modules (with trivial Galois action on
$\mbb{F}_{p^{m_i}}$). This gives $V_i$ the structure of an
$A_i[\gal{F}]$-module, where
$A_i := k_{W_i} \otimes_{\Fp} \mbb{F}_{p^{m_i}}$, with $V_i$ being
finite free as an $A_i$-module. In the sequel, duals and duality
pairings will be considered with respect to this fixed structure.\footnote{We
  use the trace maps to identify duals over the various \'{e}tale
  $\Fp$-algebras that we consider.} 

\begin{prop}\label{doublingprop}
Fix any $2|N|$-tuple $(A_n, A_n')_{n \in N}$ of elements of $\widehat{G^{\mathrm{der}}}(\mathcal{O}/\varpi^2)$. Then there is a finite $2|N|$-tuple of $K'$-trivial primes $\mathcal{Q}=\left( (v_n, v'_n) \right)_{n \in N}$, disjoint from $\mathcal{T}$ and having $\{v_n, v'_n\} \subset \mathcal{C}_n$ for all $n \in N$, and a class $h \in H^1(\Gamma_{F, {\mathcal{T}} \cup {\mathcal{Q}}}, \bar{\rho}(\mathfrak{g}^{\mathrm{der}}))$ such that
\begin{itemize}
\item $h|_{\mathcal{T}}= z_{\mathcal{T}}$.
\item For some $\alpha_n$-root vector $X_n$, $(1+\varpi h)\rho_2(\tau_{v_n})=(1+\varpi h)\rho_2(\tau_{v'_n})=u_{\alpha_n}(X_n)$.
\item For some $z_n, z'_n \in Z_{G^0}(\mathcal{O}/\varpi^2) \cap \widehat{G}(\mathcal{O}/\varpi^2)$,
\[
(1+\varpi h)\rho_2(\sigma_{v_n})= A_n \cdot z_n,
\]
and
\[
(1+\varpi h)\rho_2(\sigma_{v'_n})=A'_n \cdot z'_n.
\] 
\end{itemize}
\end{prop}
\begin{proof}
We have seen in \eqref{z_T} that there is a class $h^{\mr{old}} \in H^1(\gal{F, {\mc{T}}}, \br(\fgder))$ such that for any $N$-tuple $\un{v}= (v_n)_{n \in N} \in \mc{C}= \prod_{n \in N} \mc{C}_n$, the global class $h(\un{v})=h^{\mr{old}}+ \sum_{n \in N} h^{(v_n)}$ satisfies $h(\un{v})|_{\mc{T}}= z_{\mc{T}}$. Since we cannot say anything about the restrictions $h(\un{v})|_{v_n}$, we will use the ``doubling method" of \cite{klr} to find the desired ${\mc{Q}}$ and $h$. To that end, for any two $N$-tuples $\un{v}, \un{v}' \in \mc{C}$, we consider the class
\begin{equation}\label{doublingeqtn}
h= h^{\mr{old}}- \sum_{n \in N} h^{(v_n)}+ 2 \sum_{n \in N} h^{(v'_n)} \in H^1(\gal{F, {\mc{T}} \cup \{v_n\} \cup \{v'_n\}}, \br(\fgder)),
\end{equation}
which still satisfies $h|_{\mc{T}}= z_{\mc{T}}$ (and the inertial conditions dictated by the construction of the classes $h^{(v_n)}$). The argument will show that for a suitable choice of $\un{v}$ and $\un{v}'$, $h$ will satisfy the conclusion of the Proposition with the set ${\mc{Q}}$ equal to $\{v_n\}_{n \in N} \cup \{v'_n\}_{n \in N}$.

    We let $\mf{l}_{K'} \subset \mc{C}_{K'}$, lying above and in bijection with $\mf{l} \subset \mc{C}$, denote a positive upper-density subset (now no longer necessarily a product of \v{C}ebotarev sets) such that the $N$-tuples $(\sum_{n \in N} h^{(v_n)}(\sigma_{v_m}))_{m \in N}$, $(h^{\mr{old}}(\sigma_{v_m}))_{m \in N}$, and $(h^{(v_m)}(\tau_{v_m}))_{m \in N}$ are independent of the choice of $\un{v} \in \mf{l}$ (also recall that $K(\rho_2(\fgder)) \subset K'$, so for any $(v_n)_{n \in N} \in \mc{C}$, $\rho_2(\sigma_{v_n})$ lies in the center of $G^0$ for all $n \in N$); this is possible since as we vary over $\mc{C}$, these $N$-tuples take on only finitely many values. These are conditions in abelian extensions of $K'$ so depend only on the places in $K'$ rather than their extensions to larger number fields.

For $\un{v} \in \mf{l}$, we write $X_n$ for the now independent-of-$\un{v}$ value
$h^{(v_n)}(\tau_{v_n})$ (for all $n \in N$, this is a non-zero
multiple of $X_{\alpha_n}$). Recall the decomposition
$\br(\fgder)= \bigoplus_{i \in I} V_i = \bigoplus_{i \in I}
W_i^{\oplus m_i}$ into $\Fp[\gal{F}]$-isotypic components. For
$i \in I$ we let $X_{n, i}$ denote the $V_i$-component of $X_n$. By
construction, therefore, we have
\[
\sum_{n \in N} \Fp[\gal{F}] X_{n, i}= V_i.
\]

We will show that for any fixed $N$-tuples $(C_m)_{m \in N}$ and $(C'_m)_{m \in N}$ of elements of $\fgder$, there exist $\un{v}, \un{v}' \in \mf{l}$ such that
\begin{align}\label{doubling}
&\sum_{n \in N} h^{(v'_n)}(\sigma_{v_m})= C_m, \\ \label{doubling2}
&\sum_{n \in N} h^{(v_n)}(\sigma_{v'_m})= C'_m,
\end{align}
for all $m \in N$. This will suffice to prove the Proposition, since, by Equation (\ref{doublingeqtn}), it will allow us to prescribe the values $h(\sigma_{v_m})$ and $h(\sigma_{v'_m})$ for all $m \in N$; we then 
choose the $C_m$ and $C'_m$ such that the values $(1+\vpi h)\rho_2(\sigma_{v_m})$ and $(1+\vpi h)\rho_2(\sigma_{v'_m})$ take on whatever values  modulo the center we wish to prescribe.

We will now study the condition, for \textit{fixed} $\un{v}= (v_n)_{n
  \in N} \in \mf{l}$, imposed on $\un{v'}$ by Equations
(\ref{doubling}) and (\ref{doubling2}), beginning with Equation
(\ref{doubling2}).
    Strictly speaking, we will now construct \v{C}ebotarev conditions
    on primes of $K'$. For each $n \in N$, consider the maximal Galois
    extension $K^{(v_n)}$ of $F$ inside $K_{h^{(v_n)}}$ that is
    unramified at $v_n$; this contains $K$, and $\sum_{n \in N}
    h^{(v_n)}(\Gal(K_{h^{(v_n)}}/K^{(v_n)}))= \fgder$, since the
    $n$-component of this sum contains $X_n$ and is
    $\Fp[\gal{F}]$-stable (only the $n \in A$ are needed to guarantee
    Equation (\ref{span}) holds).
        Moreover, since $K_{h^{(v_n)}} \cap K'=K^{(v_n)}$, we likewise have
\begin{equation}\label{span}
\sum_{n \in N} h^{(v_n)}(\Gal(K'K_{h^{(v_n)}}/K'))= \fgder.
\end{equation}

For each $m \in N$, we consider the \v{C}ebotarev condition $\mf{w}'_m$ on primes $w$ of $K'$, split over $F$ (a restriction that does not change the density), that satisfy $N(w) \equiv q_m \pmod{p^2}$, where $q_m$ is the class fixed in the construction of the sets $\mc{C}_m$, and the condition
\[
\sum_{n \in N} h^{(v_n)}(\sigma_w)= C_m.
\]
Equation (\ref{span}) and the fact that $K' \cap K(\mu_{p^\infty})= K$ implies that this condition is non-empty. 

Now we turn to the condition needed to satisfy Equation (\ref{doubling}). For all $m \in N$ and $i \in I$, let $\{\eta^{(v_m)}_{i, j}\}_{j=1}^{d_i}$ be elements of $H^1(\gal{F, {\mc{T}} \cup v_m}, W_i^*)$ that lift a $k_{W_i}$-basis of $H^1(\gal{F, {\mc{T}} \cup v_m}, W_i^*)/H^1(\gal{F, {\mc{T}}}, W_i^*)$. We will use the (global duality) exact sequence
\[
H^1(\gal{F, {\mc{T}} \cup v_m, v_n'}, V_i) \to \bigoplus_{w \in {\mc{T}} \cup v_m, v_n'} H^1(\gal{F_w}, V_i) \to H^1(\gal{F_{{\mc{T}} \cup v_m, v_n'}}, V_i^*)^\vee,
\] 
where both duals are taken as $A_i$-modules; the first map is restriction, and the second is $(x_w)_w \mapsto (\psi \mapsto \sum_w \langle x_w, \psi|_{\gal{F_w}} \rangle_w)$, where $\langle \cdot, \cdot \rangle_w$ refers to the $A_i$-linear version of the local duality pairing. Thus, writing $h^{(v)}= \sum_{i \in I} h_i^{(v)}$ for the decomposition into $V_i$-components, regarding $W_i^*$ as $W_i^* \otimes 1 \subset V_i^*$, and invoking Lemma \ref{localduality} for the local calculation at the trivial primes $v_m$ and $v_n'$, we have for all $m, n, i, j$ the global duality relation
\begin{align*}
\langle \eta^{(v_m)}_{i, j}(\tau_{v_m}), h_i^{(v'_n)}(\sigma_{v_m}) \rangle&= -\sum_{x \in {\mc{T}}} \langle \eta^{(v_m)}_{i, j}, h_i^{(v'_n)}\rangle_x - \langle \eta^{(v_m)}_{i, j}(\sigma_{v'_n}) , h_i^{(v'_n)}(\tau_{v'_n})\rangle,\\
&= -\sum_{x \in {\mc{T}}} \langle \eta^{(v_m)}_{i, j}, h_i^{(v_n')}\rangle_x - \langle \eta^{(v_m)}_{i, j}(\sigma_{v'_n}) , X_{n, i}\rangle,
\end{align*}
where we systematically work with the $A_i$-linear pairings. Summing over $n$, we want to show that for all $m \in N$, $i \in I$, $j=1, \ldots, d_i$, we can prescribe by a \v{C}ebotarev condition (depending on our fixed $\un{v}$ or $\un{v}_{K'}$) on $\un{v}' \in \mc{C}$ 
the values
\[
\sum_{n \in N} \langle \eta^{(v_m)}_{i, j}(\sigma_{v'_n}), X_{n, i} \rangle \in A_i,
\]
for then (recalling that $\sum_{n \in N} h_i^{(v'_n)}|_{\mc{T}}$ is independent of $\underline{v}' \in \mc{C}$) if moreover $\un{v}'$ is in $\mf{l}$ we can achieve the same for the values
\[
\langle \eta^{(v_m)}_{i, j}(\tau_{v_m}), \sum_{n \in N} h_i^{(v'_n)}(\sigma_{v_m}) \rangle.
\]
Prescribing these values for varying $m, i, j$ will allow us to achieve the equality of Equation (\ref{doubling}).

The splitting fields $K_{\eta^{(v_m)}_{i, j}}$ are strongly linearly disjoint over $K$ from $K'$ and from one another as we vary $m\in N$, $i\in I$, and $j=1, \ldots, d_i$,\footnote{Since the $W_i^*$ are irreducible, the fields are disjoint as $m$ varies because $K_{\eta^{(v_m)}_{i, j}}$ is ramified at $v_m$ (and not at $v_{m'}$ for $m' \neq m$); they are disjoint as $i$ varies because the $W_i^*$ are mutually non-isomorphic; and they are disjoint as $j$ varies by Lemma \ref{lindisjoint}.} so it suffices for this last claim to note that for any fixed non-zero vector $w_i^* \in W_i^*$ we have
\[
\sum_{n \in N} \langle W_i^*, X_{n, i} \rangle= \sum_{n \in N} \langle \Fp[\gal{F}]w_i^*, X_{n, i} \rangle = \sum_{n \in N} \langle w_i^*, \Fp[\gal{F}]X_{n, i} \rangle= \langle w_i^*, V_i \rangle= A_i.
\]
(Here the pairing of a vector with a submodule simply means the module of all values taken by pairing elements of that submodule against the given vector.)

    Taken together, these remarks define a positive-density \v{C}ebotarev condition $\mf{l}_{\un{v}, K'}$ on $N$-tuples of primes (split over $F$) in $K'$ given by the intersection of 
\begin{itemize}
\item the conditions just prescribed in the fields $K_{\eta^{(v_m)}_{i, j}}K'$;
\item the conditions defining $\prod_{m \in N} \mf{w}_m'$.
\end{itemize}
The condition thus obtained is non-trivial since all the fields $K' K_{h^{(v_n)}}$, $K' K_{\eta^{(v_m)}_{i, j}}$, and $K'(\mu_{p^2})$ are strongly linearly disjoint over $K'$
(we use here the disjointness of $\br(\fgder)$, $\br(\fgder)^*$, and the trivial representation, and the fact that over $K'$ the field $K'K_{h^{(v_n)}}$ is totally ramified at $v_n$, and   $K'K_{\eta^{(v_m)}_{i, j}}$ is totally ramified at $v_m$). We write $L_{\un{v}}$ for the composite of all of these fields except for $K'(\mu_{p^2})$, so that $\mf{l}_{\un{v}}$ is a non-trivial condition in $L_{\un{v}}(\mu_{p^2})/K'$. Note that $L_{\un{v}}$ is linearly disjoint over $K'$ from any composite $M$ of fields $K'K_{\psi}$, $\psi \in H^1(\gal{F, {\mc{T}}}, \br(\fgder)^*)$, since the $K_{\psi}$ are ramified only in ${\mc{T}}$ (and by the choice of $K'$). The condition defining $\mc{C}$ takes place in some $M(\mu_{p^2})$, $M$ as above; $L_{\un{v}}(\mu_{p^2}) \cap M(\mu_{p^2})= K'(\mu_{p^2})$, and restricting the conditions $\mf{l}_{\un{v}}$ and $\mc{C}$ to $K'(\mu_{p^2})$ yields the same condition (i.e. that the tuple $(w_n)_{n \in N}$ of primes has $N(w_n) \equiv q_n \pmod{p^2}$). Finally, as $\un{v}$ itself varies, the fields $L_{\un{v}}(\mu_{p^2})$ are strongly linearly disjoint over $K'(\mu_{p^2})$ (by the same ramification observation).

If some $\un{v}'_{K'} \in \mf{l}_{\un{v}, K'}$ moreover lies in $\mf{l}_{K'}$, then
\begin{align*}
&\sum_{n \in N} h^{(v'_n)}(\sigma_{v_m})= C_m \\
&\sum_{n \in N} h^{(v_n)}(\sigma_{v'_m})= C'_m.
\end{align*}
We have, however, no assurance that this intersection is non-empty, so
now we must invoke the limiting logic of \cite{klr} and
\cite{ramakrishna-hamblen} that allows the doubling method to
succeed. The various disjointness remarks in the previous paragraph
are crucial to this argument.

    In the remainder of the proof, we will write $\mf{l}$, $\mf{l}_{\un{v}}$, etc., always intending the corresponding positive-density sets $\mf{l}_{K'}$, $\mf{l}_{\un{v},K'}$, etc.; we adopt this convention so as not to overburden the notation. If for each member of a finite subset $\{\un{v}_1, \ldots, \un{v}_s\} \subset \mf{l}$, the intersection $\mf{l} \cap \mf{l}_{\un{v}_k}$ is empty, then $\mf{l} \setminus \{\un{v}_{1}, \ldots, \un{v}_{s}\}$ is contained in $\mf{l} \cap \bigcap_{k=1}^s \overline{\mf{l}_{\un{v}_k}}$, and in particular is contained in $\mc{C} \cap \bigcap_{k=1}^s \overline{\mf{l}_{\un{v}_k}}$ (we take complements in the set of all $N$-tuples of primes of $K'$ split over $F$). We claim that the upper-density of this latter intersection tends to $0$ as $s$ increases. Let $d_k$ be the degree $[L_{\un{v}_k}:K']$. If a tuple $\un{v}'=(v'_n)_{n \in N}$ lies in $\mc{C} \cap_{k=1}^s \ov{\mf{l}_{\un{v}_{k}}}$, it must (by virtue of lying in $\mc{C}$) satisfy $N(v'_n) \equiv q_n \pmod{p^2}$. Thus since $\Gal(\prod_k L_{\un{v}_k}(\mu_{p^2})/K'(\mu_{p^2})) \cong \prod_k \Gal(L_{\un{v}_k}/K')$, 
\[
\delta(\mc{C} \cap \cap_{k=1}^s \ov{\mf{l}_{\un{v}_k}}) \leq \prod_{k=1}^s \left(1-\frac{1}{d_k^{|N|}}\right) \leq \left(1-\frac{1}{D^{|N|}}\right)^s,
\]
where $D$ is a uniform (independent of $k$) upper bound of the $d_k$: such a $D$ exists as a function of $p$, $\dim \fgder$, and $|N|$.

Letting $s$ tend to infinity, 
we see that $\delta^+(\mf{l})$ is less than any positive number, contradicting the fact that $\mf{l}$ has positive upper-density. We conclude that for some $\un{v} \in \mf{l}$, there is a $\un{v}' \in \mf{l} \cap \mf{l}_{\un{v}}$, and so the proof is complete.\footnote{We remark that $\delta( \cap_{k=1}^s \ov{\mf{l}_{\un{v}_k}})$ does not tend to zero as $s$ goes to infinity; this would only be the case if the conditions $\mf{l}_{\un{v}_k}$ were actually imposed in extensions that were strongly linearly disjoint over $K'$.}
\end{proof}
In our application of this result (see Theorem \ref{klr^N}), there will be two ways we choose the values $(1+\vpi h)\rho_2(\sigma_w)$ for $w \in {\mc{Q}}$; we make these choices explicit here. Note that the following proof uses the assumption $p \gg_G 0$ (in a mild way) but does not require enlarging $\mc{O}$.
\begin{lemma}\label{generalposition}
In the conclusion of Proposition \ref{doublingprop}, we can prescribe
the values 
    $t_w:= (1+\vpi h)\rho_2(\sigma_w) \in T_w(\mc{O}/\vpi^2)$, for $w \in {\mc{Q}}$, such that:
\begin{itemize}
\item If $e \geq 2$, then $t_w$ lies in $Z_{G^0}(\mc{O}/\vpi^2)$ (with value determined by $\mu$), and in particular $\alpha_w(t_w) \equiv 1 \equiv N(w) \pmod{\vpi^2}$.
\item If $e=1$, then for all roots $\beta \in \Phi(G^0, T_w)$, $\beta(t_w) \not \equiv 1 \pmod{\vpi^2}$, and $\alpha_w(t_w) \equiv N(w) \pmod{\vpi^2}$.
\end{itemize}
(In both cases the lifts belong to $\Lift^{\alpha_w, \mu}_{\br|_{\gal{F_w}}}(\mc{O}/\vpi^2)$, but we have imposed some precise conditions that will be useful in the application.)
\end{lemma}
\begin{proof}
  Recall that prior to the proof of Proposition \ref{doublingprop} we
  were able to specify congruence conditions $q_n \in (\Z/p^2)^\times$
  for all $n \in N$; any tuples $\un{v}$ and $\un{v}'$ appearing in
  the conclusion of the Proposition satisfy
  $N(v_n) \equiv N(v'_n) \equiv q_n \pmod{p^2}$. In particular, to
  prove the Lemma we just have to show that for fixed $N(w)$ there are
  elements $t_w$ as in the Lemma statement. The $e \geq 2$ case is
  obvious. The $e=1$ case is just a matter of checking that the
  condition $p \gg_G 0$ can be arranged to guarantee that elements
  $t_w$ in sufficiently general position exist. Fix a trivial prime
  $w \in {\mc{Q}}$, with associated $(T_w, \alpha_w)$, and let $q_w=
  N(w)$. Let $q_w^{1/2}$ denote the square-root of $q_w$ in
  $\mc{O}/p^2$ that is congruent to $1$ modulo $p$. Consider elements
  of $\wh{T_w^{\mr{der}}}(\mc{O}/p^2)$ of the form
  $t_b= (1+pb)\alpha_w^\vee(q_w^{1/2})$ for $b \in \ker(\alpha)$ and
  having fixed (specified by $\mu$) projection to $\mf{z}_G$ under the
  decomposition $\fg= \fgder \oplus \mf{z}_G$. By factoring out a term
  $(1+pz)$ with $z \in \mf{z}_G$, it will suffice to treat the case
  $b \in \fgder$. Any such $t_b$ satisfies $\alpha(t_b)=q_w$ and
  reduces to $1 \in G(k)$; we will find $b$ such that
  $\beta(t_b) \neq 1$ for all $\beta \in \Phi:= \Phi(G^0,
  T_w)$. Decompose $\fgder= \fg_1 \oplus \fg_2$, where $\fg_1$ is the
  simple factor supporting $\alpha_w$, and $\fg_2$ is the sum of the
  remaining simple factors. Correspondingly decompose
  $\Lie(T_w)= \mf{t}_1 \oplus \mf{t}_2$, the set of roots
  $\Phi= \Phi_1 \sqcup \Phi_2$, and the element $b= b_1+b_2$. We
  choose $b_2 \in \mf{t}_2$ such that $\beta(\mf{t}_2) \neq 0$ for all
  $\beta \in \Phi_2$; that this is possible requires that the union of
  hyperplanes $\cup_{\beta \in \Phi_2} \ker(\beta|_{\mf{t}_2})$ not
  equal all of $\mf{t}_2$, which is clearly not a problem for
  $p \gg_G 0$. To choose $b_1 \in \mf{t}_1$, first note that, for any
  $\beta \in \Phi_1$, if there is no
  $b_{\beta} \in \ker(\alpha_w|_{\mf{t}_1})$ such that
  $\beta(b_\beta)= -\frac{q_w-1}{2p}\langle \beta, \alpha^\vee
  \rangle$, then the condition $\beta(t_b) \neq 1$ will be satisfied
  automatically for any choice of
  $b_1 \in \ker(\alpha_w|_{\mf{t}_1})$, so we now restrict to the
  subset $\Phi_1^{*}$ of $\Phi_1$ for which such $b_{\beta}$ do exist
  (and we fix one such $b_{\beta}$ for each $\beta$). Then we choose
  $b_1$ in the complement of the union of hyperplanes
\[
\bigcup_{\beta \in \Phi_1^{*} \setminus \{-\alpha\}} \left(b_{\beta}+ \ker(\beta|_{\ker(\alpha_w|_{\mf{t}_1})}) \right)
\]
inside $\ker(\alpha_w|_{\mf{t}_1})$ (note that since $\beta$ belongs to the same simple factor as $\alpha_w$ and is not equal to $-\alpha_w$, $\ker(\beta|_{\ker(\alpha_w|_{\mf{t}_1})})$ is indeed a hyperplane in $\ker(\alpha_w|_{\mf{t}_1})$). The total number of such hyperplanes is bounded in a way depending only on the Dynkin type of $G^{\mr{der}}$, so for $p \gg_G 0$, this complement is non-empty. Clearly for such a $b_1$, and $b=b_1+b_2$, we have $\beta(t_b) \neq 1$ for all $\beta \in \Phi$ and $\alpha(t_b)=q_w$.
\end{proof}

\subsection{Constructing the mod $\vpi^{\n}$ lift}\label{klr^Nsection}
In the proof of the main theorem (\S \ref{relativeliftsection}) we will prove a lifting theorem that requires as input a carefully-constructed mod $\vpi^{\n}$ representation for some ${\n}$ depending on both local (the restrictions $\br|_{\gal{F_v}}$ for $v \in {\mc{S}}$) and global (the image of $\br$) properties of our given residual representation $\br$. To that end, in this section we extend Proposition \ref{doublingprop} to prove the following:
\begin{thm}\label{klr^N}
  Let $p \gg_G 0$ be a prime. Let $F$ be any number field, and let
  $\br \colon \gal{F, {\mc{S}}} \to G(k)$ be a continuous representation such
  that $\br|_{\gal{\widetilde{F}(\zeta_p)}}$ is absolutely
  irreducible.\footnote{An inspection of the proof shows that we only
    use the following: $\br$ satisfies Assumption \ref{multfree} and
    the field $K$ is linearly disjoint over $\tF(\mu_p)$ from
    $\tF(\mu_{p^\infty})$.} Assume that $[\tF(\zeta_p):\tF]$ is
  strictly greater than the constant $a_G$ of Lemma \ref{cyclicq}. Fix
  a lift $\mu$ of the multiplier character
  $\bar{\mu}= \br \pmod{G^{\mr{der}}}$. Moreover assume that for all
  $v \in {\mc{S}}$ there are lifts $\rho_v \colon \gal{F_v} \to G(\mc{O})$
  with multiplier $\mu$.
  Let ${\mc{T}} \supset {\mc{S}}$ be the set constructed in
  \S \ref{klrsection} in the discussion preceding Construction
  \ref{makelambda}, and likewise fix unramified lifts $\rho_v$ to
  $G(\mc{O})$ for each $v \in {\mc{T}} \setminus {\mc{S}}$, such that
  $\rho_v \pmod{\vpi^2}$ is the lift $\lambda_v$ specified in
  Construction \ref{makelambda}.

Then there exists a sequence of finite sets of primes of $F$,
$ {\mc{T}}\subset {\mc{T}}_2 \subset {\mc{T}}_3 \subset \dots \subset {\mc{T}}_{\n} \dots$,
and for each ${\n} \geq 2$ a
lift $\rho_{\n} \colon \gal{F, {\mc{T}}_{\n}} \to G(\mc{O}/\vpi^{\n})$ of $\br$ with
multiplier $\mu$, such that  $\rho_{\n} = \rho_{{\n}+1} \pmod{\vpi^{\n}}$. Furthermore:
\begin{enumerate}
\item If $w \in {\mc{T}}_{\n} \setminus {\mc{T}}$ is ramified in $\rho_{\n}$, then there is a split maximal torus and root $(T_w, \alpha_w)$ 
such that $\rho_{\n}(\sigma_w) \in T_w(\mc{O}/\vpi^{\n})$, $\alpha_w(\rho_{\n}(\sigma_w))= N(w) \pmod{\vpi^{\n}}$, and $\rho_{\n}|_{\gal{F_w}} \in \Lift^{\mu, \alpha_w}_{\br}(\mc{O}/\vpi^{\n})$; and one of the following two properties holds:
\begin{enumerate}
\item For some $s \leq e$,
  $\rho_s(\tau_w)$ is equal to a non-trivial element of
  $U_{\alpha}(\mc{O}/\vpi^s)$, and for all $n \geq s$
  $\rho_n|_{\gal{F_w}}$ is $\widehat{G^{\mr{der}}}(\mc{O})$-conjugate to the
  reduction mod $\vpi^n$ of a fixed lift $\mc{O}$-lift $\lambda_w$ of
  $\rho_s|_{\gal{F_w}}$. We can choose $\lambda_w$ to be constructed as in Lemma
  \ref{auxlift} to be a formally smooth point of an irreducible
  component of the generic fiber of the local lifting ring.
\item For some $s \geq e$, in fact always 1 or 2 if $e=1$ and $e$ if $e \geq 2$, $\rho_s$ is trivial mod center on $\gal{F_w}$, while $\alpha_w(\rho_{s+1}(\sigma_w)) \equiv N(w) \not \equiv 1 \pmod{\vpi^{s+1}}$, and $\beta(\rho_{s+1}(\sigma_w)) \not \equiv 1 \pmod{\vpi^{s+1}}$ for all roots $\beta \in \Phi(G^0, T_w)$.
\end{enumerate}
\item For all $v \in {\mc{T}}$, $\rho_{\n}|_{\gal{F_v}}$ is $\wh{G^{\mr{der}}}(\mc{O}/\vpi^n)$-conjugate
  to $\rho_v \pmod{\vpi^{\n}}$.
\item The image $\rho_{\n}(\gal{F})$ contains $\wh{G^{\mr{der}}}(\mc{O}/\vpi^{\n})$.
\end{enumerate}
\end{thm}

\begin{rmk}
\begin{itemize}
\item
There is a minor notational irritation in that at auxiliary primes $w$ we find some split maximal torus and root $(T_w, \alpha_w)$, and then construct our desired local lifts relative to these. We do not literally interpolate the desired local lifts, but rather, at the inductive step passing from $\rho_{n-1}$ to $\rho_n$, only up to $\ker(G^{\mr{der}}(\mc{O}/\vpi^n)\to G^{\mr{der}}(\mc{O}/\vpi^{n-1}))$-conjugacy. Thus technically at the next stage in the induction we work with some conjugate of the original $(T_w, \alpha_w)$. To track this in the notation would be a mess, so we will implicitly assume that our pairs $(T_w, \alpha_w)$ are updated in this fashion at each stage in the induction. The only thing we have to observe is that the relevant local lifting rings are isomorphic, so no assertion about the structure of subsequent local lifts changes; we have previously noted a finer version of this assertion in Remark \ref{whconj}.
\item Note that there is no need to enlarge $\mc{O}$ in the course of this proof once it is, as assumed, large enough that each $\rho_v$ for $v \in {\mc{S}}$ is $G(\mc{O})$-valued. In particular, at all auxiliary primes of ramification $w \in {\mc{T}}_{\n} \setminus {\mc{T}}$, the relevant local lifts are defined over $\mc{O}$ (as is evident from the proof's application of Lemma \ref{trivsmooth} and Lemma \ref{auxlift}).
\end{itemize}
\end{rmk}

\begin{proof}
By Corollary \ref{irrcor}, Assumption \ref{multfree} holds for $\br$, and we can apply the results and techniques of \S \ref{klrsection}. We inductively lift $\br$ to a $\rho_n \colon \gal{F, {\mc{T}}_n} \to G(\mc{O}/\vpi^n)$, for each $n=2,3, \ldots$, at each stage increasing the ramification set ${\mc{T}}_n$. The case $n=2$ is settled by Proposition \ref{doublingprop}, which produces an auxiliary set of primes ${\mc{T}}_2 \supset {\mc{T}}$ and a lift $\rho_2 \colon \gal{F, {\mc{T}}_2} \to G(\mc{O}/\vpi^2)$ such that for each $w \in {\mc{T}}_2 \setminus {\mc{S}}$, one of the following holds:
\begin{itemize}
\item Either $w$ is one of the primes (with corresponding fixed lifts) in ${\mc{T}} \setminus {\mc{S}}$ introduced in the discussion up to and including Construction \ref{makelambda}, so that $w$ is a trivial prime, and the lift $\rho_2|_{\gal{F_w}}$ is unramified;
\item or $w$ is one of the primes introduced in the proof of
  Proposition \ref{doublingprop}. That is, $w$ splits in $K'$, $N(w)
  \equiv 1 \pmod{\vpi^e}$ but $N(w) \not \equiv 1 \pmod{\vpi^{e+1}}$,
  and there is a pair $(T_w, \alpha_w)$ for which $\rho_2|_{\gal{F_w}} \in \Lift^{\mu, \alpha}_{\br|_{\gal{F_w}}}(\mc{O}/\vpi^2)$, and moreover $\rho_2(\sigma_w) \in T_w(\mc{O}/\vpi^2)$ is 
\begin{itemize}
\item trivial modulo center if $e \geq 2$; and 
\item satisfies $\beta(\rho_2(\sigma_w)) \not \equiv 1 \pmod{\vpi^2}$ for all roots $\beta$ if $e=1$. 
\end{itemize}
(This precise conclusion follows from combining Lemma \ref{generalposition} with Proposition \ref{doublingprop}.)
\end{itemize}

We now carry out the induction step, showing how, for $n \geq 3$, to
pass from a lift $\rho_{n-1} \colon \gal{F, {\mc{T}}_{n-1}} \to
G(\mc{O}/\vpi^{n-1})$ to $\rho_n$.
For each $w \in {\mc{T}}_{n-1}$, choose a lift $\lambda_w$ of
$\rho_{n-1}|_{\gal{F_w}}$ to $G(\mc{O}/\varpi^n)$ as follows:
\begin{itemize}
  \item If $w \in {\mc{T}}$ let $\lambda_w$ be the reduction of
    $\rho_w$ (from the statement of the Theorem) modulo $\varpi^n$.
    \item If $w \in {\mc{T}}_{n-1} \setminus {\mc{T}}$ and $\rho_{n-1}|_{\gal{F_w}}$
      is unramified then let $\lambda_w$ be any
      unramified (multiplier $\mu$) lift.
    \item 
      Otherwise, choose $\lambda_w$ satisfying:
     \begin{itemize}
\item $\lambda_w \in \Lift_{\br|_w}^{\mu, \alpha_w}(\mc{O}/\vpi^n)$ (recall from Lemma \ref{trivsmooth} that this functor is formally smooth).
\item If $n-1 \leq e$, and $w \in {\mc{T}}_{n-1}\setminus {\mc{T}}_{n-2}$ (i.e., $w$
  was introduced at the previous step of the induction), then we apply
  Lemma \ref{auxlift} to produce a lift $\lambda_w \in \Lift^{\mu,
    \alpha_w}_{\br}(\mc{O})$ that defines a formally smooth point on
  the corresponding generic fiber of the lifting ring, and moreover
  has $\lambda_w(\sigma_w) \in T_w(\mc{O})$ (see the proof of the
  Lemma). Here and at all later stages of the induction we use the
  reduction $\lambda_w \pmod{\vpi^n}$ of this fixed lift (or some
  $\widehat{G}$-conjugate). If $w \in {\mc{T}}_{n-2}$, then by this
  construction we have already specified the desired lift of
  $\rho_{n-1}|_{\gal{F_w}}$.
\item If $n-1 \geq e+1$, and $w \in {\mc{T}}_{n-1} \setminus {\mc{T}}_{n-2}$, then
  by induction $\rho_{n-1}(\sigma_w) \in T_w(\mc{O}/\vpi^{n-1})$,
  $\alpha_w(\rho_{n-1})(\sigma_w)= N(w) \pmod{\vpi^{n-1}}$,
  $\rho_{n-1}(\tau_w) \in U_{\alpha}(\mc{O}/\vpi^{n-1})$, and we have
  one of the following two cases:
\begin{itemize}
\item If $e=1$: if $n-1 \geq 3$, then $\rho_{n-1}|_{\gal{F_w}}$ is trivial (mod center) modulo $\vpi^2$, and $\beta(\rho_{n-1}(\sigma_w)) \not \equiv 1 \pmod{\vpi^3}$ for all $\beta \in \Phi(G^0, T_w)$ (in particular, $N(w) \not \equiv 1 \pmod{\vpi^3}$ but is trivial modulo $\vpi^2$); and if $n-1=2$, then $\rho_{n-1}|_{\gal{F_w}}$ is in general position and has already been described. We choose some $\lambda_w \pmod{\vpi^n}$ that continues to satisfy the conditions of item (1b) of the Theorem statement (including the conditions common to (1a) and (1b)).
\item If $e>1$: $\rho_{e}|_{\gal{F_w}}$ is trivial (mod center), and $\rho_{e+1}|_{\gal{F_w}}$ is in general position: $\beta(\lambda_w(\sigma_w)) \not \equiv 1 \pmod{\vpi^{e+1}}$ for all $\beta \in \Phi(G^0, T_w)$ (in particular, $N(w) \not \equiv 1 \pmod{\vpi^{e+1}}$ but is trivial modulo $\vpi^e$). Again we choose $\lambda_w \pmod{\vpi^n}$ to retain all the desired properties.
\end{itemize}
\end{itemize}
\end{itemize}

We also enlarge ${\mc{T}}_{n-1}$ by a finite set of primes split in
$K(\rho_{n-1}(\fgder))$ and introduce at these $w$ any unramified
$\lambda_w \colon \gal{F_w} \to G(\mc{O}/\vpi^n)$ with multiplier
$\mu$ (which we may assume trivial by imposing a further splitting
condition) such that $\lambda_w \pmod{\vpi^{n-1}}$ is trivial, and the
elements $\lambda_w(\sigma_w)$ generate
$\ker(G^{\mr{der}}(\mc{O}/\vpi^n) \to
G^{\mr{der}}(\mc{O}/\vpi^{n-1})$.\footnote{This is the analogue of the
  enlargement of the initial set ${\mc{T}}$ in Construction
  \ref{makelambda}. Both there and here, the only reason for
  introducing these primes is that if, without them, the localization
  map $\Psi_{\mc{T}}$ defined after Construction \ref{makelambda} had image
  containing $z_{\mc{T}}$, then the global lifts we produce wouldn't
  automatically have image containing
  $\wh{G^{\mr{der}}}(\mc{O}/\vpi^n)$. If $\Psi_{\mc{T}}$ missed $z_{\mc{T}}$, in
  which case we have to run the argument of Proposition
  \ref{doublingprop}, then this big image condition would be automatic
  from the proof of Proposition \ref{doublingprop}, namely from the
  fact that the root spaces $\fg_{\alpha_i}$ associated with the
  auxiliary \v{C}ebotarev sets $\mc{C}_i$ satisfy
  $\sum_{i \in N} \Fp[\gal{F}]\fg_{\alpha_i}=\fgder$, and the image of
  $\rho_n$ would by construction contain each
\[
\fg_{\alpha_i} \subset \fgder \cong \ker\left(\wh{G^{\mr{der}}}(\mc{O}/\vpi^n) \to \wh{G^{\mr{der}}}(\mc{O}/\vpi^{n-1}) \right).
\]
We additionally remark that when $e=1$ the fact that $\im(\rho_{n-1})$
contains $\wh{G^{\mr{der}}}(\mc{O}/\vpi^{n-1})$ formally implies the
corresponding statement for $\im(\rho_n)$, as noted in the proof of
Corollary \ref{infiniteram}.} We  denote this enlarged set by ${\mc{T}}_{n-1}'$.

Since there are no local obstructions to lifting $\rho_{n-1}$, and $\Sha^2_{{\mc{T}}_{n-1}'}(\gal{F, {\mc{T}}_{n-1}'}, \br(\fgder))=0$ (by global duality and the vanishing of $\Sha^1_{{\mc{T}}_{n-1}'}(\gal{F, {\mc{T}}_{n-1}'}, \br(\fgder)^*)$), there is some lift $\rho'_n \colon \gal{F, {\mc{T}}_{n-1}'} \to G(\mc{O}/\vpi^n)$ of multiplier $\mu$. We wish to correct $\rho'_n$ to match the $\{\lambda_w\}_{w \in {\mc{T}}_{n-1}'}$ locally, and so we again apply the method of Proposition \ref{doublingprop}. Let $z_{{\mc{T}}_{n-1}'}=(z_w)_{w \in {\mc{T}}_{n-1}'} \in \bigoplus_{w \in {\mc{T}}_{n-1}'} H^1(\gal{F_w}, \br(\fgder))$ be the collection of cohomology classes such that $(1+\vpi^{n-1}z_w)\rho'_n$ is equivalent to $\lambda_w$ for all $w \in {\mc{T}}_{n-1}'$; that is, we can choose representative cocycles so that these two lifts of $\rho_{n-1}|_{\gal{F_w}}$ are actually equal. 

    We will now apply Proposition \ref{klrstep1''} and Lemma \ref{gens} after replacing the set ${\mc{T}}$ by ${\mc{T}}_{n-1}'$ and the field $K'$ by its unramified-outside-${\mc{T}}_{n-1}'$ analogue: that is, $K'_{n-1}$ is the composite of abelian $p$-extensions $M/K$ that are Galois over $F$, unramified outside ${\mc{T}}_{n-1}'$, and have $\Gal(M/K)$ isomorphic as $\gal{F}$-module to one of the simple factors $W_j$ of $\br(\fgder)$.

Following the notation of Proposition \ref{klrstep1''} and Lemma \ref{gens}, we let $L=K(\rho_{n-1}(\fgder))K'_{n-1}$, $c= \lceil \frac{n}{e} \rceil$, and we choose $q_i \in (\Z/p^c)^\times$ such that compatibly with the condition $\alpha_i(\rho_{n-1}(g_{L/K'_{n-1}, i})) \equiv q_i \pmod{p^c}$ we may take $\rho_{n-1}(g_{L/K'_{n-1}, i}) \pmod{Z_{G^0}}$ equal to:
\begin{itemize}
\item for $n-1 \leq e$, the trivial class; 
\item for $n-1 \geq e+1$, an element $t_{n-1,i} \in T_i(\mc{O}/\vpi^{n-1})$ that is trivial modulo $\vpi^{\max\{e, 2\}}$, in general position modulo $\vpi^{\max\{e, 2\}+1}$, and satisfies $\alpha(t_{n-1, i}) \equiv q_i \pmod{p^c}$. (We do this as in Lemma \ref{generalposition}.)
\end{itemize} 
We now make a series of linear disjointness observations that show we
may apply Proposition \ref{klrstep1''} and Lemma \ref{gens} with the
above setup. First we note that the intersection of $K$ with the
cyclotomic $\Z_p$-extension of $F$ is trivial by Lemma \ref{cyclicq}
and our assumptions that $\br$ is absolutely irreducible and $p \gg_G
0$; thus $K \cap F(\mu_{p^\infty})= F(\mu_p)$. The field $L$ is linearly disjoint over $K$ from the composite of $K(\mu_{p^{\infty}})$ and any composite of extensions $K_{\psi}$, $\psi \in H^1(\gal{F, {\mc{T}}_{n-1}'}, \br(\fgder)^*)$: indeed, we can filter $\Gal(L/K)$ by normal subgroups with successive quotients that are $\gal{F}$-modules having no constituent in common with $\br(\fgder)^*$ or the trivial representation. We can therefore apply Proposition \ref{klrstep1''} and Lemma \ref{gens} with any choice of classes $q_i \in \ker((\Z/p^c)^\times \to (\Z/p)^\times)$. To show that the above choices of $g_{L/K'_{n-1}, i}$ are allowable, we have to compute $\Gal(L/K'_{n-1})$. We may inductively assume (this is something we will have to check persists at each step) that $\im(\rho_{n-1})$ contains $\wh{G^{\mr{der}}}(\mc{O}/\vpi^{n-1})$, so that $\Gal(K(\rho_{n-1}(\fgder))/K)$ is isomorphic to $\wh{G^{\mr{der}}}(\mc{O}/\vpi^{n-1})$. The abelianization of the latter group is isomorphic to $\wh{G^{\mr{der}}}(\mc{O}/\vpi^2)$ (this is proven by induction, the key point being that $\fgder= [\fgder, \fgder]$), and since $K'_{n-1}/K$ is abelian, we see that $K'_{n-1} \cap K(\rho_{n-1}(\fgder))= K(\rho_2(\fgder))$. Thus the image under $\Ad(\rho_{n-1})$ of $\Gal(L/K'_{n-1})$ contains $\ker(G^{\mr{der}}(\mc{O}/\vpi^{n-1}) \to G^{\mr{der}}(\mc{O}/\vpi^2))$, which suffices to find the desired $t_{n-1, i}$. 



The proof of Proposition \ref{doublingprop} now applies \textit{mutatis mutandis}. We obtain an indexing set $N$, \v{C}ebotarev sets $\mc{C}_i$ for $i \in N$ from Proposition \ref{klrstep1''} and Lemma \ref{gens} (with the setup of the previous paragraph), a class $h^{\mr{old}} \in H^1(\gal{F, {\mc{T}}_{n-1}'}, \br(\fgder))$, and, for each $N$-tuple $\un{v}=(v_i)_{i \in N} \in \prod_{i \in N} \mc{C}_i$, classes $h^{(v_i)} \in H^1(\gal{F, {\mc{T}}_{n-1}' \cup v_i}, \br(\fgder))$ satisfying
\[
z_{{\mc{T}}_{n-1}'}= h^{\mr{old}}|_{{\mc{T}}_{n-1}'} + \sum_{i \in N} h^{(v_i)}|_{{\mc{T}}_{n-1}'}
\] 
and the various conclusions of Proposition \ref{klrstep1''} and Lemma \ref{gens}.

We as before fix a positive upper-density subset (strictly speaking we work with the primes of $K'_{n-1}$ specified by the choice of decomposition groups in Proposition \ref{klrstep1''} and Lemma \ref{gens}) $\mf{l} \subset \prod_{i \in N} \mc{C}_i$ such that the $N$-tuples $(\sum_{i \in N} h^{(v_i)}(\sigma_{v_j}))_{j \in N}$, $(h^{\mr{old}}(\sigma_{v_j}))_{j \in N}$, $(h^{(v_i)}(\tau_{v_i}))_{i \in N}$, and $(\rho'_n(\sigma_{v_i}))_{i \in N}$ are independent of $\un{v} \in \mf{l}$ ($\rho'_n(\sigma_{v_i})$ here depends on the specified place of $L$, not just of $K'_{n-1}$). Then having fixed a $\un{v} \in \mf{l}$, we as in Proposition \ref{doublingprop} define a \v{C}ebotarev condition $\mf{l}_{\un{v}}$ in the composite over $K'_{n-1}$ of the extensions $L(\mu_{p^c})$, $K_{h^{(v_m)}}$, $K_{\eta^{(v_m)}_{i, j}}$, assuring that we can satisfy analogues of Equations (\ref{doubling}) (for tuples moreover in $\mf{l}$) and (\ref{doubling2}), and agreeing with the conditions induced by $\mc{C}= \prod_{i \in N} \mc{C}_i$ on the intersection of the respective fields in which the conditions are defined, namely $L(\mu_{p^c})$. The existence of this non-trivial condition $\mf{l}_{\un{v}}$ follows as before using that $K_{h^{(v_m)}}K'_{n-1}/K'_{n-1}$ is non-trivial and totally ramified at $v_m$, and likewise for $K_{\eta^{(v_m)}_{i, j}}K'_{n-1}/K'_{n-1}$; we use these observations as before, and also to note that any composite of such fields is linearly disjoint over $K'_{n-1}$ from $L(\mu_{p^c})$ (which is ramified only above ${\mc{T}}_{n-1}'$).

The limiting argument now works as in Proposition \ref{doublingprop}, and we can find $\un{v}, \un{v}' \in \mf{l}$ satisfying Equations (\ref{doubling}) and (\ref{doubling2}) with any choice of tuples $(C_m)_{m \in N} \in (\fgder)^N$ and $(C_m')_{m \in N} \in (\fgder)^N$. 
To conclude the proof of the theorem, we specify these values so that the modified lift $\rho_n= (1+\vpi^{n-1}h)\rho'_{n}$, where $h= h^{\mr{old}}-\sum_{i \in N}h^{(v_i)} + 2\sum_{i \in N} h^{(v'_i)}$, satisfies all of the local requirements of the Theorem. Since we have arranged $h|_{{\mc{T}}_{n-1}'}=z_{{\mc{T}}_{n-1}'}$, this new $\rho_n$ is, at each $w \in {\mc{T}}_{n-1}'$, $\ker(G(\mc{O}/\vpi^n) \to G(\mc{O}/\vpi^{n-1}))$-conjugate to the specified lift $\lambda_w$. We can choose the tuples $(C_m)_{m \in N}$ and $(C_m')_{m \in N}$ to arrange that $\rho_n(\sigma_{v_i})$ and $\rho_n(\sigma_{v'_i})$ are any desired lifts of the values previously specified for $\rho_{n-1}(g_{L/K'_{n-1}, i})$; it is clear that we can make these choices to satisfy the conclusions (1) and (2) of the Theorem, setting ${\mc{T}}_n= {\mc{T}}_{n-1}' \cup \{v_i\}_{i \in N} \cup \{v'_j\}_{j \in N}$. Part (3) follows inductively since when we enlarged ${\mc{T}}$ at the beginning of the inductive step, we ensured that for some subset of ${\mc{T}}$ the values $\rho_n(\sigma_w)=\lambda_w(\sigma_w)$ generate $\ker(G^{\mr{der}}(\mc{O}/\vpi^n) \to G^{\mr{der}}(\mc{O}/\vpi^{n-1}))$. 
\end{proof}

Letting ${\n}$ approach $\infty$ in Theorem \ref{klr^N} immediately yields a generalization to any reductive group of the main theorem of \cite{klr}. Theorem \ref{klr^N} is more precise than what is needed for this particular application, and here we sketch how to prove it somewhat more generally:
\begin{cor}\label{infiniteram}
Let $\br \colon \gal{F, {\mc{S}}} \to G(k)$ satisfy Assumption \ref{multfree}, except we do not require that $K$ does not contain $\mu_{p^2}$. (In particular, the results of Appendix \ref{groupsection} will show that for $p \gg_G 0$, it suffices here to assume $\br|_{\gal{\tF(\zeta_p)}}$ is absolutely irreducible, and $[\tF(\zeta_p):\tF]>a_G$, for the integer $a_G$ arising in Lemma \ref{cyclicq}.) Fix a lift $\mu \colon \gal{F, {\mc{S}}} \to G/G^{\mr{der}}(\mc{O})$ of $\bar{\mu}= \br \pmod{G^{\mr{der}}}$, and assume that for all $v \in {\mc{S}}$, there are lifts $\rho_v \colon \gal{F_v} \to G(\mc{O})$ of $\br|_{\gal{F_v}}$ with multiplier $\mu$. Then there exists an infinitely ramified lift
\[
\xymatrix{
& G(\mc{O}) \ar[d] \\
\gal{F} \ar@{-->}[ur]^{\rho} \ar[r]_{\br} & G(k)
}
\]
such that $\rho|_{\gal{F_v}}= \rho_v$ modulo $\wh{G^{\mr{der}}}(\mc{O})$-conjugacy for all $v \in {\mc{S}}$, and $\rho(\gal{F})$ contains $\wh{G^{\mr{der}}}(\mc{O})$.
\end{cor}
\begin{rmk}
In this degree of generality, it is not known, but certainly expected, that local lifts $\rho_v$ as above always exist.
\end{rmk}
\begin{proof}
The proof is as in Theorem \ref{klr^N}, except it is no longer necessary to track the local structure at auxiliary primes of our lifts as carefully, and in particular we can work with somewhat weaker hypotheses (compare Remark \ref{infinitetriv}). 
For any $G(\mc{O})$-valued
representation $\lambda$, write $\lambda_n$ for its reduction modulo
$\vpi^n$. 
Given a lift $\rho_{n-1} \colon \gal{F, {\mc{T}}_{n-1}} \to G(\mc{O}/\vpi^{n-1})$ such that
\begin{itemize}
\item for all $v \in {\mc{T}}$, $\rho_{n-1}|_{\gal{F_v}}$ is $\wh{G^{\mr{der}}}(\mc{O}/\vpi^{n-1})$-conjugate to the reduction $\rho_{v, n-1}$ of the fixed lift $\rho_v$;
\item for all $v \in {\mc{T}}_{n-1} \setminus {\mc{T}}$, $\rho_{n-1}|_{\gal{F_v}}$ is either unramified or belongs to some $\Lift_{\br|_{v}}^{\mu, \alpha_v}(\mc{O}/\vpi^{n-1})$;
\item and $\rho_{n-1}(\gal{F})$ contains $\wh{G^{\mr{der}}}(\mc{O}/\vpi^{n-1})$;
\end{itemize}
we sketch how to carry out the induction step and construct a lift $\rho_n$ with the analogous properties. Enlarge ${\mc{T}}_{n-1}$ to ${\mc{T}}'_{n-1}$ as in the proof of Theorem \ref{klr^N} (these are the primes that ensure maximal image of $\rho_n$). Then note that without assuming $\mu_{p^2}$ is not contained in $K$, Proposition \ref{klrstep1''} and Lemma \ref{gens} can be modified to allow:
\begin{itemize}
\item $L=K(\rho_{n-1}(\fgder))K'_{n-1}$;
\item all $g_{L/K'_{n-1}, i}$ equal to 1;
\item and all classes $q_i$ trivial, with $c$ taken to be an integer greater than or equal to $\frac{n}{e}$.
\end{itemize}
Indeed, the proof of \textit{loc.~cit.~}works as before because the (split) conditions on $g_{L/K'_{n-1}, i}$ and $q_i$ are of course compatible, and $L$ is linearly disjoint over $K'_{n-1}$ from any composite of $K'_{n-1}$, $K(\mu_{p^\infty})$, and $K_{\psi}$'s (for $\psi \in H^1(\gal{F, {\mc{T}}_{n-1}'}, \br(\fgder)^*)$) by the same maximal image argument in Theorem \ref{klr^N}. Then the proof of Theorem \ref{klr^N} works as before, except we arrange that the new lift $\rho_n$ and the primes $\{v_i, v'_i\}_{i \in N}$ introduced in the induction step have $\rho_n(\sigma_{v_i})$ and $\rho_n(\sigma_{v'_i})$ trivial mod center; since by construction $q_i \equiv 1 \pmod {\vpi^n}$, the restrictions $\rho_n|_{\gal{F_{v_i}}}$ (and likewise $\rho_n|_{\gal{F_{v'_i}}}$) do indeed belong to $\Lift^{\mu, \alpha_i}_{\br|_{v_i}}(\mc{O}/\vpi^n)$ for the corresponding root $\alpha_i$. This suffices to continue lifting inductively, since we know from Lemma \ref{trivsmooth} that this local lifting functor is formally smooth (regardless of congruence conditions on $q_i$). We have thus
constructed $\rho= \varprojlim_n \rho_n$ having the desired properties. (Note that the conclusion that $\rho(\gal{F})$ contains $\wh{G^{\mr{der}}}(\mc{O})$, clear from the construction, is in fact automatic once it has been checked modulo $\vpi^{N_1}$ for some sufficiently large $N_1$: see Lemma \ref{frattini}.)

\end{proof}

\section{Relative deformation theory}\label{relativeliftsection}
In this section we explain the relative deformation theory method and prove our main theorem, Theorem \ref{mainthm}.

\subsection{Relative Selmer groups}
We need some preliminaries before proceeding to the heart of the argument.
Let $F$ be a number field and ${\mc{S}}$ a finite set of primes of $F$. Let
$n \geq 1$ be an integer and let
$\rho_n:\Gamma_F \to G(\mc{O}/\varpi^n)$ be a continuous
homomorphism. Throughout this section, when we have an integer $r < n$, we will write $\rho_{r}$ for the reduction $\rho_n\pmod{\vpi^{r}}$. For each prime $v \in {\mc{S}}$ 
we assume that for $0 < r \leq n$ we have subgroups ${{\z}}_{r,v} \subset Z^1(\Gamma_{F_v}, \rho_r(\fgder))$ such that
\begin{itemize}
\item Each ${{\z}}_{r, v}$ contains the group of boundaries $B^1(\Gamma_{F_v}, \rho_r(\fgder))$.
\item As $r$ varies, the inclusion and reduction maps induce short exact sequences
\[
0 \to {{\z}}_{a, v} \to {{\z}}_{a+b, v} \to {{\z}}_{b, v} \to 0
\]
as in the cases of Lemma \ref{extracocycles^N} or Proposition \ref{prop:gensmooth}.
\end{itemize}
We let $L_{r,v} \subset H^1(\Gamma_{F_v}, \rho_r(\fgder))$ be the image of ${{\z}}_{r,v}$
and we let $L_{r,v}^{\perp} \subset H^1(\Gamma_{F_v},
\rho_r(\fgder)^*)$ be the annihilator of $L_{r,v}$ under the local duality pairing.

We define the Selmer group \index[terms]{H@ $H^1_{\mc{L}_r}(\Gamma_{F,{\mc{S}}},\rho_r(\fgder))$} $H^1_{\mc{L}_r}(\Gamma_{F,{\mc{S}}},
\rho_r(\fgder))$ to be
\[
\ker \left ( H^1(\Gamma_{F,{\mc{S}}}, \rho_r(\fgder)) \to \bigoplus_{v \in {\mc{S}}}
\frac{H^1(\Gamma_{F_v}, \rho_r(\fgder))}{L_{r,v}} \right )
\]
and we define the dual Selmer group \index[terms]{H@$H^1_{\mc{L}_r^{\perp}}(\Gamma_{F,{\mc{S}}}, \rho_r(\fgder)^*)$}
$H^1_{\mc{L}_r^{\perp}}(\Gamma_{F,{\mc{S}}}, \rho_r(\fgder)^*)$
analogously.

\begin{lemma} \label{lem:exseq}
  For any $a,b$ such that $0 < a,b$ and $a+b \leq n$ there are exact
  sequences
  \[
    H^1_{\mc{L}_a}(\Gamma_{F,{\mc{S}}},
    \rho_a(\fgder)) \to H^1_{\mc{L}_{a+b}}(\Gamma_{F,{\mc{S}}},
    \rho_{a+b}(\fgder)) \to H^1_{\mc{L}_b}(\Gamma_{F,{\mc{S}}},
    \rho_b(\fgder))\]
  and
   \[
    H^1_{\mc{L}_a^{\perp}}(\Gamma_{F,{\mc{S}}},
    \rho_a(\fgder)^*) \to H^1_{\mc{L}_{a+b}^{\perp}}(\Gamma_{F,{\mc{S}}},
    \rho_{a+b}(\fgder)^*) \to H^1_{\mc{L}_b^{\perp}}(\Gamma_{F,{\mc{S}}},
    \rho_b(\fgder)^*) \ .\]
\end{lemma}
\begin{proof}
By assumption, the exact sequence
  \begin{equation} \label{eq1}
    0 \to \rho_a(\fgder) \to \rho_{a+b}(\fgder) \to \rho_b(\fgder) \to
    0
    \end{equation}
    gives rise to an exact sequence
  \begin{equation} \label{eq1'}
    0 \to {{\z}}_{a,v} \to {{\z}}_{a+b,v} \to {{\z}}_{b,v} \to 0. 
  \end{equation}
 
  Suppose $\phi \in Z^1(\gal{F, {\mc{S}}}, \rho_{a+b}(\fgder))$ represents an element of $H^1_{\mc{L}_{a+b}}(\gal{F, {\mc{S}}}, \rho_{a+b}(\fgder))$ that maps to zero in $H^1(\gal{F, {\mc{S}}}, \rho_{b}(\fgder))$. We may replace $\phi$ by a cohomologous cocycle $\phi'$ that maps to zero in $Z^1(\gal{F, {\mc{S}}}, \rho_b(\fgder))$; since ${{\z}}_{a+b, v}$ contains all coboundaries, we still have $\phi'|_{\gal{F_v}} \in {{\z}}_{a+b, v}$ for all $v \in {\mc{S}}$. Moreover, the commutative diagram of exact sequences
  \[
  \xymatrix{0 \ar[r] & Z^1(\gal{F, {\mc{S}}}, \rho_a(\fgder)) \ar[r] \ar[d] & Z^1(\gal{F, {\mc{S}}}, \rho_{a+b}(\fgder)) \ar[r] \ar[d] & Z^1(\gal{F, {\mc{S}}}, \rho_b(\fgder)) \ar[d] \\
  0 \ar[r] & \bigoplus_{v \in {\mc{S}}} Z^1(\Gamma_{F_v}, \rho_a(\fgder)) \ar[r] & \bigoplus_{v \in {\mc{S}}} Z^1(\Gamma_{F_v}, \rho_{a+b}(\fgder)) \ar[r] & \bigoplus_{v \in {\mc{S}}} Z^1(\gal{F_v}, \rho_b(\fgder))}
  \]
  implies that there is a unique cocycle $\phi_a \in Z^1(\gal{F, {\mc{S}}}, \rho_a(\fgder))$ mapping to $\phi'$ and (by Equation (\ref{eq1'})) unique cocycles $\xi_{a, v} \in {{\z}}_{a, v} \subset Z^1(\gal{F_v}, \rho_a(\fgder))$ mapping to $\phi'|_{\gal{F_v}}$. By commutativity and the uniqueness, $\phi_a|_{\gal{F_v}}= \xi_{a, v} \in {{\z}}_{a, v}$ for all $v \in {\mc{S}}$, and the exactness of the first Selmer group sequence follows. Exactness of the sequence for dual Selmer groups (with the roles of $a$ and $b$ reversed) follows in a similar way, replacing the ${{\z}}_{r, v}$ with the preimages ${\z}^\perp_{r, v}$ of $L_{r, v}^\perp$ in $Z^1(\gal{F_v}, \rho_r(\fgder)^*)$: then we have an exact sequence
  \[
  0 \to {\z}^\perp_{b, v} \to {{\z}}_{a+b, v}^\perp \to {{\z}}_{a, v}^\perp,
  \]
  where the maps exist by functoriality properties of the local duality pairings, and exactness additionally uses the surjectivity of ${{\z}}_{a+b, v} \to {{\z}}_{b, v}$.
  \end{proof}

The basic object we will study in what follows is the relative (dual) Selmer group:
\begin{defn}\label{def:relativeSelmer}
For $0 < r \leq n$, we define the $r$-th relative Selmer
group \index[terms]{H@$\ov{H^1_{\mc{L}_r}(\Gamma_{F,{\mc{S}}},
    \rho_r(\fgder))}$} to be
\[
\ov{H^1_{\mc{L}_r}(\Gamma_{F,{\mc{S}}}, \rho_r(\fgder))} := \im \left ( H^1_{\mc{L}_r}(\Gamma_{F,{\mc{S}}}, \rho_r(\fgder)) \to H^1_{\mc{L}_1}(\Gamma_{F,{\mc{S}}}, \br(\fgder)) \right )
\]
and the $r$-th relative dual Selmer group \index[terms]{H@$\ov{H^1_{\mc{L}_r^{\perp}}(\Gamma_{F,{\mc{S}}}, \rho_r(\fgder)^*)}$} be
\[
\ov{H^1_{\mc{L}_r^{\perp}}(\Gamma_{F,{\mc{S}}}, \rho_r(\fgder)^*)} := \im \left ( H^1_{\mc{L}_r^{\perp}}(\Gamma_{F,{\mc{S}}}, \rho_r(\fgder)^*) \to H^1_{\mc{L}_1^{\perp}}(\Gamma_{F,{\mc{S}}}, \br(\fgder)^*) \right ).
\]
Given an element $\phi$ in a modulo $\vpi^r$ (dual) Selmer group, we will write $\bar{\phi}$ for its image in the corresponding modulo $\vpi$ (dual) Selmer group.

In addition, we say that the local conditions $\mc{L}_r$, for $0 < r \leq n$, are \emph{balanced}\index[terms]{balanced} if 
\[
 |H^1_{\mc{L}_r}(\Gamma_{F,{\mc{S}}}, \rho_r(\fgder))| = | H^1_{\mc{L}_r^{\perp}}(\Gamma_{F,{\mc{S}}}, \rho_r(\fgder)^*)|
\] 
for $0 < r \leq n$.
\end{defn}

\begin{lemma} \label{lem:bal} Suppose the local conditions
  $\mc{L}_r=\{L_{r,v}\}_{v \in {\mc{S}}}$, $0 < r \leq n$, are balanced and the spaces of
  invariants $\br(\fgder)^{\Gamma_F}$ and $(\br(\fgder)^*)^{\Gamma_F}$
  are both zero. Then the relative Selmer and dual Selmer groups are
  also balanced, i.e.,
\[
\dim( \ov{H^1_{\mc{L}_n}(\Gamma_{F,{\mc{S}}}, \rho_n(\fgder))}) = \dim(\ov{H^1_{\mc{L}_n^{\perp}}(\Gamma_{F,{\mc{S}}}, \rho_n(\fgder)^*)}).
\]
\end{lemma}
\begin{proof}
 Consider the exact sequences in Lemma \ref{lem:exseq} for the case $a=n-1$ and $b=1$. The vanishing of the invariants implies the first map in each sequence is injective, and then the result follows immediately from exactness and the balanced assumption.
\end{proof}

\subsection{Annihilating the relative (dual) Selmer group}
We now begin to work with a residual representation as in our eventual
theorem. Let $\br \colon \gal{F, {\mc{S}}} \to G(k)$ satisfy the hypotheses
of Theorem \ref{klr^N}. Let $\Gamma$ be the inverse image in
$G^{\mr{ad}}(\mc{O})$ of $\Ad \circ \br(\gal{\tF})
    \subset G^{\mr{ad}}(k)$. We apply Corollary \ref{cor:lazard} to
the $p$-adic Lie group $\Gamma$, the $\Gamma$-module $\Lie(\Gamma) = \fgder$, and the integer $m=1$ to deduce that the image of the reduction map
  \[
  H^1(\Gamma, \fgder \otimes_{\mc{O}} \mc{O}/\vpi^M) \to H^1(\Gamma, \fgder \otimes k)
  \]
  is zero for all $M$ greater than some integer $M_1$ (depending only on $\im(\br)$).
  
We now assume that we have integers $M \geq M_1$ and $N \geq M$ and a
homomorphism $\rho_N \colon \gal{F, {\mc{S}}'} \to G(\mc{O}/\vpi^N)$ lifting
our given $\br$ such that $\im(\rho_N)$ contains
$\wh{G^{\mr{der}}}(\mc{O}/\vpi^N)$ (for instance, any $\rho_N$
produced by an application of Theorem \ref{klr^N}, with ${\mc{S}}'$ being the
enlarged ramification set needed to produce the lift). We may and do
assume that $M$ and $N$ are both divisible by $e$, the ramification index of
$\mc{O}$. For any $1 \leq r \leq N$, we let $F_r$
\index[terms]{F@$F_r$} be the splitting
field $F(\br, \rho_r(\fgder))$. 
As in \S \ref{klrsection}, we let $K=
F(\br, \mu_p)$. We also let 
$F_N^*= F_N(\mu_{p^{N/e}})$. \index[terms]{F@$F_N^*$}
\begin{lemma}\label{lemma:fn}
  Provided $M$ is sufficiently large, in a manner depending only on $\im(\br)$, we have:
\begin{itemize}
 \item $H^1(\Gal(F_N^*/F), \rho_M(\fgder)^*)=0$.
 \item The map $H^1(\Gal(F_N^*/F), \rho_M(\fgder)) \to H^1(\Gal(F_N^*/F), \br(\fgder))$ is zero.
\end{itemize}
\end{lemma}
\begin{proof}
For the first item, it suffices to prove that $H^1(\Gal(F_N^*/F), \br(\fgder)^*) = 0$. 
Consider the inflation-restriction sequence
  \begin{equation} \label{eq:infres} H^1(\Gal(K/F),
    \br(\fgder)^*) \to H^1(\Gal(F_N^*/F),
    \br(\fgder)^*) \to
    H^1(\Gal(F_N^*/K),\br(\fgder)^*)^{\Gal(K/F)}
  \end{equation}
The first term in (\ref{eq:infres}) is zero by Corollary
    \ref{irrcor}, which tells us that all conditions in Assumption
    \ref{multfree} are satisfied under the hypotheses of Theorem
    \ref{klr^N}. The last term is $\Hom_{\gal{F}}(\Gal(F_N^*/K),
    \br(\fgder)^*)$, and since $\im(\rho_N)$ contains
    $\wh{G^{\mr{der}}}(\mc{O}/\vpi^N)$, the maximal abelian, $p$-torsion
    quotient of $\Gal(F_N^*/K)$ is a quotient of (in fact equal to)
    $\br(\fgder) \oplus \Z/p$ (as noted in the proof of Theorem
    \ref{klr^N}). This $\gal{F}$-module admits no non-zero
    $\gal{F}$-invariant maps to $\br(\fgder)^*$, as follows again from
    Corollary \ref{irrcor}, so the last term in (\ref{eq:infres}) is also zero.


For the second item,  let $F_N$ be as above, 
and consider the inflation-restriction exact sequence
\begin{equation*} 
0 \to H^1(\Gal(F_N/F), \rho_M(\fgder)) \to H^1(\Gal(F_N^*/F), \rho_M(\fgder)) \to H^1(\Gal(F_N^*/F_N),\rho_M(\fgder))^{\Gal(F_N/F)}.
\end{equation*}
The right-hand side is zero: since $F(\mu_{p^{N/e}})/F$ is abelian, $\Gal(F_N/F)$ acts trivially on $\Gal(F_N^*/F_N)$, and, again by Corollary \ref{irrcor}, $\rho_M(\fgder)$ has no trivial proper $\gal{F}$-submodule.

We are reduced to proving that the map $H^1(\Gal(F_N/F), \rho_M(\fgder)) \to H^1(\Gal(F_N/F), \br(\fgder))$ is zero if $ N \geq M \geq M_1$. Since $[F_N: \tF(\rho_N(\fgder))]$ and $[\tF:F]$ are coprime to $p$, inflation-restriction reduces us to the same claim with $\ov{\Gamma}:= \Gal(\tF(\rho_N(\fgder))/\tF)$ in place of $\Gal(F_N/F)$. Let $\Gamma$ as above be the inverse image in $G^{\mr{ad}}(\mc{O})$ of $\ov{\Gamma} \subset G^{\mr{ad}}(\mc{O}/\vpi^N)$.
Using the fact that the inflation map $H^1(\ov{\Gamma}, A) \to H^1(\Gamma, A)$ is injective for any $\ov{\Gamma}$-module $A$, 
the claim then follows from the choice of $M \geq M_1$ produced by Corollary \ref{cor:lazard}.
\end{proof}
\begin{rmk}\label{advanishing}
Note that even if we were to assume $H^1(\Gal(K/F), \br(\fgder))=0$ (a mild assumption for $\br$ absolutely irreducible, by Lemma \ref{irr}), we would not gain vanishing of $H^1(\Gal(F_N^*/F), \br(\fgder))$ (in contrast to the case of $\br(\fgder)^*$ coefficients), because $\br(\fgder)$ is a quotient of $\Gal(F_N^*/K)$.
\end{rmk}

We continue with our fixed lift $\rho_N \colon \gal{F,{\mc{S}}'} \to G(\mc{O}/\vpi^N)$, with reduction $\rho_M$. We define the set of auxiliary primes that we will consider in annihilating the (relative) dual Selmer group:
\begin{defn}\label{reldefaux}
  Let ${\mc{Q}}_N$\index[terms]{Q@$\mc{Q}_N$} 
  be the set of trivial primes $v$ of $F$ satisfying the following properties:
\begin{itemize}
 \item $N(v) \equiv 1 \pmod {\vpi^M}$ but $N(v) \not \equiv 1 \pmod{\vpi^{M+1}}$. (Recall that $M$ is a multiple of $e$.)
 \item $\rho_N|_{\gal{F_v}}$ is unramified (with multiplier $\mu$), and $\rho_M|_{\gal{F_v}}$ is trivial mod center.
 \item There exist a split maximal torus $T$ of $G$ and a root $\alpha \in \Phi(G^0, T)$ such that $\rho_N(\sigma_v) \in T(\mc{O}/\vpi^N)$ and $\alpha(\rho_N(\sigma_v))= \kappa(\sigma_v)=N(v)$.
 \item For all roots $\beta \in \Phi(G^0, T)$, $\beta(\rho_{M+1}(\sigma_v)) \not \equiv 1 \pmod{\vpi^{M+1}}$.
\end{itemize}
\end{defn}
The set ${\mc{Q}}_N$ corresponds to a \v{C}ebotarev condition in the
extension $F_N^*/F$; it is a set of positive density because
$\im(\rho_N)$ contains $\wh{G^{\mr{der}}}(\mc{O}/\vpi^N)$, and (as
before) $K(\rho_N(\fgder))$ and $K(\mu_{p^{\infty}})$ are linearly
disjoint over $K$. For these primes we will consider the functors of
lifts $\Lift^{\mu, \alpha}_{\br}$ and $\Lift^{\mu, \alpha}_{\rho_M}$
as in Definitions \ref{triviallifts} and \ref{relativetriv}.

\begin{lemma}\label{lem:triv2}
Assume $N \geq 4M$. Then for any $v \in {\mc{Q}}_N$, with corresponding $(T, \alpha)$, the following properties hold:
\begin{itemize}
 \item For any $m \geq N-M$ and $1 \leq r \leq M$, the fibers of $\Lift^{\mu, \alpha}_{\rho_M}(\mc{O}/\vpi^{m+r}) \to \Lift^{\mu, \alpha}_{\rho_M}(\mc{O}/\vpi^m)$ are non-empty and stable under $Z^{\alpha}_{r, v}$, where $Z^{\alpha}_{r, v} \subset Z^1(\gal{F_v}, \rho_r(\fgder))$ is the submodule produced in Lemma \ref{extracocycles^N}.
 \item Let $L^{\alpha}_{r, v}$ be the image of $Z^{\alpha}_{r, v}$ in $H^1(\gal{F_v}, \rho_r(\fgder))$, and let $L^{\alpha, \perp}_{r, v}$ be its annihilator in $H^1(\gal{F_v}, \rho_r(\fgder)^*)$ under the local duality pairing. Then
 \begin{itemize}
  \item $|L^{\alpha}_{r, v}|= |H^0(\gal{F_v}, \rho_r(\fgder))|= | H^1_{\mr{unr}}(\gal{F_v}, \rho_r(\fgder))|$.
  \item The inclusion 
\[
L^{\alpha}_{r, v} \cap H^1_{\mr{unr}}(\gal{F_v}, \rho_r(\fgder)) \into H^1_{\mr{unr}}(\gal{F_v}, \rho_r(\fgder))
\]  
has cokernel isomorphic to $\mc{O}/\vpi^r$, and this cokernel is generated by the image of the unramified cocycle that maps $\sigma_v$ to $H_{\alpha}= d(\alpha^\vee)(1)$, the usual coroot element in $\mf{t}^{\mr{der}}$.
\item The inclusion 
\[
L^{\alpha, \perp}_{r, v} \cap H^1_{\mr{unr}}(\gal{F_v}, \rho_r(\fgder)^*) \into H^1_{\mr{unr}}(\gal{F_v}, \rho_r(\fgder)^*)
\]  
has cokernel isomorphic to $\mc{O}/\vpi^r$, and this cokernel is generated by the image of the unramified cocycle that maps $\sigma_v$ to any element of $(\fgder)^*$ whose restriction to $\fg_{\alpha}$ spans the free rank one $\mc{O}/\vpi^r$-module $\Hom_{\mc{O}}(\fg_{\alpha}, \mc{O}/\vpi^r)$.
 \end{itemize}
\end{itemize}
\end{lemma}
\begin{proof}
 For the first point, we simply apply Lemma \ref{extracocycles^N} with $s=M$. For the second point, we note that since $\rho_M$ is trivial at $v$, and $N(v) \equiv 1 \pmod{\vpi^M}$, $H^1_{\mr{unr}}(\gal{F_v}, \rho_r(\fgder))$ consists of all cocycles $\phi^{\mr{un}, r}_X$ for $X \in \rho_r(\fgder)$ (here we use the notation of Lemma \ref{extracocycles^N}, i.e. $\phi^{\mr{un}, r}_X(\sigma_v)= X$). By the proof of Lemma \ref{extracocycles^N}, we see that $L^\alpha_{r, v}$ consists of the unramified cocycles $\phi^{\mr{un}, r}_X$ for all $X \in \ker(\alpha|_{\mf{t}^{\mr{der}}}) \oplus \bigoplus_{\beta \in \Phi(G^0, T)} \fg_{\beta}$, as well as the ramified cocycle $\phi^r_{\alpha}$; indeed, since $\rho_N|_{\gal{F_v}}$ is unramified, the cocycles denoted $\phi^r_{X_\beta}$ in Lemma \ref{extracocycles^N} are equal to the cocycles $\phi^{\mr{un}, r}_{X_\beta}$. It follows immediately that the cokernel of $L^{\alpha}_{r, v} \cap H^1_{\mr{unr}}(\gal{F_v}, \rho_r(\fgder)) \into H^1_{\mr{unr}}(\gal{F_v}, \rho_r(\fgder))$ is the free $\mc{O}/\vpi^r$-module of rank 1 spanned by (the image of) $\phi^{\mr{un}, r}_{H_{\alpha}}$. The claim about the other inclusion follows from this description of $L^{\alpha}_{r, v}$, the fact that $H^1(\gal{F_v}, \rho_r(\fgder))= \Hom(\gal{F_v}, \rho_r(\fgder))$ is isomorphic to two copies of $\rho_r(\fgder)$ (via evaluation at $\sigma_v$ and $\tau_v$), and Lemma \ref{localduality}. 
\end{proof}

The following two central results, Proposition \ref{prop:killrel} and Theorem \ref{thm:killrel}, are our replacements for Ramakrishna's original Selmer-annihilation arguments. We will prove them under the restricted hypothesis that $\fgder$ consists of a single $\pi_0(G)$-orbit of simple factors; in our main theorem, Theorem \ref{mainthm}, we will reduce to this case.
\begin{prop} \label{prop:killrel} 
Assume that $\fgder$ consists of a single $\pi_0(G)$-orbit of simple factors. Let ${\mc{Q}}$ be any finite subset of ${\mc{Q}}_N$, and let $\phi \in H^1_{\mc{L}_{M}}(\Gamma_{F, {\mc{S}}' \cup {\mc{Q}}}, \rho_M(\fgder))$ and $\psi \in H^1_{\mc{L}_{M}^{\perp}}(\Gamma_{F, {\mc{S}}' \cup {\mc{Q}}}, \rho_M(\fgder)^*)$ be such that $0 \neq \ov{\phi} \in \ov{H^1_{\mc{L}_{M}}(\Gamma_{F, {\mc{S}}' \cup {\mc{Q}}}, \rho_M(\fgder))}$ and $0 \neq \ov{\psi} \in \ov{H^1_{\mc{L}_{M}^{\perp}}(\Gamma_{F, {\mc{S}}' \cup {\mc{Q}}}, \rho_M(\fgder)^*)}$. Then there exists a prime $v \in {\mc{Q}}_N$, with associated torus and root $(T, \alpha)$, such that
  \begin{itemize}
  \item $\ov{\psi}|_{\Gamma_{F_v}} \notin L_{1,v}^{\alpha, \perp}$; and
  \item $\phi|_{\gal{F_v}} \not \in L^{\alpha}_{M, v}$.
  \end{itemize}
\end{prop}

\begin{proof}
We first note that both $\phi|_{\gal{F_N^*}}$ and $\ov{\psi}|_{\gal{F_N^*}}$ are non-zero. For the former, consider the diagram
\[
 \xymatrix{
      0 \ar[r]  & H^1(\Gal(F_N^*/F), \rho_M(\fgder)) \ar[r] \ar[d]^0 &
      H^1(\Gamma_{F,{\mc{S}}' \cup {\mc{Q}}}, \rho_{M}(\fgder)) \ar[r] \ar[d] & H^1(\Gamma_{F_N^*},
      \rho_M(\fgder)) \ar[d] \\
      0 \ar[r] & H^1(\Gal(F_N^*/F), \br(\fgder)) \ar[r]  &
      H^1(\Gamma_{F,{\mc{S}}' \cup {\mc{Q}}}, \br(\fgder)) \ar[r]  & H^1(\Gamma_{F_N^*},
      \br(\fgder)), 
}
\]
where the (exact) rows are inflation-restriction sequences, and where
Lemma \ref{lemma:fn} implies the map labeled by
$0$ is zero. If $\phi|_{\gal{F_N^*}}$ were zero, then
$\ov{\phi}$ would be zero, a contradiction. In particular,
$\phi(\gal{F_N^*})$ is a non-zero
$\gal{F}$-stable submodule of
$\rho_M(\fgder)$. Similarly, Lemma \ref{lemma:fn} implies that
$\ov{\psi}|_{\gal{F_N^*}}$ is non-zero. As usual, the fixed fields
$F_N^*(\phi)$ and
$F_N^*(\ov{\psi})$ of these restricted cocycles are then non-trivial
and linearly disjoint extensions of
$F_N^*$. We claim that there is a pair $(T,
\alpha)$ consisting of a split maximal torus $T$ and a root $\alpha
\in \Phi(G^0, T)$ such that
$\phi(\gal{F_N^*})$ is not contained in $\ker(\alpha|_{\mf{t}}) \oplus
\bigoplus_{\beta} \fg_{\beta}$ and such that
$\ov{\psi}(\gal{F_N^*})$ is not contained in the annihilator of
$\fg_{\alpha}$ under local duality.  Granted this claim, we explain
how to finish the proof. Let $\gamma_1 \in
\gal{F}$ be any element such that
$\rho_N(\gamma_1)$ satisfies the conditions (on
$\rho_N(\sigma_v)$ and
$\kappa(\gamma_1)$) of Definition \ref{reldefaux} for
the pair $(T,
\alpha)$ (we have already noted why such elements
    exist).
  By the claim and the linear disjointness, we can choose an element $\gamma_2 \in \Gal(F_N^*(\phi, \ov{\psi})/F_N^*)$ such that 
  \begin{align*}
 & \phi(\gamma_2 \gamma_1)= \phi(\gamma_2) + \phi(\gamma_1) \not \in \ker(\alpha|_{\mf{t}}) \oplus \bigoplus_{\beta} \fg_{\beta},\\
 & \ov{\psi}(\gamma_2 \gamma_1)= \ov{\psi}(\gamma_2) + \ov{\psi}(\gamma_1) \not \in \fg_{\alpha}^{\perp}.
  \end{align*}
  Applying the \v{C}ebotarev density theorem, we take $v$ to be any prime in the positive-density set of primes whose Frobenius elements are equal to the element $\gamma_2 \gamma_1$ in $\Gal(F_N^*(\phi, \ov{\psi})/F)$.
  
To finish the proof, we return to the claim that such a pair $(T,
\alpha)$ exists. First note that the images $\phi(\gal{F_N^*})$ and
$\ov{\psi}(\gal{F_N^*})$ have non-trivial projection to each simple
factor of $\fgder$: this follows from the $\gal{F}$-equivariance of
$\phi|_{\gal{F_N^*}}$ and $\ov{\psi}|_{\gal{F_N^*}}$ and the fact that
$\fgder$ is a single $\pi_0(G)$-orbit of simple factors. We will meet
the desired conditions on $(T, \alpha)$ if and only if we do so after
replacing $\phi(\gal{F_N^*})$ and $\ov{\psi}(\gal{F_N^*})$ by their
projections to a single simple factor of $\fgder$ (namely, the one
containing $\fg_{\alpha}$); thus in the (purely Lie-theoretic)
remainder of the argument we may and do assume that $G$ is connected
and simple. Fix a pair $(T_1, \alpha_1)$. For any $g \in G(\mc{O})$,
we write $(T_g, \alpha_g)$ for $\Ad(g)(T_1, \alpha_1)$. To the non-zero cocycle $\ov{\psi}$ we can associate a proper closed subscheme $Y_{\ov{\psi}}$ of $G_k$ whose functor of points assigns to a $k$-algebra $R$ the set:
 \[
 Y_{\ov{\psi}}(R)= \{g \in G(R): \langle \ov{\psi}(\gal{F_N^*}), \mf{g}_{\alpha_g} \rangle = 0 \}.
 \]
 For notational convenience, we modify the initial choice $(T_1, \alpha_1)$ so that $1 \not \in Y_{\ov{\psi}}$.\footnote{This modification depends on $\ov{\psi}$, but the bound on $p$ we derive will not depend on this.} We let $U_{\ov{\psi}}$ be the open complement $G \setminus Y_{\ov{\psi}}$, and we let $U_{\ov{\psi}, M}$ be the following set (not scheme):
 \[
 U_{\ov{\psi}, M}= \{g \in G(\mc{O}/\vpi^M): g\pmod{\vpi} \in U_{\ov{\psi}}(k)\}.
 \]
 We claim there is a constant $C_G$ depending only on the root datum of $G$ such that for all $p > C_G$ the intersection
 \[
 \bigcap_{g \in U_{\ov{\psi}, M}} \Ad(g) \left( \ker(\alpha_1|_{\mf{t}_1}) \oplus \bigoplus_{\beta \in \Phi(G^0, T_1)} \fg_{\beta} \right)
 \]
 is zero. In what follows, we write $X_g$ for the term in this intersection corresponding to $g$, and for any subset $I$ of $U_{\ov{\psi}, M}$ we write $X_I$ for $\bigcap_{g \in I} X_g$. Granted that $X_{U_{\ov{\psi}, M}}$ is zero, we are done: there is then some $g \in U_{\ov{\psi}, M}$ such that $\phi(\gal{F_N^*})$ is not contained in $X_g$, and $(T_g, \alpha_g)$ is then the desired pair $(T, \alpha)$.
 
We next observe that $X_{\wh{G}(\mc{O}/\vpi^M)}$ is contained in $\vpi^{M-1} \rho_M(\fgder)$: indeed, for any $Z \in \rho_M(\fgder)$ and $1 \leq r \leq M$ such that $Z \not \equiv 0 \pmod{\vpi^r}$, the span $\mc{O}/\vpi^M[\wh{G}(\mc{O}/\vpi^M)] \cdot Z$ contains $\vpi^{r} \rho_M(\fgder)$. If $X_{\wh{G}(\mc{O}/\vpi^M)}$ were not contained in $\vpi^{M-1} \rho_M(\fgder)$, then it would contain $\vpi^{M-1} \rho_M(\fgder)$, which is not possible (for $p \gg_G 0$) for even a single $X_g$. Consequently the entire intersection $X_{U_{\ov{\psi}, M}}$ is zero provided $\bigcap_{g \in U_{\ov{\psi}}(k)} \ov{X}_g$ is zero, where we write $\ov{X}_g$ for $\Ad(g)(X_1 \otimes k)$. It therefore suffices to show that if we fix any two non-zero elements $\ov{A} \in \fgder$ and $\ov{B} \in (\fgder)^*$, then for $p \gg_G 0$ there exists $g \in G(k)$ such that $\Ad(g)^{-1} \ov{A} \not \in \ker(\alpha_{\mf{t}_1}) \oplus \bigoplus_{\beta} \fg_{\beta}$ and $\Ad(g)^{-1} \ov{B} \not \in \fg_{\alpha_1}^\perp$: indeed, taking $\ov{B}$ to be a non-zero element of $\ov{\psi}(\gal{F_N^*})$, the set $U_{\ov{\psi}}(k)$ contains the locus thus associated to $\ov{B}$, so if $\ov{A}$ were a non-zero element of $\cap_{g \in U_{\ov{\psi}(k)}} \ov{X}_g$ we would obtain a contradiction.

Thus we are reduced to showing that if $p \gg_G 0$, then for \textit{any} pair $(\ov{A}, \ov{B})$ as above, 
\[
 \{\text{$g \in G(k): \Ad(g) \ov{A} \not \in \ker(\alpha_{\mf{t}_1}) \oplus \bigoplus_{\beta \in \Phi(G^0, T_1)} \fg_{\beta}$ and $\Ad(g) \ov{B} \not \in \fg_{\alpha_1}^\perp$}\}
\]
is non-empty. The locus in question is the set of $k$-points of the complement $G \setminus (\Phi_{\ov{A}} \cup \Phi_{\ov{B}})$ of two proper closed subschemes $\Phi_{\ov{A}}$ and $\Phi_{\ov{B}}$. By the Bruhat decomposition (\cite[Corollary 14.14]{borel:linalg}), $G$ contains an open subset $X$ that is isomorphic to an explicit open subset, namely $\mathbb{G}_m^{\dim(T)} \times \mathbb{A}^{|\Phi(G, T)|}$, of an affine space $\mathbb{A}^{d_G}$, where $d_G= \dim(G)$, so $|G(k)| \geq (q-1)^{\dim(T)}\cdot q^{|\Phi(G, T)|}$. It suffices to produce a lower bound for the number of $k$-points of $X \setminus \left(X \cap (\Phi_{\bar{A}} \cup \Phi_{\bar{B}})\right)$. From the explicit descriptions of $\Phi_{\ov{A}}$ and $\Phi_{\ov{B}}$, there is an integer $d$ depending only on the root datum of $G$ such that $X \cap \Phi_{\ov{A}}$ and $X \cap \Phi_{\ov{B}}$ are cut out (in $X \subset \mathbb{A}^{d_G}$) by some equations of degree at most $d$: note that if $t_1, \ldots, t_{\dim(T)}$ are the coordinates on the copies of $\mathbb{G}_m$, the equations we at first write down have $t_i^{-1}$ in them, but multiplying by a power of $\prod t_i$ that is independent of $\ov{A}$ and $\ov{B}$ we can assume that the equations do extend to make sense on all of $\mathbb{A}^{d_G}$, and it is at this point that we extract the degree $d$. Therefore $|\left(X \cap (\Phi_{\ov{A}} \cup \Phi_{\ov{B}})\right)(k)|$ is at most $2d \cdot q^{d_G-1}$, since any non-zero $f \in k[x_1, \ldots, x_a]$ of degree $b$ has at most $b q^{a-1}$ solutions. We conclude that
\[
|G \setminus (\Phi_{\ov{A}} \cup \Phi_{\ov{B}})(k)| \geq (q-1)^{\dim(T)} \cdot q^{|\Phi(G, T)|} - 2dq^{d_G-1} \geq \frac{1}{2^{\dim(T)}}q^{d_G}- 2d q^{d_G-1}.
\]
In particular, for $p \gg_G 0$, this complement is non-empty, and the proof is complete. 
\end{proof}
  Recall that we say a submodule $L_{M, v} \subset H^1(\gal{F_v}, \rho_M(\fgder))$ is \textit{balanced} if $|L_{M, v}|$ satisfies part (3) of the conclusion of Proposition \ref{prop:gensmooth}.
\begin{thm} \label{thm:killrel} Continue to assume that $\fgder$ consists of a single $\pi_0(G)$-orbit of simple factors. For any initial set $\mc{L}_M= \{L_{M, v}\}_{v \in {\mc{S}}'}$ of balanced local conditions at primes $v \in {\mc{S}}'$, there exists a finite set ${\mc{Q}} \subset {\mc{Q}}_N$ such that  
  \[
    \ov{H^1_{\mc{L}_{M} \cup \{L_{M, v}\}_{v \in {\mc{Q}}}}(\Gamma_{{\mc{S}}' \cup {\mc{Q}}}, \rho_M(\fgder))} = 0
  \]
  and 
  \[ 
    \ov{H^1_{\mc{L}_{M}^{\perp} \cup \{L_{M, v}^\perp\}_{v \in {\mc{Q}}}}(\Gamma_{{\mc{S}}' \cup {\mc{Q}}}, \rho_M(\fgder)^*)}
    = 0,
  \]
  where for $v \in {\mc{Q}}$ we take $L_{M, v}$ to be the submodule (where for notational simplicity we omit the root $\alpha$) in Lemma \ref{lem:triv2}. 
\end{thm}
\begin{proof}
  We will show that either the relative Selmer group in question is zero, or we can find a prime $v \in {\mc{Q}}_N$ such that 
  \[
  |H^1_{\mc{L}_{M} \cup L_{M, v}}(\Gamma_{{\mc{S}}' \cup \{v\}}, \rho_M(\fgder))|<|H^1_{\mc{L}_{M}}(\Gamma_{{\mc{S}}'}, \rho_M(\fgder))|. 
  \]
Applying this assertion inductively, we either arrive at a trivial relative Selmer group as in the conclusion of the Theorem or, even stronger, we annihilate the entire mod $\vpi^M$ Selmer group. (The reader should note that the logic of the proof is not that at each step we decrease the size of the relative Selmer group.)

By the Greenberg--Wiles formula and the assumption that the local conditions are balanced, if $\ov{H^1_{\mc{L}_{M}}(\Gamma_{{\mc{S}}'}, \rho_M(\fgder))}$ is non-zero, then we can find $\phi$ and $\psi$ as in the hypotheses of Proposition \ref{prop:killrel}. Let $v \in {\mc{Q}}_N$ be any prime as in its conclusion. Since $\ov{\psi}|_{\Gamma_{F_v}}$ is the (non-zero) image of $\psi|_{\Gamma_{F_v}} \in H^1_{\unr}(\Gamma_{F_v}, \rho_M(\fgder)^*)$ in $H^1_{\unr}(\Gamma_{F_v}, \br(\fgder)^*)$, it
follows that $\psi$ must generate a submodule of $H^1_{\unr}(\Gamma_{F_v}, \rho_M(\fgder)^*)$ isomorphic to $\mc{O}/\varpi^M$. The choice of local condition at $v$ (namely, the result of Lemma \ref{lem:triv2}) and the fact that $\ov{\psi}|_{\Gamma_{F_v}} \notin L_{1,v}^{\perp}$ then implies that $\varpi^i \psi|_{\Gamma_{F_v}} \in L_{M,v}^{\perp}$ if and only if $\varpi^i \psi =0$.

Let $L_{M,v}' = L_{M,v} + H^1_{\unr}(\Gamma_{F_v}, \rho_M(\fgder))$, so (by Lemma \ref{lem:triv2}) $L_{M,v}'/L_{M,v} \cong \mc{O}/\varpi^{M}$. Since $(L_{M,v}')^{\perp} = L_{M,v}^{\perp} \cap H^1_{\unr}(\Gamma_{F_v}, \rho_M(\fgder)^*)$, it follows from the previous paragraph that
\[
H^1_{\mc{L}_M^{\perp}}(\Gamma_{F,{\mc{S}}'}, \rho_M(\fgder)^*)/ H^1_{\mc{L}_M^{\perp} \cup (L_{M,v}')^{\perp}}(\Gamma_{F,{\mc{S}}' \cup \{v\}}, \rho_M(\fgder)^*)
\] 
is also isomorphic to $\mc{O}/\varpi^M$. Two applications of the Greenberg--Wiles formula then imply that the inclusion\begin{equation} \label{eq:inc}
    H^1_{\mc{L}_{M}}(\Gamma_{F,{\mc{S}}'}, \rho_M(\fgder)) \hookrightarrow
    H^1_{\mc{L}_{M} \cup L_{M,v}'}(\Gamma_{F,{\mc{S}}' \cup \{v\}}, \rho_M(\fgder))
\end{equation}
is an equality. But now the fact that $\phi|_{\gal{F_v}}$ does not belong to $L_{M, v}$ implies that the inclusion
\[
H^1_{\mc{L}_M \cup L_{M, v}}(\gal{F, {\mc{S}}' \cup \{v\}}, \rho_M(\fgder)) \subset H^1_{\mc{L}_M \cup L'_{M, v}}(\gal{F, {\mc{S}}' \cup \{v\}}, \rho_M(\fgder))=H^1_{\mc{L}_M}(\gal{F, {\mc{S}}'}, \rho_M(\fgder))
\]
is strict, and our induction argument can proceed.

\end{proof}
\begin{rmk}
We remark  that it is entirely  possible that   the relative mod $p$ dual Selmer group $\ov{H^1_{\mc{L}_M^{\perp}}(\Gamma_{F,{\mc{S}}}, \rho_M(\fgder)^*)}$ is all of $ H^1_{\mathcal{L}_1^{\perp}}(\Gamma_{F,{\mc{S}}' \cup {\mc{Q}}},\br(\fgder)^*)$. This is because while Lie elements of mod $p$ Selmer groups   by Lazard's results  do not lift to mod $p^M$ Lie elements,  they could well lift to mod $p^M$ elements of Galois cohomology.  One of the key steps  of killing relative dual Selmer is that given $\alpha \in H^1_{{\mathcal{L}_M}}(\Gamma_{F,{\mc{S}}' \cup {\mc{Q}}},\rho_M(\fgder))$  that maps to  a non-zero element of the Selmer group  $H^1_{\mathcal{L}_1}(\Gamma_{F,{\mc{S}}' \cup {\mc{Q}}},\br(\fgder))$, there are auxiliary primes  at which $\alpha$ does not satisfy the Selmer condition. We do this simultaneously with killing suitable elements of dual Selmer. This  is  the  reason that the  number of places needed to kill the relative dual Selmer  $\ov{H^1_{\mc{L}_M^{\perp}}(\Gamma_{F,{\mc{S}}}, \rho_M(\fgder)^*)}$  may be larger than its dimension.
\end{rmk}

\subsection{Main theorem}
We can now finally collect all of our results to prove the main theorem. 
\begin{thm}\label{mainthm}
Let $p \gg_G 0$ be a prime. Let $F$ be a totally real field, and let $\br \colon \gal{F, {\mc{S}}} \to G(k)$ be a continuous representation unramified outside a finite set of finite places ${\mc{S}}$ containing the places above $p$. Let $\tF$ denote the smallest extension of $F$ such that $\br(\gal{\tF})$ is contained in $G^0(k)$, and assume that  $[\tF(\zeta_p): \tF]$ is strictly greater than the integer $a_G$ arising in Lemma \ref{cyclicq}. Fix a lift $\mu \colon \gal{F, {\mc{S}}} \to G/G^{\mr{der}}(\mc{O})$ of $\bar{\mu}=\br \pmod{G^{\mr{der}}}$, and assume that $\br$ satisfies the following:
\begin{itemize}
\item $\br$ is odd, i.e. for all infinite places $v$ of $F$, $h^0(\gal{F_v}, \br(\fgder))= \dim(\mr{Flag}_{G^{\mr{der}}})$.
\item $\br|_{\gal{\tF(\zeta_p)}}$ is absolutely irreducible.
\item For all $v \in {\mc{S}}$, $\br|_{\gal{F_v}}$ has a lift $\rho_v \colon \gal{F_v} \to G(\mc{O})$ of type $\mu|_{\gal{F_v}}$; and that for $v \mid p$ this lift may be chosen to be de Rham and regular in the sense that the associated Hodge--Tate cocharacters are regular.
\end{itemize}
Then there exist a finite extension $E'$ of $E=\Frac(\mc{O})$ (whose ring of integers and residue field we denote by $\mc{O}'$ and $k'$), and depending only on the set $\{\rho_v\}_{v \in {\mc{S}}}$; a finite set of places $\wt{{\mc{S}}}$ containing ${\mc{S}}$; and a geometric lift 
\[
\xymatrix{
& G(\mc{O}') \ar[d] \\
\gal{F, \wt{{\mc{S}}}} \ar[r]_{\br} \ar[ur]^{\rho} & G(k') 
}
\]
of $\br$, and having multiplier $\mu$, such that $\rho(\gal{F})$ contains $\wh{G^{\mr{der}}}(\mc{O}')$. Moreover, if we fix an integer $t_0$ and for each $v \in {\mc{S}}$ an irreducible component defined over $\mc{O}$ and containing $\rho_v$ of:
\begin{itemize}
\item for $v \in {\mc{S}} \setminus \{v \mid p\}$, the generic fiber of the local lifting ring, $R^{\square, \mu}_{\br|_{\gal{F_v}}}[1/\vpi]$ (where $R^{\square, \mu}_{\br|_{\gal{F_v}}}$ pro-represents $\Lift_{\br}^{\mu}|_{\gal{F_v}}$); and
\item for $v \mid p$, the lifting ring $R_{\br|_{\gal{F_v}}}^{\square, \mu, \tau, \mbf{v}}[1/\vpi]$ whose $\ov{E}$-points parametrize lifts of $\br|_{\gal{F_v}}$ with specified inertial type $\tau$ and Hodge type $\mbf{v}$ (see \cite[Prop. 3.0.12]{balaji} for the construction of this ring);  
\end{itemize}
then the global lift $\rho$ may be constructed such that, for all
$v \in {\mc{S}}$, $\rho|_{\gal{F_v}}$ is congruent modulo $\vpi^{t_0}$
to some $\wh{G^{\mr{der}}}(\mc{O}')$-conjugate of $\rho_v$, and
$\rho|_{\gal{F_v}}$ belongs to the specified irreducible component for
every $v \in {\mc{S}}$.\footnote{To be clear, the set $\wt{{\mc{S}}}$
  may depend on the integer $t_0$, but the extension $\mc{O}'$ does
  not depend on $t_0$.}
\end{thm}
\begin{rmk}\label{Mbound}
A number of our preliminary results either take in or produce Galois representations modulo $\vpi^M$ and $\vpi^N$ for integers $M$ and $N$ that satisfy various bounds, both absolute and relative to each other. In the proof of the present theorem, we will finally gather these statements together and use certain fixed values of $M$ and $N$. There are three requirements on $M$: 
\begin{itemize}
\item Most important, $M$ is at least the integer $M_1$ resulting from applying Corollary \ref{cor:lazard} (so ultimately $M$ depends on the image $\br(\gal{F})$), as described before Lemma \ref{lemma:fn}.
\item $M$ must be divisible by $e$ (to apply Definition \ref{reldefaux} and Lemma \ref{lem:triv2}).
\item Theorem \ref{klr^N} produces lifts mod $\vpi^{\n}$ for any $\n$, and at the auxiliary primes of ramification $w$ described by case (1b) of Theorem \ref{klr^N}, there is an integer $s \in \{1, 2, e\}$ such that $\rho_s|_{\gal{F_w}}$ is trivial mod center, but $\rho_{s+1}|_{\gal{F_w}}$ is ``in general position." $M$ must be at least any of these integers $s$, because for these $s$ we apply Lemma \ref{extracocycles^N} to $\rho_{M+s}$ (in the notation of that lemma). Taking into account the previous condition that $M$ is divisible by $e$, this just means we further require $M \geq 2$ when $e=1$.
\end{itemize} 
The requirements on $N$ are more complex and are described in the paragraph leading up to Equation (\ref{bounds}) below.
\end{rmk}
\begin{proof}
We begin with a preliminary reduction:
\begin{claim} \label{orbitredn}
It suffices to prove the theorem when $G^0$ is adjoint, and $\fg$ (now equal to $\fgder$) is equal to a single $\pi_0(G)$-orbit of simple factors.
\end{claim}
\begin{proof}[Proof of Claim:]
Writing $\br_Z$ for the image of $\br$ in $G/Z_{G^0}$, reduction modulo $Z_{G^0}$ induces morphisms of functors $\Lift_{\br}^\mu \to \Lift_{\br_Z}$ and $\Def_{\br}^\mu \to \Def_{\br_Z}$. Moreover, by Assumption \ref{minimalp} on $p$, both of these morphisms are isomorphisms, since they induce isomorphisms of both tangent and obstruction spaces.

In the remainder of the proof, we assume that $G^0$ is an adjoint group. It is then a product of simple adjoint groups, and we partition this product into the set $\Sigma$ of $\pi_0(G)$-orbits, writing $G^0= \prod_{s \in \Sigma} G_s$,
where $\pi_0(G)$ transitively permutes the simple factors of $G_s$.  We also write $G_{\neq s}$ for $\prod_{t \neq s} G_t$. Then the embedding $G \to G/G_s \times G/G_{\neq s}$ induces an isomorphism of $G$ onto the subgroup of pairs $(x_{\neq s}, x_{s})$ such that $x_{\neq s}$ and $x_s$ have the same image in $\pi_0(G)$ (consider the components under the orbit decomposition of $x_s^{-1} x_{\neq s} \in G^0$ to find suitable modifications of $x_s$ and $x_{\neq s}$ modulo $G_{\neq s}$ and $G_{s}$, respectively). Consequently, $G$ is isomorphic to the subgroup of $\prod_{s \in \Sigma} G/G_{\neq s}$ of elements with common projection to $\pi_0(G)$, and (since all lifts over $R \in \mc{C}_{\mc{O}}^f$ have $\pi_0(G)$-projection determined by the residual representation) we conclude $\Lift_{\br}$ and $\Def_{\br}$ themselves canonically decompose into the product of the corresponding functors with $\br$ replaced by its image $\br_{s}$ in each $G/G_{\neq s}$, each of which is a single $G$-orbit of simple factors. 

We therefore need only check that the hypotheses of our theorem still hold for each $\br_{s}$, and in turn that the conclusion of the theorem for each $\br_{s}$ then implies it for the original $\br$. The hypotheses of absolute irreducibility and existence of local lifts clearly still hold for each $\br_{s}$. To see that each $\br_{s}$ is odd, note that for all $v \mid \infty$, 
\[
h^0(\gal{F_v}, \br(\fg))= \sum_{s \in \Sigma} h^0(\gal{F_v}, \br_{s}(\mr{Lie}(G_s))) \geq \sum_{s \in \Sigma} \dim(\mr{Flag}_{G_s})= \dim(\mr{Flag}_G)= h^0(\gal{F_v}, \br(\fg)),
\]
so in each term ($s \in \Sigma$) in the sum, equality must hold, and therefore each $\br_{s}$ is odd. Finally, given at each $v \in {\mc{S}}$ a local lift $\rho_v$ of $\br|_{\gal{F_v}}$, if we can approximate $\rho_v \pmod{G_{\neq s}}$ (modulo $\wh{G_s}$-conjugacy) to any desired accuracy for all $s \in \Sigma$, then we have succeeded in approximating $\rho_v$ itself (modulo $\wh{G}$-conjugacy).
\end{proof}
For the remainder of the proof of the theorem, we therefore assume that $G^0$ is an adjoint group, and that $\fgder=\fg$ consists of a single $\pi_0(G)$-orbit of simple factors.

Now, for all $v \in {\mc{S}}$, fix as in the theorem statement irreducible components $\ov{R}_v[1/\vpi]$, containing the specified lifts $\rho_v$, of $R^{\square, \mu}_{\br|_{\gal{F_v}}}[1/\vpi]$ (for $v$ not above $p$) or $R_{\br|_{\gal{F_v}}}^{\square, \mu, \tau, \mbf{v}}[1/\vpi]$ (for $v$ above $p$); we may if desired enlarge $\mc{O}$ and make the component choice after extending scalars. Denote by $\ov{R}_v$ the scheme-theoretic closure of this component in $R_{\br|_{\gal{F_v}}}^{\square, \mu}$. Building on work of Kisin (\cite{kisin:pst}), Bellovin--Gee (\cite[Theorem 3.3.3]{Bellovin-Gee-G}) have shown that 
\[
\dim(\ov{R}_v[1/\vpi])=
\begin{cases}
\text{$\dim(\fgder)$ for $v$ not above $p$;}\\
\text{$\dim(\fgder)+\dim(\mr{Flag}_{G^{\mr{der}}})$ for $v$ above $p$,}
\end{cases} 
\]
and that $\ov{R}_v[1/\vpi]$ has a dense set of formally smooth closed points. (For $v$ above $p$, the dimension cited here results from the fact that the given lift $\rho_v$ is regular, so the parabolic associated to any Hodge--Tate cocharacter is in fact a Borel.) Recall from Lemma \ref{smoothapprox} that there is moreover a finite extension $\mc{O}'$ of $\mc{O}$, independent of the choice of $t_0$, and lifts $\rho_v'$ of $\br|_{\gal{F_v}}$ corresponding to $\mc{O}'$-points of the $\ov{R}_v$ such that $\rho_v' \equiv \rho_v \pmod{\vpi^{t_0}}$ and $\rho_v'$ defines a formally smooth point on $\ov{R}_v[1/\vpi]$. 
We replace $\rho_v$ by this approximation $\rho_v'$ and the ring $\mc{O}$ by $\mc{O}'$; for simplicity in the remainder of the proof, we will retain the original notation $\rho_v$, $\mc{O}$. The ring $\mc{O}$ is not enlarged at any later stage of the proof.

    Next we apply Lemma \ref{lemma:fn} to $\br$, producing an integer $M$, which we assume enlarged to satisfy $M \geq 2$ and $M$ is a multiple of $e$ (see Remark \ref{Mbound} for an explanation of these choices).
We run the first $e$ steps in the inductive argument of Theorem \ref{klr^N}, providing us with a lift $\rho_e \colon \gal{F, {\mc{T}}_e} \to G(\mc{O}/\vpi^e)$ satisfying all of the conclusions of \textit{loc.~cit.}. At the finite number of primes $w \in {\mc{T}}_e \setminus {\mc{S}}$ of ramification introduced in the course of constructing $\rho_e$, we have in Theorem \ref{klr^N} fixed local $G(\mc{O})$-valued lifts $\rho_w$ (by invoking Lemma \ref{auxlift}) that define formally smooth points on suitable irreducible components $\ov{R}_w[1/\vpi]$ of $R_{\br|_{\gal{F_w}}}^{\square, \mu}[1/\vpi]$. \textit{In the remainder of the argument, we treat these primes on the same terms as those of ${\mc{S}}$, so for notational convenience we enlarge ${\mc{S}}$ to include this finite set of primes.}

Thus for all $v$ in our updated ${\mc{S}}$ we have the components $\ov{R}_v[1/\vpi]$. We then apply Proposition \ref{prop:gensmooth} with $r_0=M$ to each pair $(\ov{R}_v[1/\vpi], \rho_v)_{v \in {\mc{S}}}$. Let $N_0$ be the maximum (over all $v \in {\mc{S}}$) of the integers $n_0$ thus produced by Proposition \ref{prop:gensmooth}, let $N_1$ be the integer produced by Lemma \ref{frattini}, and let 
\begin{equation}\label{bounds}
N= \max\{N_0+M, N_1+M, t+M, 4M\}. 
\end{equation} 
By increasing this $N$ if necessary, we may and
  do assume that $e \mid N$.
    With this choice of $N$ we continue (starting from
    $\rho_e$) the induction of Theorem \ref{klr^N}, targeting at
    places in ${\mc{S}}$ the fixed local lifts $\rho_v$. We thereby produce,
    for some finite set of primes ${\mc{S}}' \supset {\mc{S}}$, a $\rho_N \colon \gal{F, {\mc{S}}'} \to G(\mc{O}/\vpi^N)$ satisfying the conclusions of Theorem \ref{klr^N}: 
\begin{itemize}
\item For all $v \in {\mc{S}}$, $\rho_N|_{\gal{F_v}}$ is $\wh{G^{\mr{der}}}(\mc{O})$-conjugate to the fixed lift $\rho_v$.
\item The image $\im(\rho_N)$ contains $\wh{G^{\mr{der}}}(\mc{O}/\vpi^N)$.
\item For all $v \in {\mc{S}}' \setminus {\mc{S}}$, $\rho_N|_{\gal{F_v}}$ is
  unramified or satisfies (modulo $\varpi^{M+s}$) the hypotheses of Lemma \ref{extracocycles^N} for some $s \in \{1, 2, e\}$.
\end{itemize}
At each $v \in {\mc{S}}'$, we let $L_{r, v} \subset H^1(\gal{F_v},
\rho_r(\fgder))$, $1 \leq r \leq M$ be the subspace arising from the proof of Theorem \ref{klr^N} and the results of \S \ref{trivialsection} and \S \ref{genericfibersection} as follows:
\begin{itemize}
\item If $v \in {\mc{S}}$, 
we just noted that $\rho_N \equiv \rho_v \pmod{\vpi^N}$ modulo $\wh{G^{\mr{der}}}(\mc{O})$-conjugacy. We \textit{replace} $\rho_v$ with its suitable $\wh{G^{\mr{der}}}(\mc{O})$-conjugate, noting as in Remark \ref{whconj} that Proposition \ref{prop:gensmooth} applies---with the same quantitative bounds---to this conjugate, and then we let $L_{r, v}$ be the subspace of $H^1(\gal{F_v}, \rho_v(\fgder) \otimes_{\mc{O}} \mc{O}/\vpi^r)$ produced by applying Proposition \ref{prop:gensmooth} to our new $\rho_v$. 
\item If $v \in {\mc{S}}' \setminus {\mc{S}}$, and $\rho_N|_{\gal{F_v}}$ is unramified, then we simply take $L_{r, v}$ to be the (image in cohomology of the) subspace of unramified cocycles.
\item If $v \in {\mc{S}}' \setminus {\mc{S}}$, and $\rho_N|_{\gal{F_v}}$ is ramified, then the space $L_{r, v}$ is the one constructed in Lemma \ref{extracocycles^N}. 
\end{itemize}
We can then consider the (balanced) relative Selmer and dual Selmer groups $\ov{H^1_{\mc{L}_M}(\gal{F, {\mc{S}}'}, \rho_M(\fgder))}$ and $\ov{H^1_{\mc{L}_M^\perp}(\gal{F, {\mc{S}}'}, \rho_M(\fgder)^*)}$. We apply Theorem \ref{thm:killrel} (here is where we use the reduction of Claim \ref{orbitredn}) to produce a finite set ${\mc{Q}} \subset {\mc{Q}}_N$ of primes such that the ${\mc{Q}}$-new relative (dual) Selmer group $\ov{H^1_{\mc{L}_M^\perp \cup \{L_{M, v}^\perp\}_{v \in {\mc{Q}}}}(\gal{F, {\mc{S}}' \cup {\mc{Q}}}, \rho_M(\fgder)^*)}$ vanishes. For all $v \in {\mc{S}}' \cup {\mc{Q}}$ and all $n \geq N-M$, we let $D_{n, v} \subset \Lift_{\br|_{\gal{F_v}}}(\mc{O}/\vpi^n)$ be the class of lifts produced by applying Proposition \ref{prop:gensmooth} (for $v \in {\mc{S}}$), Lemma \ref{extracocycles^N} (for $v \in {\mc{S}}' \setminus {\mc{S}}$ that are ramified in $\rho_N$, putting us in the case (1b) of Theorem \ref{klr^N}), or Lemma \ref{lem:triv2} (for $v \in {\mc{Q}}$), or by simply taking $D_{n, v}$ to be the unramified lifts (for $v \in {\mc{S}}' \setminus {\mc{S}}$ that are unramified in $\rho_N$). In particular, for all $1 \leq r \leq M$, the fibers of $D_{n+M, v} \to D_{n, v}$ are stable under the preimages in $Z^1(\gal{F_v}, \rho_r(\fgder))$ of the subspaces $L_{r, v}$. (The only case in which we have not discussed this explicitly is the unramified case; but there it is straightforward.) The Theorem now results from the following:
\begin{claim}\label{relativediagram}
For each $n \geq N-M$, we have pairs of (multiplier $\mu$) liftings $(\tau_n, \rho_{n+M})$, where $\tau_n \colon \gal{F, {\mc{S}}' \cup {\mc{Q}}} \to G(\mc{O}/\vpi^n)$ and $\rho_{n+M} \colon \gal{F, {\mc{S}}' \cup {\mc{Q}}} \to G(\mc{O}/\vpi^{n+M})$, and $\tau_n= \rho_{n+M} \pmod{\vpi^n}$, with the following properties:
\begin{enumerate}
\item For each $v \in {\mc{S}}' \cup {\mc{Q}}$, $\tau_n|_{\gal{F_v}}$ belongs to $D_{n, v}$, and $\rho_{n+M}|_{\gal{F_v}}$ belongs to $D_{n+M, v}$.
\item $\tau_{n+1}= \rho_{n+M} \pmod{\vpi^{n+1}}$.
\item $\tau_n = \tau_{n+1} \pmod{ \vpi^n}$.
\end{enumerate}
\end{claim}
\begin{proof}[Proof of Claim:] 
The proof will be by induction on $n$. We start the induction by setting $\tau_{N-M}$ to be $\rho_N \pmod {\varpi^{N-M}}$. The assumptions on $\rho_N$ ensure that the first property holds for the initial pair $(\tau_{N-M}, \rho_N)$. 

We now assume that we have constructed a pair of representations $(\tau_n, \rho_{n+M})$ satisfying the first property. We then set $\tau_{n+1}= \rho_{n+M} \pmod {\varpi^{n+1}}$, so the second and third properties also hold for this $n$. Recall that we have arranged the vanishing of $\Sha^1_{{\mc{S}}'}( \Gamma_{F,{\mc{S}}'}, \br(\fgder)^*)$, which implies the same for $\Sha^1_{{\mc{S}}' \cup {\mc{Q}}}( \Gamma_{F,{\mc{S}}'}, \br(\fgder)^*) \cong \Sha^2_{{\mc{S}}' \cup {\mc{Q}}}(\Gamma_{F, {\mc{S}}'}, \br(\fgder))^\vee$, so $\rho_{n+M}$ lifts to a homomorphism $\rho_{n+M+1}': \Gamma_{F,{\mc{S}}' \cup {\mc{Q}}} \to G(\mc{O}/\varpi^{n+M+1})$. Viewing the restriction of $\rho_{n+M+1}'$ to $\Gamma_{F_v}$ for $v \in {\mc{S}}' \cup {\mc{Q}}$ as a lift of the restriction of $\tau_{n+1}$ to $\Gamma_{F_v}$ and comparing with an element of the fiber of $D_{n+M+1, v} \to D_{n+1, v}$ over $\tau_{n+1}|_{\gal{F_v}}$, we get an element
\[
  (f_v)_{v \in {\mc{S}}' \cup {\mc{Q}}} \in \bigoplus_{v \in {\mc{S}}' \cup {\mc{Q}}} \frac{H^1(\Gamma_{F_v},
    \rho_M(\fgder))}{L_{M,v}} \ .
\]
Since $\rho_{n+M+1}'$ can be viewed as a lift of $\rho_{n+M} \pmod {\varpi^{n+1}}$, and $\rho_{n+M+1}' \mod \varpi^{n+M} = \rho_{n+M}$ is at each $v$ in the fiber of $D_{n+M, v} \to D_{n+1, v}$ over $\rho_{n+M} \pmod {\varpi^{n+1}} = \tau_{n+1}$, the image of $(f_v)$ in $\bigoplus_{v \in {\mc{S}}' \cup {\mc{Q}}} \frac{H^1(\Gamma_{F_v}, \rho_{M-1}(\fgder))}{L_{M-1,v}}$ vanishes.

We have a commutative diagram
  \[
    \xymatrix{
            H^1(\Gamma_{F, {\mc{S}}' \cup {\mc{Q}}}, \rho_M(\fgder)) \ar[r]
      \ar[d] & \bigoplus_{v \in {\mc{S}}' \cup {\mc{Q}}} \frac{H^1(\Gamma_{F_v}, \
      \rho_M(\fgder))}{L_{M,v}} \ar[r] \ar[d] &
    H^1_{\mc{L}_{M}^{\perp} \cup \{L_{M, v}^\perp\}_{v \in {\mc{Q}}}}(\Gamma_{F, {\mc{S}}' \cup {\mc{Q}}}, \rho_M(\fgder)^*)^{\vee} \ar[d]\\
           H^1(\Gamma_{F, {\mc{S}}' \cup {\mc{Q}}}, \rho_{M-1}(\fgder)) \ar[r]
       & \bigoplus_{v \in {\mc{S}}' \cup {\mc{Q}}} \frac{H^1(\Gamma_{F_v}, \
      \rho_{M-1}(\fgder))}{L_{M-1,v}} \ar[r]  &
    H^1_{\mc{L}_{M-1}^{\perp}\cup \{L_{M-1, v}^\perp\}_{v \in {\mc{Q}}}}(\Gamma_{F, {\mc{S}}' \cup {\mc{Q}}}, \rho_{M-1}(\fgder)^*)^{\vee} \\
  }
\]
in which the rows come from (part of) the Poitou--Tate exact sequence, and the vertical maps are induced by the reduction map $\rho_{M}(\fgder) \to \rho_{M-1}(\fgder)$ (and, for the third vertical map, dualizing twice); commutativity of the left-hand square is obvious and of the right-hand square follows from the properties of cup-product and the local duality pairings.  
Lemma \ref{lem:exseq} and the vanishing of the relative dual Selmer group
$\ov{H^1_{\mc{L}_{M}^{\perp}\cup \{L_{M, v}^\perp\}_{v \in {\mc{Q}}}}(\Gamma_{F,{\mc{S}}' \cup {\mc{Q}}}, \rho_M(\fgder)^*)}$ together imply that the map 
\[
H^1_{\mc{L}_{M-1}^{\perp} \cup  \{L_{M-1, v}^\perp\}_{v \in {\mc{Q}}}}(\Gamma_{F,{\mc{S}}' \cup {\mc{Q}}}, \rho_{M-1}(\fgder)^*) \to H^1_{\mc{L}_{M}^{\perp} \cup  \{L_{M, v}^\perp\}_{v \in {\mc{Q}}}}(\Gamma_{F,{\mc{S}}' \cup {\mc{Q}}}, \rho_{M}(\fgder)^*)
\] 
is surjective, so the last vertical map in the diagram is injective. The commutativity of the diagram then implies that $(f_v)_{v \in {\mc{S}}' \cup {\mc{Q}}}$ maps to zero in $H^1_{\mc{L}_{M}^{\perp}\cup \{L_{M, v}^\perp\}_{v \in {\mc{Q}}}}(\Gamma_{F, {\mc{S}}' \cup {\mc{Q}}},
\rho_M(\fgder)^*)^{\vee}$, so $\rho_{n+M+1}'$ can be modified by a cocycle in $H^1(\Gamma_{F,{\mc{S}}' \cup {\mc{Q}}}, \rho_M(\fgder)^*)$ to get a lift $\rho_{n+M+1}$ of $\tau_{n+1}$ such that for all $v \in {\mc{S}}' \cup {\mc{Q}}$, $\rho_{n+M+1}|_{\gal{F_v}}$ belongs to $D_{n+M+1, v}$\footnote{Note that we do not claim that $\rho_{n+M+1}$ is a lift of $\rho_{n+M}$.}; here crucially we use that for $n$ in the considered range the fibers of $D_{n+M+1, v} \to D_{n+1, v}$ are stable under all cocycles (with image) in $L_{M, v}$. This completes the induction step.
\end{proof}
Having established the Claim, we set $\rho= \varprojlim \tau_n \colon \gal{F, {\mc{S}}' \cup {\mc{Q}}} \to G(\mc{O})$. Since $\rho$ lifts $\rho_{N-M}$, we have achieved the desired (modulo $\vpi^{t_0}$) approximation of our fixed local lifts, and Lemma \ref{frattini} implies that $\im(\rho)$ contains $\wh{G^{\mr{der}}}(\mc{O})$, concluding the proof of the theorem.
\end{proof}
\begin{lemma}\label{frattini}
There is an integer $N_1$ depending only on $G^{\mr{der}}$ such that
any closed (in the $p$-adic topology) subgroup $P$ of $\wh{G^{\mr{der}}}(\mc{O})$ whose reduction modulo $\vpi^{N_1}$ equals $\wh{G^{\mr{der}}}(\mc{O}/\vpi^{N_1})$ must in fact equal $\wh{G^{\mr{der}}}(\mc{O})$.
\end{lemma}
\begin{proof}
Since $H:= \wh{G^{\mr{der}}}(\mc{O})$ is a $p$-adic Lie group, the subgroup $H^p$ generated by $p^{th}$ powers is open, so there exists an integer $N_1$ such that $H^p$ contains the kernel of reduction modulo $\vpi^{N_1}$. In particular, if $P \pmod{\vpi^{N_1}}$ is equal to $\wh{G^{\mr{der}}}(\mc{O}/\vpi^{N_1})$, then $P$ surjects onto $H/H^p$. Since the Frattini subgroup of a finite $p$-group contains the subgroup generated by $p^{th}$ powers, Frattini's theorem for finite groups implies that $P$ surjects onto any finite quotient of $H$. Since $P$ is closed, we must in fact have $P=H$.
\end{proof}
\begin{rmk}\label{lnotplocal}
When $G= \mr{GL}_n$, $\mr{GSp}_{2n}$, or $\mr{GO}_n$, local lifts $\rho_v$ as in the theorem statement are known to exist for $v \nmid p$, by \cite[\S 2.4.4]{clozel-harris-taylor} and \cite[\S 7]{booher:minimal} (after possibly replacing $k$ with a finite extension). Many other cases for general $G$ can be worked out by hand (e.g., \cite[\S 4]{stp:exceptional} and \cite[\S 4]{stp:exceptional2}), but there is as yet no general result. We expect that it is always possible to find such $\rho_v$.

For $v \mid p$, much less is known, but the work of Emerton and Gee (\cite{emerton-gee:moduli}) shows that such lifts $\rho_v$ always exist when $G= \mr{GL}_n$. It is to be hoped that their methods will eventually eliminate this hypothesis entirely for arbitrary $G$. 
\end{rmk}
\begin{rmk}\label{effective}
We make some remarks on the effectivity of the bound $p\gg_G 0$ in the
theorem. The possible need to increase $p$ arises at several points in
the paper. In \S \ref{trivialsection}, $p$ is any prime. In \S
\ref{klrsection} \S \ref{klr^Nsection} we have not computed an explicit bound on $p$, but we could easily derive one by following the arguments of those sections; the bounds coming from these sections essentially amount to the condition that certain $\Fp$-vector spaces not be covered by a finite (bounded absolutely in terms of $G$) number of hyperplanes. In deducing the image hypothesis of \S \ref{klrsection} and \S \ref{klr^Nsection} from the irreducibility hypothesis of Theorem \ref{mainthm}, there is an explicit bound ensuring the cohomology ($H^0$ and $H^1$) vanishing, and an explicit bound (see Remark \ref{explicit}) to ensure disjointness of $\br(\fgder)$ and $\br(\fgder)^*$. The same remark and an inspection of the proof yields an effective bound for the integer $a_G$ in Lemma \ref{cyclicq}. Finally, the integer $d$ 
appearing in the proof of Proposition \ref{prop:killrel} is effective. In sum, the bound in Theorem \ref{mainthm} can be made effective.
\end{rmk}
\begin{rmk}
If we are instead given a homomorphism $\rho_n \colon \gal{F, {\mc{S}}} \to G(\mc{O}/\vpi^n)$ such that $\br$ satisfies the hypotheses of the theorem, and moreover (for all $v \in {\mc{S}}$) $\rho_n|_{\gal{F_v}}$ has a lift $\rho_v$ as in the theorem statement, then we can produce a geometric lift of $\rho_n$ (moreover approximating the given $\rho_v$'s).
\end{rmk}
\begin{rmk}\label{potautcompare}
When $G^0= \mr{GL}_n$ (or some minor variant thereof), we compare our results to the lifting results coming from potential automorphy theorems. The main lifting theorem of \cite[Theorem 4.3.1]{blggt:potaut} implies the existence of lifts with prescribed local behavior of an odd homomorphism $\br \colon \gal{F, {\mc{S}}} \to \mc{G}_n(k)$ such that $\br|_{\gal{\tF(\zeta_p)}}$ is absolutely irreducible, $p \geq 2(n+1)$; to be precise, one asks that:
\begin{itemize}
\item for $v \nmid p$, $\br|_{\gal{F_v}}$ has a lift on a given irreducible component of $R^{\square, \mu}_{\br|_{\gal{F_v}}}$; and 
\item for $v \mid p$, $\br|_{\gal{F_v}}$ has a potentially
  diagonalizable lift on a given irreducible component of some
  $R_{\br|_{\gal{F_v}}}^{\square, \mu, \tau,
    \mbf{v}}$ (for an inertial type $\tau$ and regular
  $p$-adic Hodge type $\mbf{v}$);
\end{itemize}
and one concludes that $\br$ has a global lift $\rho$ such that $\rho|_{\gal{F_v}}$ lies on the given irreducible components for all $v \in {\mc{S}}$. Our result (in addition to applying to general $G$) strengthens this in two ways: at $v \mid p$, we do not require the local lift $\rho_v$ to be potentially diagonalizable (we remark that the recent preprint \cite{calegari-emerton-gee} improves ``potentially diagonalizable" to ``globally realizable"); and for all $v \in {\mc{S}}$, we can approximate (modulo $\wh{G^{\mr{der}}}(\mc{O})$-conjugacy) the fixed local lifts to any desired degree of precision. What we lose is some sharpness in the bound on allowable $p$, and, more important, minimality of the lifts (i.e., our lifts are not unramified outside ${\mc{S}}$). And of course we do not establish potential automorphy!
\end{rmk}
\begin{rmk}
It would be interesting to pursue an analogue of our main theorem for \textit{reducible} $\br$; indeed, in some sense the seed of our project was the study of the paper \cite{ramakrishna-hamblen}, which produces irreducible lifts of (certain) reducible $\gal{\Q} \to \mr{GL}_2(k)$. Our methods will certainly adapt to cover many reducible cases as well, and we intend to pursue this problem in the future.\footnote{We have made progress on the reducible case: see \cite{fkp:reduciblearXiv}.}
\end{rmk}

We also note that the method of proof allows us, without assuming oddness of $\br$, to construct possibly non-geometric $p$-adic deformations, since the arguments of Theorem \ref{mainthm} only require that whenever we have a non-trivial dual Selmer class, we can also find a non-trivial Selmer class:
\begin{thm}\label{notodd}
Let $p\gg_G 0$ be a prime, let $F$ be any number field, and let $\br \colon \gal{F, {\mc{S}}} \to G(k)$ be a continuous representation unramified outside a finite set of places ${\mc{S}}$ containing those above $p$. Fix a lift $\mu \colon \gal{F, {\mc{S}}} \to G/G^{\mr{der}}(\mc{O})$ of $\br \pmod{G^{\mr{der}}}$. Let $\tF$ be as in Theorem \ref{mainthm}, assume that $[\tF(\zeta_p):\tF]=p-1$, and that $\br$ satisfies the following:
\begin{itemize}
\item $\br|_{\gal{\tF(\zeta_p)}}$ is absolutely irreducible.
\item For all $v \in {\mc{S}}$, $\br|_{\gal{F_v}}$ has a lift $\rho_v$ of type $\mu|_{\gal{F_v}}$, and for $v \mid p$ this lift can be chosen to correspond to a formally smooth point on an irreducible component of $R^{\square, \mu}_{\br|_{\gal{F_v}}}[1/\vpi]$ of dimension $(1+[F_v:\Q_p])\dim_k(\fgder)$. (For instance, this hypothesis holds if $H^2(\gal{F_v}, \br(\fgder))=0$.)
\end{itemize}
Then for some finite set of primes $\wt{{\mc{S}}} \supset {\mc{S}}$ and finite extension $\mc{O}'$ of $\mc{O}$, $\br$ admits a lift $\rho \colon \gal{F, \wt{{\mc{S}}}} \to G(\mc{O}')$, and $\rho$ may be arranged such that, for all $v \in {\mc{S}}$, $\rho|_{\gal{F_v}}$ is congruent modulo $\vpi^{t_0}$ to some $\wh{G^{\mr{der}}}(\mc{O}')$-conjugate of $\rho_v$.
\end{thm}
\begin{rmk}
There is not to our knowledge a known result, analogous to the results of \cite{Bellovin-Gee-G} on the generic fibers of the unrestricted lifting rings $R^{\square, \mu}_{\br|_{\gal{F_v}}}$, and such results do not follow formally from the methods of \cite{Bellovin-Gee-G}. 
\end{rmk}
\begin{proof}
The argument is the same as that of Theorem \ref{mainthm}, except at places $v \mid p$ we consider all lifts of $\br|_{\gal{F_v}}$ and choose an irreducible component $\ov{R}_v[1/\vpi]$ of the full local lifting ring $R^{\square, \mu}_{\br|_{\gal{F_v}}}$ as specified in the hypotheses of the theorem. The corresponding subspaces $L_{r, v}$ produced by an analogue of Proposition \ref{prop:gensmooth} have order $|\rho_r(\fgder)^{\gal{F_v}}|\cdot |\mc{O}/\vpi^r|^{\dim_k(\fgder)[F_v:\Q_p]}$, and the corresponding application of the Greenberg--Wiles formula shows that
\[
|H^1_{\mc{L}_M}(\gal{F, {\mc{S}}'}, \rho_M(\fgder))|/|H^1_{\mc{L}_M^{\perp}}(\gal{F, {\mc{S}}'}, \rho_M(\fgder)^*)| \geq 1
\]
(equality holds when $F$ is totally real, and $\br(c_v)=1$ for all complex conjugations $c_v$), with analogous conclusions for the relative Selmer and dual Selmer groups. This inequality suffices to proceed as in the proof of Theorem \ref{mainthm}.
\end{proof}

\section{Examples}\label{examples}
In this section we gather a few examples to which our methods apply.
\subsection{The principal $\mr{SL}_2$}
In \cite{stp:exceptional} (and \cite{stp:exceptional2}) it was shown how the original lifting argument of \cite{ramakrishna02} and \cite{taylor:icos2} could be adapted to prove lifting results for $\br \colon \gal{F} \to G(k)$ whose image was (approximately) a principal $\mr{SL}_2$. In fact, the argument in that paper was carried out for the exceptional groups, at one point relying on a brute-force Magma computation (see \cite[Lemma 7.6]{stp:exceptional}); for the classical Dynkin types except for $D_{2n}$, case-by-case matrix calculations (not carried out in \cite{stp:exceptional}, but some of which appear in \cite{tang:thesisANT}) complete the argument. The arguments of the present paper apply to these examples without relying on case-by-case calculation, and moreover treating type $D_{2n}$ as well.

Let $G^0$ be a split connected reductive group over $\Z_p$. Recall that for $p \gg_{G^0} 0$, there is a unique conjugacy class of principal homomorphisms $\varphi \colon \mr{SL}_2 \to G^0$ defined over $\Z_p$ (see \cite{serre:principalsl2}). Assume that $G= {}^L H$, the L-group of a connected reductive group $H$ over $F$; that is, we choose over $\overline{F}$ a maximal torus and Borel subgroup $T_{\overline{F}} \subset B_{\overline{F}} \subset H_{\overline{F}}$ to obtain a based root datum, and then a choice of pinning allows us to define an $L$-group $G= {}^L H= H^\vee \rtimes \Gal(\tF/F)$ for some finite extension $\tF/F$. The principal $\mr{SL}_2$ extends to a homomorphism $\varphi \colon \mr{SL}_2 \times \gal{F} \to {}^L H$ (\cite[\S 2]{gross:principalsl2}), and we assume that $\varphi$ extends to a homomorphism $\mr{GL}_2 \times \gal{F} \to {}^L H$ (this is always the case if, e.g., $H$ is simply-connected, and in general it can be arranged by enlarging the center of $H$). The following crucial assumption is needed to use the principal $\mr{SL}_2$ to produce \textit{odd} homomorphisms valued in ${}^L H$:
\begin{assumption}\label{oddgroup}
Assume that $\tF/F$ is contained in a quadratic totally imaginary extension of the totally real field $F$, and that the automorphism of $H^\vee$ given by projecting any complex conjugation $c \in \gal{F}$ to $\Gal(\tF/F)$ preserves each simple factor $\mf{h}^\vee_i$ of $\mf{h}^\vee= \Lie(H^\vee)$, and acts on $\mf{h}^\vee_i$ as the identity if $-1 \in W_{\mf{h}_i}$ and as the opposition involution if $-1 \not \in W_{\mf{h}_i}$. 
\end{assumption}
This assumption leads to the following archimedean calculation:
\begin{lemma}
Let $\theta_v \in {}^L H(k)$ be the element 
\[
\theta_v= \varphi \left( \begin{pmatrix}
-1 &0 \\ 0 & 1
\end{pmatrix} \times c_v \right).
\]
Then 
\[
\dim_k(\fgder)^{\Ad(\theta_v)=1}= \dim_k (\mf{n}),
\]
where $\mf{n}$ is the unipotent radical of a Borel subgroup of $G$ (or of $G^{\mr{der}}$).
\end{lemma}
\begin{proof}
Combining \cite[Lemma 4.19, 10.1]{stp:exceptional}, we find that $\dim_k (\fgder)^{\Ad(\theta_v)=1}= \dim_k (\mf{n})$. 
\end{proof}
\begin{thm}
Let $G= {}^L H$ be constructed as above, satisfying Assumption \ref{oddgroup}, let $p \gg_G 0$, and assume that $[\tF(\zeta_p): \tF]>a_G$. Let ${\mc{S}}$ be a finite set of places of the totally real field $F$ containing all $v \mid p$, assuming for simplicity that all places in ${\mc{S}}$ are split in $\tF/F$, and let $\bar{r} \colon \gal{F, {\mc{S}}} \to \mr{GL}_2(k)$ be a continuous representation satisfying the following properties:
\begin{enumerate}
\item For some subfield $k_0 \subset k$, the projective image of $\bar{r}$ contains $\mr{PSL}_2(k_0)$. 
\item $\det \bar{r}(c)=-1$ for all complex conjugations $c \in \gal{F}$.
\end{enumerate}
Then there exist a finite extension $\mc{O}'$ of $\mc{O}$ and a finite set of trivial primes ${\mc{Q}}$ such that $\br= \varphi \circ \bar{r} \colon \gal{F, {\mc{S}} \cup {\mc{Q}}} \to G(k)$ has a geometric lift $\rho \colon \gal{F, {\mc{S}} \cup {\mc{Q}}} \to G(\mc{O})$, with $\im(\rho)$ containing $\wh{G^{\mr{der}}}(\mc{O})$. (The more refined local conclusions of Theorem \ref{mainthm} also hold, given local liftings $\rho_v$, $v \in {\mc{S}}$, that one wants to approximate.)
\end{thm}
\begin{proof}
For $p \gg_G 0$, it is straightforward to verify the irreducibility hypotheses of Theorem \ref{mainthm} (compare \cite[Theorem 7.4, Theorem 10.4]{stp:exceptional}). Note that the constant $a_G$ of Lemma \ref{cyclicq} can be replaced by 2, by the assumption on $\im(\bar{r})$. To satisfy the local hypotheses of Theorem \ref{mainthm}, it will even suffice to construct local lifts $r_v \colon \gal{F_v} \to \mr{GL}_2(\mc{O})$ such that, for $v \mid p$, $r_v$ is de Rham and regular. If $v \in {\mc{S}} \setminus \{v \mid p\}$, such lifts are known to exist even with $\mr{GL}_N$ in place of $\mr{GL}_2$ (\cite[Corollary 2.4.21]{clozel-harris-taylor}). If $v \mid p$, $\bar{r}|_{\gal{F_v}}$ admits a Hodge--Tate regular, potentially crystalline lift $r_v$ by \cite[Theorem 2.5.3, Theorem 2.5.4]{muller:thesis}. The theorem now follows.
\end{proof}
\begin{rmk}
In particular, starting with $\bar{r} \colon \gal{\Q} \to \mr{GL}_2(k)$ coming either from classical modular forms or elliptic curves, we can construct geometric representations $\rho \colon \gal{\Q} \to G(\mc{O})$ whose image has Zariski closure containing $G^{\mr{der}}$. This was the application of the lifting theorems in \cite{stp:exceptional}.
\end{rmk}
\subsection{Normalizers of tori}
In this subsection we make no effort to be maximally general. For simplicity we assume that $G^0/Z_{G^0}$ is simple. Let $T$ be a (split) maximal torus of $G^0$. Residual representations valued in $N_G(T)(k)$ lift to $G(\mc{O})$, since (provided $p$ does not divide $|W_{G^0}|$) the image of $\br$ has order prime to $p$. Our main theorem shows that non-trivial lifts, with image containing an open subgroup of $G^{\mr{der}}(\mc{O})$, also exist under suitable hypotheses on $\br$.

\begin{thm}
Let $p\gg_{G} 0$, and assume $[\tF(\zeta_p): \tF] > a_G$. Let $\br \colon \gal{F, {\mc{S}}} \to N_G(T)(k)$ satisfy the following:
\begin{itemize}
\item $\br|_{\gal{\tF(\zeta_p)}}$ is absolutely irreducible. For instance, we could assume $\im(\br)$ contains $N_G(T)(\Fp)$ (and $p \gg_G 0$); or that $\br(\gal{\tF(\zeta_p)})$ contains a regular semisimple element of $T$ whose centralizer is $T$ (automatic if $G^0$ is simply-connected), and that the projection of $\br(\gal{\tF(\zeta_p)})$ to the Weyl group contains a Coxeter element.
\item $\br$ is odd. For instance, we can make one of the following assumptions:
\begin{itemize}
\item If $-1 \in W_{G^0}$, then for all $v \mid \infty$, $\br(c_v)$ either projects to $-1 \in W_{G^0}$, or projects to $\rho^\vee(-1) \in G^{\mr{ad}}$ (where $\rho^\vee$ is the usual half-sum of the positive co-roots of $G$). 
\item If $-1 \not \in W_{G^0}$, then for all $v \mid \infty$, $\br(c_v)$ either equals $(w_0, \tau) \in G^0 \rtimes \pi_0(G)$, where $w_0$ lifts the longest element of $W_{G^0}$, and $\tau$ is a pinned outer automorphism of $G^0$ acting as the opposition involution on $T \cap G^{\mr{der}}$; or it projects to $(\rho^\vee(-1), \tau) \in G^{\mr{ad}} \rtimes \pi_0(G)$.
\end{itemize}
\item For all $v \mid p$, $\br|_{\gal{F_v}}$ factors through $T(k)$.
\end{itemize}
Then for some finite set of places ${\mc{T}} \supset {\mc{S}}$, $\br$ admits a geometric lift $\rho \colon \gal{F, {\mc{T}}} \to G(\mc{O})$ whose image contains $\wh{G^{\mr{der}}}(\mc{O})$.
\end{thm}
\begin{proof}
First we check the local hypotheses of Theorem \ref{mainthm}. At primes in ${\mc{S}} \setminus \{v \mid p\}$ we may take the obvious Teichm\"{u}ller lifts, since the order of $\im(\br)$ is prime to $p$. At $v \mid p$, it is easy to lift a $\gal{F_v} \to T(k)$ to a potentially crystalline and Hodge--Tate regular $\rho_v \colon \gal{F_v} \to T(\mc{O})$. That the examples given of possible $\br(c_v)$ are in fact involutions follows from \cite[Lemma 2.3]{yun:exceptional}, \cite[Lemma 10.1]{stp:exceptional}, and a similar check in the case $\br(c_v)= (w_0, \tau)$. To check the irreducibility hypothesis is similarly straightforward: compare the proof of \cite[Proposition 10.7]{bhkt:fnfieldpotaut}.
\end{proof}
We have certainly not optimized the explicit descriptions of the possible local or global images here. It can be difficult, for instance, to realize $N_G(T)(\Fp)$ as a Galois group over $\Q$ (the sequence $1 \to T \to N_G(T) \to W_G \to 1$ need not split), and Theorem \ref{mainthm} is easily seen to apply when $\br(\gal{\tF})$ equals certain somewhat smaller subgroups of $N_G(T)(\Fp)$. Here we give two examples from the recent literature:
\begin{eg} In \cite{tang:thesisANT}, Tang classifies those connected reductive groups $G$ that arise as the Zariski closure of the image of a homomorphism $\gal{\Q} \to G(\overline{\Q}_p)$. The main theorem of \cite{tang:thesisANT} gives a complete answer to this question (for $p \gg 0$) modulo some elusive cases, consisting of certain simply-connected groups (e.g. $E_7^{\mr{sc}}$) for $p$ failing to satisfy some congruence condition (see \cite[Theorem 1.3]{tang:thesisANT}). As explained in \cite[Theorem 1.5, \S 3.4]{tang:thesisANT}, our main theorem allows Tang to treat these remaining cases by deforming $\br$ valued in a ``large enough" subgroup of $N_G(T)(\Fp)$.
\end{eg}
\begin{eg}
The main theorem of \cite{bcelmpp} produces $E_6$-Galois representations that arise in the cohomology of algebraic varieties (and are potentially automorphic) by studying the deformation theory of a carefully-constructed $\br \colon \gal{\Q} \to E_6^{\mr{sc}}(\Fp) \rtimes \Out(E_6)$ whose image projects onto a 3-Sylow subgroup of the Weyl group of $E_6$. Our theorem applies to find geometric lifts with full image of these $\br$ as well, and it also can lift them with sets of Hodge--Tate weights not accessible by the methods of \cite{bcelmpp} (which rely on potential automorphy theorems after composing with the minuscule representation of $E_6$). Of course, our arguments do not show these lifts are motivic or potentially automorphic!
\end{eg} 
\subsection{Deforming exotic finite subgroups}
We conclude by constructing some odd irreducible representations $\br \colon \gal{\Q} \to G(k)$ of a less Lie-theoretic flavor that Theorem \ref{mainthm} will lift to Zariski-dense geometric representations $\br \colon \gal{\Q} \to G(\mc{O})$. Recall that over $\CC$ we have an embedding $F_4(\CC) \into E^{\mr{sc}}_6(\CC)$ given by identifying $F_4$ to the stabilizer of a vector in one of the 27-dimensional minuscule representations $V_{\mr{min}}$ of $E_6^{\mr{sc}}$. Letting $H$ and $G$ be the split groups (over $\Z$) of type $F_4$ and $E_6^{\mr{sc}}$, we can realize this embedding $H \into G$ over $R= \mc{O}_E[\frac{1}{N}]$ for some number field $E$ and integer $N$ (a quantitative refinement of this soft ``spreading-out" assertion is of course possible). By \cite[1.1 Main Theorem]{cohen-wales}, the finite groups $A_6$ and $\mr{PSL}_2(\mathbb{F}_{13})$ embed into $F_4(\CC)$, and, perhaps after replacing $E$ (and hence $R$) by a finite extension, we may assume these groups are embedded into $H(R)$. This theorem also tells us the characters of $A_6$ and $\mr{PSL}_2(\mathbb{F}_{13})$ in $V_{\mr{min}}$ and the adjoint representation of $E_6$. Recalling the decompositions as $F_4$-representations
\begin{align*}
\Lie(E_6)&= \Lie(F_4) \oplus U, \\
V_{\mr{min}}&= \mathbbm{1} \oplus U,
\end{align*}
where $U$ is the irreducible 26-dimensional representation of $F_4$, we compute the following decompositions of $\Lie(F_4)$ as $A_6$ and $\mr{PSL}_2(\mathbb{F}_{13})$-representations:
\[
\Lie(F_4) \cong \begin{cases}
\chi_4 \oplus 3 \cdot \chi_5 \oplus 2\cdot \chi_7 \quad \text{(case $A_6$)}; \\
\chi_4 \oplus \{\text{$\chi_5$ or $\chi_6$}\} \oplus 2 \cdot \chi_9 \quad \text{(case $\mr{PSL}_2(\mathbb{F}_{13})$)},
\end{cases}
\]
where we use the ATLAS notation (\cite{ATLAS}) for characters. It turns out that for our purposes knowing whether $\chi_5$ or $\chi_6$ appears in the decomposition in the $\mr{PSL}_2(\mathbb{F}_{13})$ case is irrelevant. In particular, letting $c$ denote the unique conjugacy class of order 2 in either case, the ATLAS character tables tell us that the trace of $c$ acting on $\Lie(F_4)$ is $-4= -\rk(F_4)$, and so
\[
\dim \Lie(F_4)^{\Ad(c)=1}= \frac{\dim(F_4)- \rk(F_4)}{2}= \dim \mr{Flag}_{F_4},
\]
i.e. $\Ad(c)$ is an odd involution of $\Lie(F_4)$.
\begin{prop}\label{f4eg}
For all sufficiently large primes, there are representations $\br_1 \colon \gal{\Q} \to F_4(\overline{\mathbb{F}}_p)$ and $\br_2 \colon \gal{\Q} \to F_4(\overline{\mathbb{F}}_p)$ that have images $\im(\br_1) \cong A_6$, $\im(\br_2) \cong \mr{PSL}_2(\mathbb{F}_{13})$, and that admit geometric deformations $\rho_1, \rho_2 \colon \gal{\Q} \to F_4(\ov{\Z}_p)$ with Zariski-dense image. 
\end{prop}
\begin{proof}
There are Galois extensions $L_1/\Q$ and $L_2/\Q$ satisfying $\Gal(L_1/\Q) \cong A_6$, $\Gal(L_2/\Q) \cong \mr{PSL}_2(\mathbb{F}_{13})$, and complex conjugation $c$ is non-trivial in each $\Gal(L_i/\Q)$: the constructions of $L_1$ and $L_2$ are due to Hilbert and Shih, respectively, and both are explained in \cite[\S 4.5, Theorem 5.1.1]{serre:topics}. It is easy to see that we can take $c$ to be non-trivial, and note that to apply Shih's theorem we use that $\left(\frac{2}{13}\right)=-1$. Let $p$ be any sufficiently large (in the sense of Theorem \ref{mainthm} for $F_4$, not dividing $N$, and not dividing $|\im(\br_i)|$) prime that is unramified in $L_i/\Q$. Reducing the inclusions $\Gal(L_i/\Q) \into H(R)$ modulo a prime of $R$ above $p$, we obtain residual representations $\br_i \colon \gal{\Q} \to H(\overline{\mathbb{F}}_p)$ satisfying the hypotheses of Theorem \ref{mainthm}. Indeed, by the character calculation preceding Proposition \ref{f4eg}, both $\br_i$ are odd. At primes $v \nmid p$ there are obvious Teichm\"{u}ller lifts of $\br_i|_{\gal{F_v}}$. At the prime $p$, $\br_i(\gal{\Q_p})$ is valued in some torus of $F_4$, since $p$ is unramified in $L_i$ and $\im(\br_i)$ is coprime to $p$. Finally, $L_i$ is linearly disjoint from $\Q(\mu_p)$ over $\Q$, and $\br_i$ satisfies our global image requirements: the hypotheses of Assumption \ref{multfree} are clearly satisfied since $\im(\br_i)$ has order coprime to $p$, and $\im(\br_i)$ has no non-trivial cyclic quotient, substituting for the application of Lemma \ref{cyclicq} in \S \ref{klr^Nsection} and \S \ref{relativeliftsection}.
\end{proof}
\begin{rmk}
We note that the multiplicities of $A_6$ and $\mr{PSL}_2(\mathbb{F}_{13})$ acting on $\Lie(F_4)$ are for certain irreducible constituents greater than 1, so these examples use the full generality of our methods: compare Remark \ref{multfreermk} and the discussion in the introduction. We also note that every non-trivial irreducible representation of $F_4$ has some multiplicity greater than 1 in its formal character, so we cannot apply potential automorphy theorems as in \cite{bcelmpp} to lift our $\br_i$: indeed, there are no $F_4$-valued Galois representations that are Hodge-Tate regular after embedding $F_4$ in some $\mr{GL}_N$. Moreover, the actions of the subgroups $A_6$ and $\mr{PSL}_2(\mathbb{F}_{13})$ on $U$ (the irreducible 26-dimensional representation of $F_4$) are reducible, further precluding the use of potential automorphy theorems.
\end{rmk}
We cannot resist two more examples, which likewise cannot be treated by other methods (again because any non-trivial irreducible representation of $E_8$ has some multiplicity greater than 1 in its formal character):
\begin{prop}
For all sufficiently large primes $p$, there are representations $\br_1 \colon \gal{\Q} \to E_8(\overline{\mathbb{F}}_p)$ and $\br_2 \colon \gal{\Q} \to E_8(\overline{\mathbb{F}}_p)$ having images $\im (\br_1) \cong \mr{PSL}_2(\mathbb{F}_{41})$ and $\im(\br_2) \cong \mr{PSL}_2(\mathbb{F}_{49})$, and that admit geometric deformations $\rho_1, \rho_2 \colon \gal{\Q} \to E_8(\overline{\Z}_p)$ with Zariski-dense images.
\end{prop}
\begin{proof}
The same method of proof applies: $\mr{PSL}_2(\mathbb{F}_{41})$ and $\mr{PSL}_2(\mathbb{F}_{49})$ embed into the complex Lie group $E_8(\CC)$ (\cite{griess-ryba}), and the necessary oddness follows from \cite[Table 1, Table 2]{griess-ryba}. Shih's theorem (\cite[Theorem 5.1.1]{serre:topics}) still applies to construct $\mr{PSL}_2(\mathbb{F}_{41})$ as a Galois group over $\Q$ (note $\left(\frac{3}{41}\right)=-1$), and the paper \cite{dieulefait-vila} establishes the realization of $\mr{PSL}_2(\mathbb{F}_{49})$ as a Galois group over $\Q$ (via modular forms, so that complex conjugation is non-trivial). 
\end{proof}

\appendix
\section{Some group theory: irreducible $G(k)$-representations for $p
  \gg_G 0$}\label{groupsection}

We prove a few group-theoretic lemmas showing that the image
hypotheses of \S \ref{klrsection} (namely, Assumption \ref{multfree})
in fact follow from the seemingly simpler assumption that $\br$ is
``absolutely irreducible,'' as long as $p$ is sufficiently large. We
note that the explicit bounds extracted here depend on the
classification of finite simple groups. Recall that a subgroup $\Gamma
\subset G(k)$, with $G$ a connected reductive group over $k$, is absolutely irreducible if $\Gamma$ is not contained in any proper parabolic subgroup of $G_{\overline{k}}$.
\begin{lemma}\label{irr}
 Let $\Gamma \subset G(k)$ be an absolutely irreducible finite subgroup. Assume $p > 2(\dim_k(\fgder) +1)$. Then:
\begin{enumerate}
\item $\fgder$ is a semisimple $k[\Gamma]$-module.
\item $\Gamma$ does not contain any non-trivial normal subgroup of
  $p$-power order.
 \item $H^1(\Gamma, \fgder)=0$, and the same holds if the action of $\Gamma$ on $\fgder$ is twisted by a character of $\Gamma$.
\end{enumerate}
\end{lemma}
\begin{proof}
 Let $h_G$ be the maximum of the Coxeter numbers of the simple factors
 of $G$. By \cite[Corollaire 5.5]{serre:CR}, for $p> 2h_G-2$,
 $\mf{g}$, and hence its summand $\fgder$, is a semisimple
 $\Gamma$-module.

 Now suppose $H \unlhd \Gamma$ is a non-trivial normal subgroup of $G$
 of $p$-power order. Consider any irreducible
 $\overline{k}[\Gamma]$-summand $U$ of $\fg_{\overline{k}}$. The
 $\overline{k}$-vector space of invariants $U^H$ is non-trivial (since
 $H$ is a $p$-group) and is stabilized by $\Gamma$, hence must equal
 all of $U$. This holds for all $U$, so $\fg$ is a trivial $H$-module,
 and therefore $H$ is contained in the center $Z_{G}(k)$; but the
 latter clearly has order prime to $p$, so this contradiction proves (2).

 Finally, (2) allows us to apply \cite[Theorem A]{guralnick:CR} to
 deduce that $H^1(\Gamma, \fgder)=0$ for $p>2 (\dim_k(\fgder) +1)$ (to
 be precise, apply this result to $\Gamma/\Gamma \cap Z_{G}(k)$ acting
 on $\fgder$).
\end{proof}
The following lemma, with a different proof, also appears in \cite[Lemma 5.1]{bhkt:fnfieldpotaut}:
\begin{lemma}\label{invariants}
Let $G$ be a connected reductive group over $\bar{k}$. Assume $p >5$, and that $p \nmid n+1$ for any simple factor of $G^{\mr{ad}}$ of Dynkin type $A_n$. Let $\Gamma \subset G(\bar{k})$ be absolutely irreducible. Then $H^0(\Gamma, \fgder)=0$.
\end{lemma}
\begin{proof}
By our characteristic assumptions (which imply that $G^{\mr{der}}$ and
$G^{\mr{ad}}$ have isomorphic Lie algebras), we may and do assume
$G=G$ is an adjoint group, and by considering each simple factor of
$G$ we may and do further assume that $G$ is simple. Let $X$ be an
element of $\fg^\Gamma$. We have the Jordan decomposition $X=X_s+X_n$
into semisimple and nilpotent parts in $\fg$, and uniqueness of
Jordan decomposition implies that both $X_s$ and $X_n$ are
$\Gamma$-invariant. Since $\Gamma$ is then contained in the
intersection $C_G(X_s) \cap C_G(X_n)$, it suffices to show that
$C_G(X)$ is contained in a proper parabolic when $X$ is either
semisimple or nilpotent. In either case, as long as $p>5$ (for $G$
not of type $A_n$) or $p \nmid n+1$ (for $G$ of type $A_n$), $C_G(X)$
is smooth (by a theorem of Richardson: see \cite[2.5
Theorem]{jantzen:nilporbits}). Assume $X$ is a non-zero
nilpotent. Then \cite[5.9 Proposition]{jantzen:nilporbits} implies
that $C_G(X)$ is contained in a proper parabolic subgroup. Now assume
$X$ is a non-zero semisimple element. There is a maximal torus $T$ of
$G$ such that $X$ belongs to $\mf{t}= \Lie(T)$
(\cite[11.8]{borel:linalg}). As usual, we can diagonalize the
$T$-action on $\fg$ to obtain a root system (in the real vector space
$X^\bullet(T) \otimes_{\Z} \RR$). The subgroup $C_G(X)$ is a connected
reductive group containing $T$: for the connectedness, we use that
$p>5$ (ensuring $p$ is not a ``torsion prime'') so that we can invoke \cite[Theorem 3.14]{steinberg:torsion}. By \cite[3.4 Proposition]{borel-tits:reductive}, $C_G(X)$ is determined by the root subgroups it contains (since it contains a maximal torus of $G$). For a root $\alpha \in \Phi(G, T)$, let $u_{\alpha} \colon \mathbf{G}_a \to G$ be the corresponding root subgroup. For $t \in T$, the relation 
\[
u_{\alpha}(y)t u_{\alpha}(y)^{-1}= t\cdot u_{\alpha}((\alpha(t)^{-1}-1)y)
\] 
lets us compute that (passing to the Lie algebra) $C_G(X)$ precisely contains those $U_{\alpha}= \im(u_{\alpha})$ drawn from the subset
\[
\Phi'= \{\alpha \in \Phi(G, T): d\alpha(X)=0\}
\]
of $\Phi= \Phi(G, T)$. We claim that the semisimple rank of $C_G(X)$ is strictly less than that of $G$. Temporarily granting this, we have that the roots $\Phi'$ span a proper subspace $\RR \Phi' \subset \RR \Phi=X^\bullet(T)_{\RR}$. By \cite[VI.1.7 Proposition 23]{bourbaki:lie456}, $\Phi'$ is a root system in the real vector space $\RR \Phi'$, and we can also consider it as a subsystem of the root system $\Phi''= \RR \Phi' \cap \Phi$. The latter, by \cite[VI.1.7 Proposition 24]{bourbaki:lie456} has a basis $I$ that extends to a basis of $\Phi$; and since $\RR \Phi'$ is strictly contained in $\RR \Phi$ this basis of $\Phi''$ is a proper subset of the extended basis of $\Phi$. It follows that $C_G(X)$ is contained in the (proper) Levi subgroup of $G$ associated to $I$, and therefore that $\Gamma$ is reducible.

To complete the proof, we establish the postponed claim that the inclusion $\RR \Phi' \subseteq \RR \Phi$ is proper. It suffices to show that $C_G(X)$ is not semisimple, i.e. has positive-dimensional center. Suppose it were semisimple. Its root system is a (not necessarily simple) subsystem of that of $G$, and so there are only finitely many possibilities for the root systems of the simple factors $H$ of $C_G(X)^{\mr{ad}}$. Under our assumptions on $p$, each of these simple factors satisfies the following two properties:
\begin{itemize}
\item $H^{\mr{sc}} \to H^{\mr{ad}}$ induces an isomorphism on Lie algebras.
\item $\Lie(H)$ has trivial center.
\end{itemize}
Indeed, note that $\Lie(H)$ has non-trivial center only when $p \leq 3$ or $H$ is of type $A_n$ and $p \mid n+1$: see the discussion of \cite[pp. 47-48]{seligman:modular} (which ensures that $\Lie(H)$ has a nonsingular trace form), and then apply \cite[Theorem I.7.2]{seligman:modular}. Thus under our assumptions on $p$, $\Lie(C_G(X))= C_{\fg}(X)$ must have trivial center. But $X$ visibly lies in the center, and we have therefore contradicted the supposed semisimplicity of $C_G(X)$.

\end{proof}
In the main theorem, we will use the next three lemmas (Lemma \ref{cyclicq}, specifically) to show that $\br(\fgder)$ and $\br(\fgder)^*$ have no common subquotient.

\begin{lemma} \label{larsenpink} Given integers $n, c_1 >0$, there
  exists an integer $c_2 > 0$ (depending only on $n$ and $c_1$) such
  that if $\Gamma \subset \mr{GL}_n(k)$ is a finite subgroup admitting
  a cyclic quotient of order $c_2$ and not containing any normal
  subgroup of order $p^a$ with $a>0$, then the centre of $\Gamma$
  contains a cyclic subgroup of order prime to $p$ and $\geq c_1$.
\end{lemma}

\begin{proof}
  By Theorem 0.2 of \cite{larsen-pink:finite}, any finite subgroup
  $\Gamma \subset \mr{GL}_n(k)$ has normal subgroups
  $\Gamma_3 \subset \Gamma_2 \subset \Gamma_1 \subset \Gamma$ such that
  $\Gamma_3$ is a $p$-group, $\Gamma_2/\Gamma_3$ is an abelian group
  of order prime to $p$, $\Gamma_1/\Gamma_2$ is a product of finite
  simple groups of Lie type and $\Gamma/\Gamma_1$ has order bounded
  by a constant depending only on $n$. Our assumptions imply that
  $\Gamma_3$ is trivial. From the proof of the theorem
  \cite[p. 1156]{larsen-pink:finite} this imples that $\Gamma_2$ is in the
  centre of $\Gamma_1$, so the conjugation action of $\Gamma$ on
  $\Gamma_2$ factors through $\Gamma/\Gamma_1$.

  Let $\Gamma' = (\Gamma_1)^{\mr{der}}$. Clearly $\Gamma'$ lies in the
  kernel of any homomorphism from $\Gamma$ to an abelian group and
  $\Gamma_2$ surjects onto $\Gamma_1/\Gamma'$. Furthermore,
  $\Gamma' \cap \Gamma_2$ has order bounded by a constant depending
  only on $n$: this again follows from the construction of $\Gamma_1$ and $\Gamma_2$
  in \cite[p. 1156]{larsen-pink:finite} (note particularly the construction of the group denoted $G_2$ in \textit{loc.~cit.}). 
  Since the order of $\Gamma/\Gamma_1$ is bounded, if
  $\Gamma$ has a large cyclic quotient, the coinvariants of the action
  of $\Gamma/\Gamma_1$ on $\Gamma_1/\Gamma'$ must also have a large
  cyclic quotient, and so also a large cyclic subgroup. The lemma
  follows since if $A$ is any abelian group with an action of a finite
  group $\Delta$, the kernel of the averaging map from $A_{\Delta}$ to $A^{\Delta}$ is
  killed by the order of $\Delta$.
\end{proof}
\begin{rmk}\label{explicit}
The constant $c_2$ can be effectively bounded by invoking an explicit bound on the index $[\Gamma:\Gamma_1]$ obtained by Collins (\cite{collins:modularjordan}) using (unlike \cite{larsen-pink:finite}) the classification of finite simple groups.
\end{rmk}

\begin{lemma} \label{cent} For $G$ any (split) connected reductive
  group over $k$ there exists a constant $n_G$, depending only on the
  root datum of $G$, such that for any semisimple element
  $s \in G(k)$ the centralizer of $s^n$ in $G$ is a (not necessarily
  proper) Levi subgroup of $G$ for some $n$ dividing $n_G$.
\end{lemma}

\begin{proof}
  Let $T$ be a maximal torus of $G$ containing $s$ and let $t$ be any
  element of $T(\bar{k})$. By the theorem in \S 2.2 of
  \cite{humphreys:conjugacy}, $C_G(t)$ is generated by $T$, the root
  subgroups $U_{\alpha}$ for which $\alpha(t) = 1$ and representatives
  (in $N(T)$) of the subgroup $W(t)$ of the Weyl group $W(G,T)$ fixing
  $t$.  Let $\Phi(t)$ be the subset of $\Phi(G,T)$ consisting of all
  roots which are trivial on $t$. Let $T^{W(t)}$ be the subgroup of
  $T$ fixed pointwise by $W(t)$ and let
  $T^{\Phi(t)} = \cap_{\alpha \in \Phi(t)} \; \mr{Ker}(\alpha)$.  Let
  $n_G'$ be the lcm of the orders of the torsion subgroups of all the
  character groups of the groups of multiplicative type
  $T^{W(t)} \cap T^{\Phi(t)}$ for all $t \in T(\bar{k})$; there are
  only finitely many distinct such subgroups since both $W(G,T)$ and
  $\Phi(G,T)$ are finite sets. Then $n_G'$ depends only on the root
  datum of $G$, and the order of the component group of any subgroup
  $T^{W(t)} \cap T^{\Phi(t)}$ divides $n_G'$.

  It follows that $s_1 := s^{n_G'}$ is contained in a torus $T_1$ such
  that $T_1 \subset T^{W(s)} \cap T^{\Phi(s)}$. We clearly have
  $W(s) \subset W(s_1)$ and $\Phi(s) \subset \Phi(s_1)$. If both inclusions
  are equalities then $C_G(s)$ equals $C_G(s_1)$. Since
  $C_G(s_1) \supset C_G(T_1) \supset C_G(s)$ by construction, it would follow that
  $C_G(s)$ is equal to the centralizer of a torus, hence (by \cite[4.15 Th\'{e}or\`{e}me]{borel-tits:reductive}) a Levi
  subgroup. If either of the inclusions is strict, we repeat the
  procedure after replacing $s$ by $s_1$. Since $W(G,T)$ and
  $\Phi(G,T)$ are both finite, after at most
  $m_G := |W(G,T)| + |\Phi(G,T)|$ steps we must have equality. Thus,
  we may take $n_G$ to be $(n'_G)^{m_G}$.
  \end{proof}

  \begin{lemma} \label{cyclicq}
    For $G$ any split semisimple group over $k$ there exists a
    constant $a_G$ depending only on the root datum of $G$ such that
    if $\Gamma \subset G(k)$ is an absolutely irreducible subgroup
    then $\Gamma$ has no cyclic quotient of order $\geq a_G$.
  \end{lemma}

  \begin{proof}

    By embedding $G$ in $\mr{GL}_n$ for some $n$, and using Lemma
    \ref{irr} (2), we may apply Lemma \ref{larsenpink} with $c_1-1 $
    equal to the number $n_G$ obtained from Lemma \ref{cent}, to get
    $c_2$ such that if $\Gamma$ has a cyclic quotient of order
    $\geq c_2$ then the centre of $\Gamma$ contains a cyclic subgroup
    $Z$ of order at least $c_1$ and of order prime to $p$. By Lemma
    \ref{cent} there exists an integer $n < c_1$ so that $C_G(s^n)$ is
    a Levi subgroup, where $s$ is any generator of $Z$. By
    construction, $s^n$ is not the identity and since $G$ is adjoint,
    it is also not central, so $C_G(s^n)$ is a proper Levi subgroup of
    $G$. But $\Gamma \subset C_G(s^n)$ and this contradicts
    irreducibility once again.
  \end{proof}
Putting together the results of this section, we deduce:
\begin{cor}\label{irrcor}
Let $p \gg_G 0$ be a prime, and let $\br \colon \gal{F, {\mc{S}}} \to G(k)$ be a representation such that $\br|_{\gal{\tF(\zeta_p)}}$ is absolutely irreducible. Assume $[\tF(\zeta_p):\tF]$ is greater than the constant $a_G$ of Lemma \ref{cyclicq}. Then all of the conditions in Assumption \ref{multfree} hold for $\br$.
\end{cor}
\begin{proof}
  Since $[K:\tF(\br(\fgder), \mu_p)]$ and
    $[\tF(\mu_p):F]$ are both prime to $p$, under our assumptions on
    $p$ and absolute irreducibility of $\br|_{\gal{\tF(\zeta_p)}}$,
    Lemmas \ref{irr}, \ref{invariants}, and \ref{cyclicq} imply that the first two
    conditions of Assumption \ref{multfree}, as well as the semisimplicity and the condition
    on no trivial subrepresentations, hold (to show $\mu_{p^2}$ is not contained in $K$, note that $p>[\tF(\zeta_p):\tF]>a_G$, so that $\Ad\circ \br(\gal{\tF})$ has no order $p$ quotient). Moreover, Lemma
    \ref{cyclicq} implies that $\br(\fgder)$ and $\br(\fgder)^*$ have
    no common $\Fp[\gal{\tF}]$-subquotient. To see this, first note
    that the lemma implies that the fixed field $\tF(\br(\fgder))$
    cannot contain $\tF(\zeta_p)$ (for then the adjoint image of
    $\br(\gal{\tF})$ would have a large cyclic quotient, as we take
    $[\tF(\zeta_p):\tF] >a_G$). Letting $\{W_i\}_{i \in I}$ be the
    simple $\Fp[\gal{\tF}]$-module constituents of $\br(\fgder)$, if
    $\br(\fgder)$ and $\br(\fgder)^* \cong \br(\fgder)(1)$ had a
    common constituent, there would be an isomorphism
    $W_i \cong W_j(1)$ for some $i, j \in I$. We can choose
    $\sigma \in \gal{\tF}$ acting trivially on $W_i$ and $W_j$ but
    non-trivially on $\tF(\zeta_p)$, contradicting the equivalence
    $W_i \cong W_j(1)$. Thus, all the conditions of Assumption
    \ref{multfree} hold.
\end{proof}

\section{Application of some results of Lazard}\label{lazardsection}

The following lemma and its corollary, the latter being crucial for us,  are deduced from results  of \cite{lazard}.

\begin{lemma} \label{lem:laz} Let $G$ be a compact $p$-adic Lie group
  such that its Lie algebra $\fg$ is semisimple. Let $\mc{O}$ be
  the ring of integers in a finite extension of $\Q_p$ and let $M$ be
  a finitely generated free $\mc{O}$-module on which $G$ acts
  continuously and $\mc{O}$-linearly. If $M\otimes_{\Z_p}\Q_p$ does
  not contain the trivial representation of $\fg$,\footnote{The Lie
    algebra $\fg$ acts on  $M\otimes_{\Z_p}\Q_p$ by
    \cite{lazard}, V, Lemma (2.4.4).} then there exists
  an integer $n \geq 0$ such that for $\varpi$ a uniformizer of $\mc{O}$,
  $H^i(G,M/\varpi^mM)$ is killed by $\varpi^n$ for all $i > 0$ and
  $m \geq 0$. If $M$ contains the trivial representation then the same
  holds for $i=1$.
\end{lemma}

\begin{proof}
  We first assume that $G$ is a pro $p$-group which is $p$-valued
  (\cite{lazard}, III \S 2). By
  (\cite{lazard},V, (2.2.3.1)), there exists a ring $A$
  (the completed group algebra of $G$ over $\Z_p$) such that
\[
H^i(G,M) = \mr{Ext}^i_A(\Z_p,M) .
\]
Moreover, by (\cite{lazard}, V, (2.2.2.3) ), $\Z_p$ has a
finite resolution by free $A$-modules of finite rank. It follows that
$H^*(G,M)$ is the cohomology of a bounded complex of $\mc{O}$-modules, each
term of which is a finite direct sum of copies of $M$. Furthermore,
$H^*(G,M/\varpi^mM)$ is computed by tensoring this complex with
$\mc{O}/\varpi^m\mc{O}$ and taking cohomology.

By (\cite{lazard}, V, Theorem (2.4.10)), the
semisimplicity of $\fg$, and Theorems 21.1 and 24.1 of
\cite{chevalley-eilenberg}, it follows that
$H^i(G,M)\otimes_{\Z_p} \Q_p = 0$ for all $i>0$ if $M$ does not
contain the trivial representation and for $i=1,2$ otherwise. Since
$M$ is assumed to be finitely generated, the previous paragraph shows
that $H^i(G,M)$ is also finitely generated, so $H^i(G,M)$ is killed by
$\varpi^n$ for some $n \geq 0$, for all $i>0$ if $M$ does not contain the
trivial representation and for $i=1,2$ otherwise. The statement for
$M/\varpi^mM$ follows from this by applying the universal coefficient
theorem to the complex computing the cohomology:
\[
  H^i(G,M/\varpi^mM) \cong (H^i(G,M) \otimes_{\mc{O}} \mc{O}/\varpi^n\mc{O})
  \oplus \mr{Tor}_1^{\mc{O}}(H^{i+1}(G,M), \mc{O}/ \varpi^n\mc{O}) .
\]

For a general compact analytic $G$, by (3.1.3) and (3.1.7.4) of
(\cite{lazard}, III) there exists a normal subgroup of
finite index which is a $p$-valued pro $p$-group. The result then
follows from the above special case by using the Hochschild--Serre
spectral sequence.
\end{proof}

We owe the deduction of the following corollary to D. Prasad.

\begin{cor}\label{cor:lazard}
  Keep the assumptions of Lemma \ref{lem:laz} and also assume that
  $H^0(G,M/\varpi M) = 0$. Then for any $m\geq 0$ there exists an
  integer $N(m)$ so that the map
\[
  H^1(G, M/\varpi^{N}M) \to H^1(G,M/\varpi^m M)
\]
is the zero map if $N \geq N(m)$.
\end{cor}

\begin{proof}
  The maps $M/\varpi^N M \to M/\varpi^N M$ given by multiplication by
  $\varpi^n$ induce the zero map on $H^1$ by Lemma \ref{lem:laz}. Also,
  since $H^0(G, M/\varpi M ) = 0$, $H^0(G, M/\varpi^rM) = 0$ for any
  $r \geq 0$, which implies that all the maps
  $H^1(G,M/\varpi^sM) \to H^1(G,M/\varpi^NM)$ (induced by inclusions of
  $M/\varpi^sM$ in $M/\varpi^NM$) are injective. The lemma follows
  immediately from these two statements.

\end{proof}

\bibliographystyle{amsalpha}
\bibliography{biblio.bib}

\printindex[terms]

\end{document}